\documentclass[10pt]{amsart}

\usepackage{amsfonts,amssymb}

\input{xy}
\xyoption{all}

\newtheorem{proposition}{Proposition}[section]
\newtheorem{lemma}[proposition]{Lemma}
\newtheorem{corollary}[proposition]{Corollary}
\newtheorem{theorem}[proposition]{Theorem}
\theoremstyle{definition}
\newtheorem{definition}[proposition]{Definition}

\theoremstyle{remark}
\newtheorem{remark}[proposition]{Remark}

\newcommand{\thlabel}[1]{\label{th:#1}}
\newcommand{\thref}[1]{Theorem~\ref{th:#1}}
\newcommand{\selabel}[1]{\label{se:#1}}
\newcommand{\seref}[1]{Section~\ref{se:#1}}
\newcommand{\lelabel}[1]{\label{le:#1}}
\newcommand{\leref}[1]{Lemma~\ref{le:#1}}
\newcommand{\prlabel}[1]{\label{pr:#1}}
\newcommand{\prref}[1]{Proposition~\ref{pr:#1}}
\newcommand{\colabel}[1]{\label{co:#1}}

\newcommand{\relabel}[1]{\label{re:#1}}
\newcommand{\reref}[1]{Remark~\ref{re:#1}}

\newcommand{\delabel}[1]{\label{de:#1}}
\newcommand{\deref}[1]{Definition~\ref{de:#1}}

\newcommand{\eqlabel}[1]{\label{eq:#1}}
\newcommand{\equref}[1]{(\ref{eq:#1})}
\def\equuref#1#2{(\ref{eq:#1}.#2)}

\def\ra{\rightarrow}

\def\Id{{\rm Id}}

\def\circo{~\circ~}

\def\ACA{{{}_A^{~ !}\Cc_A^{\vbox{}}}}

\def\mfC{{\mf C}}

\def\ot{\otimes}
\def\ota{\hspace*{-1pt}\bullet\hspace*{-1pt}}
\def\va{\varepsilon}

\def\un{\underline}

\def\mf{\mathfrak}

\def\va{\varepsilon}

\def\ra{\rightarrow}

\def\d{\delta}

\def\ov{\overline}
\def\cal{\mathcal}
\def\un{\underline}
\def\Colim{{\rm Colim}}
\def\ACA{{}_A^{~ !}\Cc_A^{\vbox{}}}

\newcommand{\Cc}{\cal C}
\newcommand{\Dc}{\cal D}
\newcommand{\Ec}{\cal E}

\newcommand{\Tc}{\cal T}

\newcommand{\Nat}{{\rm Nat}}

\newcommand{\gbrcb}{\mbox{$\gbrc\gvac{-1}\gnot{\hspace*{-4mm}\bullet}$}}

\def\equal#1{\smash{\mathop{=}\limits^{#1}}}

\def\equalupdown#1#2{\smash{\mathop{=}\limits^{#1}\limits_{#2}}}
\allowdisplaybreaks[4]

 \newcount\grcalca
 \newcount\grcalcb
 \newcount\grcalcc
 \newcount\grcalcd
 \newcount\grcalce
 \newcount\grcalcf
 \newcount\grcalcg
 \newcount\grcalch
 \newcount\grcalci
 \newcount\grrow
 \newcount\grcolumn
 \newcount\grwidth
 \newcount\factor
 \newcount\hfactor
 \newcount\qfactor
 \newcount\tfactor
 \newcount\sfactor
 \newcount\dfactor
 \hfactor = \factor
 \divide    \hfactor by 2
 \qfactor = \hfactor
 \divide    \qfactor by 2
 \tfactor = \factor
 \divide    \tfactor by 3
 \sfactor = \factor
 \divide    \sfactor by 6
 \dfactor = \factor
 \divide    \dfactor by 12

 \newcommand{\gbeg}[2]{
   \unitlength=1pt
   \grrow = #2
   \grcolumn = 0
   \grcalca = #1
   \grcalcb = #2
   \multiply \grcalca by \factor
   \grwidth = \grcalca
   \multiply \grcalcb by \factor
   \begin{minipage}{\grcalca pt}
   \begin{picture}(\grcalca,\grcalcb)
   \advance \grcalcb by -\factor
   \put(0, \grcalcb){\line(1,0){\grwidth}} }

 \newcommand{\gend}{
   \put(0, \factor){\line(1,0){\grwidth}}
   \end{picture}
   {\vskip2.5ex}
   \end{minipage} }

 \newcommand{\gnl}{
   \advance \grrow by -1
   \grcolumn = 0}

 \newcommand{\gvac}[1]{       
   \advance \grcolumn by #1} 
 
 \newcommand{\gcl}[1]{
   \grcalca = \grcolumn
   \multiply \grcalca by \factor
   \advance \grcalca by \hfactor
   \grcalcb = \grrow
   \multiply \grcalcb by \factor
   \grcalcc = #1
   \multiply \grcalcc by \factor
   \put(\grcalca,\grcalcb) {\line(0,-1){\grcalcc}} 
   \advance \grcolumn by 1}

 \newcommand{\gcn}[4]{
   \grcalca = \grcolumn
   \multiply \grcalca by \factor
   \grcalci = #3
   \multiply \grcalci by \hfactor
   \advance \grcalca by \grcalci
   \grcalcb = \grcolumn
   \multiply \grcalcb by \factor 
   \grcalci = #3
   \advance \grcalci by #4
   \multiply \grcalci by \qfactor
   \advance \grcalcb by \grcalci
   \grcalcc = \grcolumn
   \multiply \grcalcc by \factor
   \grcalci = #4
   \multiply \grcalci by \hfactor
   \advance \grcalcc by \grcalci
   \grcalcd = \grrow
   \multiply \grcalcd by \factor 
   \grcalce = \grrow
   \multiply \grcalce by \factor 
   \grcalci = #2
   \multiply \grcalci by \tfactor
   \advance \grcalce by -\grcalci
   \grcalcf = \grrow
   \multiply \grcalcf by \factor 
   \grcalci = #2
   \multiply \grcalci by \hfactor
   \advance \grcalcf by -\grcalci
   \grcalcg = \grrow
   \multiply \grcalcg by \factor 
   \grcalci = #2
   \multiply \grcalci by \tfactor
   \multiply \grcalci by 2
   \advance \grcalcg by -\grcalci
   \grcalch = \grrow
   \advance \grcalch by -#2
   \multiply \grcalch by \factor 
   \qbezier(\grcalca,\grcalcd)(\grcalca,\grcalce)(\grcalcb,\grcalcf) 
   \qbezier(\grcalcb,\grcalcf)(\grcalcc,\grcalcg)(\grcalcc,\grcalch) 
   \advance \grcolumn by #1}

 \newcommand{\gnot}[1]{
   \grcalca = \grcolumn
   \multiply \grcalca by \factor
   \advance \grcalca by \hfactor
   \grcalcb = \grrow
   \multiply \grcalcb by \factor
   \advance \grcalcb by -\hfactor
   \put(\grcalca,\grcalcb) {\makebox(0,0){$\scriptstyle #1$}} }

 \newcommand{\got}[2]{
   \grcalca = \grcolumn
   \multiply \grcalca by \factor
   \grcalcc = #1
   \multiply \grcalcc by \hfactor
   \advance \grcalca by \grcalcc
   \grcalcb = \grrow
   \multiply \grcalcb by \factor
   \advance \grcalcb by -\tfactor
   \advance \grcalcb by -\tfactor
   \put(\grcalca,\grcalcb){\makebox(0,0)[b]{$#2$}}
   \advance \grcolumn by #1}

 \newcommand{\gob}[2]{
   \grcalca = \grcolumn
   \multiply \grcalca by \factor
   \grcalcc = #1
   \multiply \grcalcc by \hfactor
   \advance \grcalca by \grcalcc
   \put(\grcalca,0){\makebox(0,0)[b]{$#2$}}
   \advance \grcolumn by #1}

 \newcommand{\gmu}{  
   \grcalca = \grcolumn
   \advance \grcalca by 1
   \multiply \grcalca by \factor
   \grcalcb = \grrow
   \multiply \grcalcb by \factor
   \grcalcc = \factor
   \advance \grcalcc by \hfactor
   \put(\grcalca,\grcalcb){\oval(\factor,\grcalcc)[b]}
   \advance \grcalcb by -\hfactor
   \advance \grcalcb by -\qfactor
   \put(\grcalca,\grcalcb) {\line(0,-1){\qfactor}} 
   \advance \grcolumn by 2}

 \newcommand{\gcmu}{   
   \grcalca = \grcolumn
   \advance \grcalca by 1
   \multiply \grcalca by \factor
   \grcalcb = \grrow
   \advance \grcalcb by -1
   \multiply \grcalcb by \factor
   \grcalcc = \factor
   \advance \grcalcc by \hfactor
   \put(\grcalca,\grcalcb){\oval(\factor,\grcalcc)[t]}
   \advance \grcalcb by \factor
   \put(\grcalca,\grcalcb) {\line(0,-1){\qfactor}} 
   \advance \grcolumn by 2}

 \newcommand{\glm}{
   \grcalca = \grcolumn
   \multiply \grcalca by \factor
   \advance \grcalca by \hfactor
   \grcalcb = \grcalca
   \advance \grcalcb by \factor
   \grcalcc = \grrow
   \multiply \grcalcc by \factor
   \grcalcd = \grcalcc
   \advance \grcalcd by -\tfactor
   \grcalce = \grcalcd
   \advance \grcalce by -\tfactor
   \put(\grcalca, \grcalcc){\line(0,-1){\tfactor}}
   \put(\grcalca, \grcalcd){\line(1,0){\factor}}
   \put(\grcalca, \grcalcd){\line(3,-1){\factor}}
   \put(\grcalcb, \grcalcc){\line(0,-1){\factor}}
   \advance \grcolumn by 2}

 \newcommand{\grm}{
   \grcalcb = \grcolumn
   \multiply \grcalcb by \factor
   \advance \grcalcb by \hfactor
   \grcalca = \grcalcb
   \advance \grcalca by \factor
   \grcalcc = \grrow
   \multiply \grcalcc by \factor
   \grcalcd = \grcalcc
   \advance \grcalcd by -\tfactor
   \grcalce = \grcalcd
   \advance \grcalce by -\tfactor
   \put(\grcalca, \grcalcc){\line(0,-1){\tfactor}}
   \put(\grcalca, \grcalcd){\line(-1,0){\factor}}
   \put(\grcalca, \grcalcd){\line(-3,-1){\factor}}
   \put(\grcalcb, \grcalcc){\line(0,-1){\factor}}
   \advance \grcolumn by 2}

 \newcommand{\glcm}{
   \grcalca = \grcolumn
   \multiply \grcalca by \factor
   \advance \grcalca by \hfactor
   \grcalcb = \grcalca
   \advance \grcalcb by \factor
   \grcalcc = \grrow
   \advance \grcalcc by -1
   \multiply \grcalcc by \factor
   \grcalcd = \grcalcc
   \advance \grcalcd by \tfactor
   \grcalce = \grcalcd
   \advance \grcalce by \tfactor
   \put(\grcalca, \grcalcc){\line(0,1){\tfactor}}
   \put(\grcalca, \grcalcd){\line(1,0){\factor}}
   \put(\grcalca, \grcalcd){\line(3,1){\factor}}
   \put(\grcalcb, \grcalcc){\line(0,1){\factor}}
   \advance \grcolumn by 2}

 \newcommand{\grcm}{
   \grcalcb = \grcolumn
   \multiply \grcalcb by \factor
   \advance \grcalcb by \hfactor
   \grcalca = \grcalcb
   \advance \grcalca by \factor
   \grcalcc = \grrow
   \advance \grcalcc by -1
   \multiply \grcalcc by \factor
   \grcalcd = \grcalcc
   \advance \grcalcd by \tfactor
   \grcalce = \grcalcd
   \advance \grcalce by \tfactor
   \put(\grcalca, \grcalcc){\line(0,1){\tfactor}}
   \put(\grcalca, \grcalcd){\line(-1,0){\factor}}
   \put(\grcalca, \grcalcd){\line(-3,1){\factor}}
   \put(\grcalcb, \grcalcc){\line(0,1){\factor}}
   \advance \grcolumn by 2}

 \newcommand{\gwmu}[1]{    
   \grcalca = \grcolumn
   \multiply \grcalca by \factor
   \grcalcd = \hfactor
   \multiply \grcalcd by #1
   \advance \grcalca by \grcalcd
   \grcalcb = \grrow
   \multiply \grcalcb by \factor
   \grcalcc = \factor
   \advance \grcalcc by \hfactor
   \grcalcd = #1
   \advance \grcalcd by -1
   \multiply \grcalcd by \factor
   \put(\grcalca,\grcalcb){\oval(\grcalcd,\grcalcc)[b]}
   \advance \grcalcb by -\hfactor
   \advance \grcalcb by -\qfactor
   \put(\grcalca,\grcalcb) {\line(0,-1){\qfactor}} 
   \advance \grcolumn by #1}

 \newcommand{\gwcm}[1]{   
   \grcalca = \grcolumn
   \multiply \grcalca by \factor
   \grcalcd = \hfactor
   \multiply \grcalcd by #1
   \advance \grcalca by \grcalcd
   \grcalcb = \grrow
   \advance \grcalcb by -1
   \multiply \grcalcb by \factor
   \grcalcc = \factor
   \advance \grcalcc by \hfactor
   \grcalcd = #1
   \advance \grcalcd by -1
   \multiply \grcalcd by \factor
   \put(\grcalca,\grcalcb){\oval(\grcalcd,\grcalcc)[t]}
   \advance \grcalcb by \factor
   \put(\grcalca,\grcalcb) {\line(0,-1){\qfactor}} 
   \advance \grcolumn by #1}

 \newcommand{\gwmuc}[1]{    
   \grcalca = \grcolumn
   \multiply \grcalca by \factor
   \advance \grcalca by \hfactor
   \grcalcb = \grrow
   \multiply \grcalcb by \factor
   \grcalcc = #1
   \advance \grcalcc by -1
   \multiply \grcalcc by \factor
   \put(\grcalca,\grcalcb){\line(1,0){\grcalcc}}
   \advance \grcalca by -\hfactor
   \grcalcd = \hfactor
   \multiply \grcalcd by #1
   \advance \grcalca by \grcalcd
   \grcalcc = \factor
   \advance \grcalcc by \hfactor
   \grcalcd = #1
   \advance \grcalcd by -1
   \multiply \grcalcd by \factor
   \put(\grcalca,\grcalcb){\oval(\grcalcd,\grcalcc)[b]}
   \advance \grcalcb by -\hfactor
   \advance \grcalcb by -\qfactor
   \put(\grcalca,\grcalcb) {\line(0,-1){\qfactor}} 
   \advance \grcolumn by #1}

 \newcommand{\gwcmc}[1]{   
   \grcalca = \grcolumn
   \multiply \grcalca by \factor
   \advance \grcalca by \hfactor
   \grcalcb = \grrow
   \multiply \grcalcb by \factor
   \advance \grcalcb by -\factor
   \grcalcc = #1
   \advance \grcalcc by -1
   \multiply \grcalcc by \factor
   \put(\grcalca,\grcalcb){\line(1,0){\grcalcc}}
   \grcalcd = #1
   \advance \grcalcd by -1
   \multiply \grcalcd by \hfactor
   \advance \grcalca by \grcalcd
   \grcalcc = \factor
   \advance \grcalcc by \hfactor
   \grcalcd = #1
   \advance \grcalcd by -1
   \multiply \grcalcd by \factor
   \put(\grcalca,\grcalcb){\oval(\grcalcd,\grcalcc)[t]}
   \advance \grcalcb by \factor
   \put(\grcalca,\grcalcb) {\line(0,-1){\qfactor}} 
   \advance \grcolumn by #1}

 \newcommand{\gev}{  
   \grcalca = \grcolumn
   \advance \grcalca by 1
   \multiply \grcalca by \factor
   \grcalcb = \grrow
   \multiply \grcalcb by \factor
   \grcalcc = \factor
   \advance \grcalcc by \hfactor
   \put(\grcalca,\grcalcb){\oval(\factor,\grcalcc)[b]}
   \advance \grcolumn by 2}

 \newcommand{\gdb}{   
   \grcalca = \grcolumn
   \advance \grcalca by 1
   \multiply \grcalca by \factor
   \grcalcb = \grrow
   \advance \grcalcb by -1
   \multiply \grcalcb by \factor
   \grcalcc = \factor
   \advance \grcalcc by \hfactor
   \put(\grcalca,\grcalcb){\oval(\factor,\grcalcc)[t]}
   \advance \grcolumn by 2}

 \newcommand{\gwev}[1]{    
   \grcalca = \grcolumn
   \multiply \grcalca by \factor
   \grcalcd = \hfactor
   \multiply \grcalcd by #1
   \advance \grcalca by \grcalcd
   \grcalcb = \grrow
   \multiply \grcalcb by \factor
   \grcalcc = \factor
   \advance \grcalcc by \hfactor
   \grcalcd = #1
   \advance \grcalcd by -1
   \multiply \grcalcd by \factor
   \put(\grcalca,\grcalcb){\oval(\grcalcd,\grcalcc)[b]}
   \advance \grcolumn by #1}

 \newcommand{\gwdb}[1]{   
   \grcalca = \grcolumn
   \multiply \grcalca by \factor
   \grcalcd = \hfactor
   \multiply \grcalcd by #1
   \advance \grcalca by \grcalcd
   \grcalcb = \grrow
   \advance \grcalcb by -1
   \multiply \grcalcb by \factor
   \grcalcc = \factor
   \advance \grcalcc by \hfactor
   \grcalcd = #1
   \advance \grcalcd by -1
   \multiply \grcalcd by \factor
   \put(\grcalca,\grcalcb){\oval(\grcalcd,\grcalcc)[t]}
   \advance \grcolumn by #1}

 \newcommand{\gbr}{
   \grcalca = \grcolumn
   \multiply \grcalca by \factor
   \advance \grcalca by \hfactor
   \grcalcb = \grcalca
   \advance \grcalcb by \hfactor
   \grcalcc = \grcalca
   \advance \grcalcc by \factor
   \grcalcd = \grrow
   \multiply \grcalcd by \factor
   \grcalce = \grcalcd
   \advance \grcalce by -\tfactor
   \grcalcf = \grcalcd
   \advance \grcalcf by -\hfactor
   \grcalcg = \grcalce
   \advance \grcalcg by -\tfactor
   \grcalch = \grcalcd
   \advance \grcalch by -\factor
   \qbezier(\grcalca,\grcalcd)(\grcalca,\grcalce)(\grcalcb,\grcalcf) 
   \qbezier(\grcalcb,\grcalcf)(\grcalcc,\grcalcg)(\grcalcc,\grcalch) 
   \advance \grcalcf by -\dfactor
   \advance \grcalcb by -\sfactor
   \qbezier(\grcalca,\grcalch)(\grcalca,\grcalcg)(\grcalcb,\grcalcf) 
   \advance \grcalcf by \sfactor
   \advance \grcalcb by \tfactor
   \qbezier(\grcalcc,\grcalcd)(\grcalcc,\grcalce)(\grcalcb,\grcalcf) 
   \advance \grcolumn by 2}

 \newcommand{\gibr}{
   \grcalca = \grcolumn
   \multiply \grcalca by \factor
   \advance \grcalca by \hfactor
   \grcalcb = \grcalca
   \advance \grcalcb by \hfactor
   \grcalcc = \grcalca
   \advance \grcalcc by \factor
   \grcalcd = \grrow
   \multiply \grcalcd by \factor
   \grcalce = \grcalcd
   \advance \grcalce by -\tfactor
   \grcalcf = \grcalcd
   \advance \grcalcf by -\hfactor
   \grcalcg = \grcalce
   \advance \grcalcg by -\tfactor
   \grcalch = \grcalcd
   \advance \grcalch by -\factor
   \qbezier(\grcalcc,\grcalcd)(\grcalcc,\grcalce)(\grcalcb,\grcalcf) 
   \qbezier(\grcalcb,\grcalcf)(\grcalca,\grcalcg)(\grcalca,\grcalch) 
   \advance \grcalcf by -\dfactor
   \advance \grcalcb by \sfactor
   \qbezier(\grcalcc,\grcalch)(\grcalcc,\grcalcg)(\grcalcb,\grcalcf) 
   \advance \grcalcf by \sfactor
   \advance \grcalcb by -\tfactor
   \qbezier(\grcalca,\grcalcd)(\grcalca,\grcalce)(\grcalcb,\grcalcf) 
   \advance \grcolumn by 2}

\newcommand{\gsy}{
   \grcalca = \grcolumn
   \multiply \grcalca by \factor
   \advance \grcalca by \hfactor
   \grcalcb = \grcalca
   \advance \grcalcb by \hfactor
   \grcalcc = \grcalca
   \advance \grcalcc by \factor
   \grcalcd = \grrow
   \multiply \grcalcd by \factor
   \grcalce = \grcalcd
   \advance \grcalce by -\tfactor
   \grcalcf = \grcalcd
   \advance \grcalcf by -\hfactor
   \grcalcg = \grcalce
   \advance \grcalcg by -\tfactor
   \grcalch = \grcalcd
   \advance \grcalch by -\factor
   \qbezier(\grcalcc,\grcalcd)(\grcalcc,\grcalce)(\grcalcb,\grcalcf) 
   \qbezier(\grcalcb,\grcalcf)(\grcalca,\grcalcg)(\grcalca,\grcalch) 
   \advance \grcalcf by -\dfactor
   \advance \grcalcb by \sfactor
   \qbezier(\grcalcc,\grcalch)(\grcalcc,\grcalcg)(\grcalcb,\grcalcf) 
   \qbezier(\grcalca,\grcalcd)(\grcalca,\grcalce)(\grcalcb,\grcalcf) 
   \advance \grcolumn by 2}

 \newcommand{\gbrc}{
   \grcalca = \grcolumn
   \multiply \grcalca by \factor
   \advance \grcalca by \hfactor
   \grcalcb = \grcalca
   \advance \grcalcb by \hfactor
   \grcalcc = \grcalca
   \advance \grcalcc by \factor
   \grcalcd = \grrow
   \multiply \grcalcd by \factor
   \grcalce = \grcalcd
   \advance \grcalce by -\tfactor
   \grcalcf = \grcalcd
   \advance \grcalcf by -\hfactor
   \grcalcg = \grcalce
   \advance \grcalcg by -\tfactor
   \grcalch = \grcalcd
   \advance \grcalch by -\factor
   \put(\grcalcb,\grcalcf){\circle{\hfactor}}
   \qbezier(\grcalca,\grcalcd)(\grcalca,\grcalce)(\grcalcb,\grcalcf) 
   \qbezier(\grcalcb,\grcalcf)(\grcalcc,\grcalcg)(\grcalcc,\grcalch) 
   \advance \grcalcf by -\dfactor
   \advance \grcalcb by -\sfactor
   \qbezier(\grcalca,\grcalch)(\grcalca,\grcalcg)(\grcalcb,\grcalcf) 
   \advance \grcalcf by \sfactor
   \advance \grcalcb by \tfactor
   \qbezier(\grcalcc,\grcalcd)(\grcalcc,\grcalce)(\grcalcb,\grcalcf) 
   \advance \grcolumn by 2}

 \newcommand{\gibrc}{
   \grcalca = \grcolumn
   \multiply \grcalca by \factor
   \advance \grcalca by \hfactor
   \grcalcb = \grcalca
   \advance \grcalcb by \hfactor
   \grcalcc = \grcalca
   \advance \grcalcc by \factor
   \grcalcd = \grrow
   \multiply \grcalcd by \factor
   \grcalce = \grcalcd
   \advance \grcalce by -\tfactor
   \grcalcf = \grcalcd
   \advance \grcalcf by -\hfactor
   \grcalcg = \grcalce
   \advance \grcalcg by -\tfactor
   \grcalch = \grcalcd
   \advance \grcalch by -\factor
   \put(\grcalcb,\grcalcf){\circle{\hfactor}}
   \qbezier(\grcalcc,\grcalcd)(\grcalcc,\grcalce)(\grcalcb,\grcalcf) 
   \qbezier(\grcalcb,\grcalcf)(\grcalca,\grcalcg)(\grcalca,\grcalch) 
   \advance \grcalcf by -\dfactor
   \advance \grcalcb by \sfactor
   \qbezier(\grcalcc,\grcalch)(\grcalcc,\grcalcg)(\grcalcb,\grcalcf) 
   \advance \grcalcf by \sfactor
   \advance \grcalcb by -\tfactor
   \qbezier(\grcalca,\grcalcd)(\grcalca,\grcalce)(\grcalcb,\grcalcf) 
   \advance \grcolumn by 2}

 \newcommand{\gu}[1]{
   \grcalca = \grcolumn
   \multiply \grcalca by \factor
   \grcalcd = \hfactor
   \multiply \grcalcd by #1
   \advance \grcalca by \grcalcd
   \grcalcb = \grrow
   \advance \grcalcb by -1
   \multiply \grcalcb by \factor
   \put(\grcalca,\grcalcb) {\line(0,1){\hfactor}} 
   \advance \grcalcb by \hfactor
   \put(\grcalca,\grcalcb) {\circle*{3}}
   \advance \grcolumn by #1}

 \newcommand{\gcu}[1]{
   \grcalca = \grcolumn
   \multiply \grcalca by \factor
   \grcalcd = \hfactor
   \multiply \grcalcd by #1
   \advance \grcalca by \grcalcd
   \grcalcb = \grrow
   \multiply \grcalcb by \factor
   \put(\grcalca,\grcalcb) {\line(0,-1){\hfactor}} 
   \advance \grcalcb by -\hfactor
   \put(\grcalca,\grcalcb) {\circle*{3}}
   \advance \grcolumn by #1}

 \newcommand{\gmp}[1]{
   \grcalca = \grcolumn
   \multiply \grcalca by \factor
   \advance \grcalca by \hfactor
   \grcalcb = \grrow
   \multiply \grcalcb by \factor
   \put(\grcalca,\grcalcb) {\line(0,-1){\dfactor}} 
   \advance \grcalcb by -\factor
   \put(\grcalca,\grcalcb) {\line(0,1){\dfactor}} 
   \advance \grcalcb by \hfactor
   \grcalcc = \factor
   \advance \grcalcc by -\qfactor
   \put(\grcalca,\grcalcb) {\circle{\grcalcc}}
   \put(\grcalca,\grcalcb) {\makebox(0,0){$\scriptstyle #1$}}
   \advance \grcolumn by 1}

 \newcommand{\gbmp}[1]{
   \grcalca = \grcolumn
   \multiply \grcalca by \factor
   \advance \grcalca by \hfactor
   \grcalcb = \grrow
   \multiply \grcalcb by \factor
   \put(\grcalca,\grcalcb) {\line(0,-1){\dfactor}} 
   \advance \grcalcb by -\factor
   \put(\grcalca,\grcalcb) {\line(0,1){\dfactor}} 
   \advance \grcalca by -\hfactor
   \advance \grcalca by \dfactor
   \advance \grcalcb by \dfactor
   \grcalcc = \factor
   \advance \grcalcc by -\sfactor
   \put(\grcalca,\grcalcb) {\framebox(\grcalcc,\grcalcc){$\scriptstyle #1$}}
   \advance \grcolumn by 1}

 \newcommand{\gbmpt}[1]{
   \grcalca = \grcolumn
   \multiply \grcalca by \factor
   \advance \grcalca by \hfactor
   \grcalcb = \grrow
   \multiply \grcalcb by \factor
   \put(\grcalca,\grcalcb) {\line(0,-1){\dfactor}} 
   \advance \grcalcb by -\factor
   \advance \grcalca by -\hfactor
   \advance \grcalca by \dfactor
   \advance \grcalcb by \dfactor
   \grcalcc = \factor
   \advance \grcalcc by -\sfactor
   \put(\grcalca,\grcalcb) {\framebox(\grcalcc,\grcalcc){$\scriptstyle #1$}}
   \advance \grcolumn by 1}

 \newcommand{\gbmpb}[1]{
   \grcalca = \grcolumn
   \multiply \grcalca by \factor
   \advance \grcalca by \hfactor
   \grcalcb = \grrow
   \multiply \grcalcb by \factor
   \advance \grcalcb by -\factor
   \put(\grcalca,\grcalcb) {\line(0,1){\dfactor}} 
   \advance \grcalca by -\hfactor
   \advance \grcalca by \dfactor
   \advance \grcalcb by \dfactor
   \grcalcc = \factor
   \advance \grcalcc by -\sfactor
   \put(\grcalca,\grcalcb) {\framebox(\grcalcc,\grcalcc){$\scriptstyle #1$}}
   \advance \grcolumn by 1}

 \newcommand{\gbmpn}[1]{
   \grcalca = \grcolumn
   \multiply \grcalca by \factor
   \advance \grcalca by \hfactor
   \grcalcb = \grrow
   \multiply \grcalcb by \factor
   \advance \grcalcb by -\factor
   \advance \grcalca by -\hfactor
   \advance \grcalca by \dfactor
   \advance \grcalcb by \dfactor
   \grcalcc = \factor
   \advance \grcalcc by -\sfactor
   \put(\grcalca,\grcalcb) {\framebox(\grcalcc,\grcalcc){$\scriptstyle #1$}}
   \advance \grcolumn by 1}

 \newcommand{\glmptb}{    
   \grcalca = \grcolumn
   \multiply \grcalca by \factor
   \advance \grcalca by \hfactor
   \grcalcb = \grrow
   \multiply \grcalcb by \factor
   \put(\grcalca,\grcalcb) {\line(0,-1){\dfactor}} 
   \advance \grcalcb by -\factor
   \put(\grcalca,\grcalcb) {\line(0,1){\dfactor}} 
   \advance \grcalca by -\hfactor
   \advance \grcalca by \dfactor
   \advance \grcalcb by \dfactor
   \put(\grcalca,\grcalcb) {\line(1,0){\factor}} 
   \advance \grcalcb by \factor
   \advance \grcalcb by -\sfactor
   \put(\grcalca,\grcalcb) {\line(1,0){\factor}} 
   \grcalcc = \factor
   \advance \grcalcc by -\sfactor
   \put(\grcalca,\grcalcb) {\line(0,-1){\grcalcc}} 
   \advance \grcolumn by 1}

 \newcommand{\glmpt}{    
   \grcalca = \grcolumn
   \multiply \grcalca by \factor
   \advance \grcalca by \hfactor
   \grcalcb = \grrow
   \multiply \grcalcb by \factor
   \put(\grcalca,\grcalcb) {\line(0,-1){\dfactor}} 
   \advance \grcalca by -\hfactor
   \advance \grcalca by \dfactor
   \advance \grcalcb by -\dfactor
   \put(\grcalca,\grcalcb) {\line(1,0){\factor}} 
   \advance \grcalcb by -\factor
   \advance \grcalcb by \sfactor
   \put(\grcalca,\grcalcb) {\line(1,0){\factor}} 
   \grcalcc = \factor
   \advance \grcalcc by -\sfactor
   \put(\grcalca,\grcalcb) {\line(0,1){\grcalcc}} 
   \advance \grcolumn by 1}

 \newcommand{\glmpb}{    
   \grcalca = \grcolumn
   \multiply \grcalca by \factor
   \advance \grcalca by \hfactor
   \grcalcb = \grrow
   \multiply \grcalcb by \factor
   \advance \grcalcb by -\factor
   \put(\grcalca,\grcalcb) {\line(0,1){\dfactor}} 
   \advance \grcalca by -\hfactor
   \advance \grcalca by \dfactor
   \advance \grcalcb by \dfactor
   \put(\grcalca,\grcalcb) {\line(1,0){\factor}} 
   \advance \grcalcb by \factor
   \advance \grcalcb by -\sfactor
   \put(\grcalca,\grcalcb) {\line(1,0){\factor}} 
   \grcalcc = \factor
   \advance \grcalcc by -\sfactor
   \put(\grcalca,\grcalcb) {\line(0,-1){\grcalcc}} 
   \advance \grcolumn by 1}

 \newcommand{\glmp}{    
   \grcalca = \grcolumn
   \multiply \grcalca by \factor
   \advance \grcalca by \dfactor
   \grcalcb = \grrow
   \multiply \grcalcb by \factor
   \advance \grcalcb by -\dfactor
   \put(\grcalca,\grcalcb) {\line(1,0){\factor}} 
   \advance \grcalcb by -\factor
   \advance \grcalcb by \sfactor
   \put(\grcalca,\grcalcb) {\line(1,0){\factor}} 
   \grcalcc = \factor
   \advance \grcalcc by -\sfactor
   \put(\grcalca,\grcalcb) {\line(0,1){\grcalcc}} 
   \advance \grcolumn by 1}

 \newcommand{\gcmptb}{    
   \grcalca = \grcolumn
   \multiply \grcalca by \factor
   \advance \grcalca by \hfactor
   \grcalcb = \grrow
   \multiply \grcalcb by \factor
   \put(\grcalca,\grcalcb) {\line(0,-1){\dfactor}} 
   \advance \grcalcb by -\factor
   \put(\grcalca,\grcalcb) {\line(0,1){\dfactor}} 
   \advance \grcalca by -\hfactor
   \advance \grcalcb by \dfactor
   \put(\grcalca,\grcalcb) {\line(1,0){\factor}} 
   \advance \grcalcb by \factor
   \advance \grcalcb by -\sfactor
   \put(\grcalca,\grcalcb) {\line(1,0){\factor}} 
   \advance \grcolumn by 1}

\newcommand{\gmpcu}[1]{
   \grcalca = \grcolumn
   \multiply \grcalca by \factor
   \advance \grcalca by \hfactor
   \grcalcb = \grrow
   \multiply \grcalcb by \factor
   \put(\grcalca,\grcalcb) {\line(0,-1){\dfactor}} 
   \advance \grcalcb by -\factor
   \advance \grcalcb by \hfactor
   \grcalcc = \factor
   \advance \grcalcc by -\qfactor
   \put(\grcalca,\grcalcb) {\circle{\grcalcc}}
   \put(\grcalca,\grcalcb) {\makebox(0,0){$\scriptstyle #1$}}
   \advance \grcolumn by 1}
   
\newcommand{\gmpu}[1]{
   \grcalca = \grcolumn
   \multiply \grcalca by \factor
   \advance \grcalca by \hfactor
   \grcalcb = \grrow
   \multiply \grcalcb by \factor
   \advance \grcalcb by -\factor
   \put(\grcalca,\grcalcb) {\line(0,1){\dfactor}} 
   \advance \grcalcb by \hfactor
   \grcalcc = \factor
   \advance \grcalcc by -\qfactor
   \put(\grcalca,\grcalcb) {\circle{\grcalcc}}
   \put(\grcalca,\grcalcb) {\makebox(0,0){$\scriptstyle #1$}}
   \advance \grcolumn by 1}      

 \newcommand{\gcmpt}{    
   \grcalca = \grcolumn
   \multiply \grcalca by \factor
   \advance \grcalca by \hfactor
   \grcalcb = \grrow
   \multiply \grcalcb by \factor
   \put(\grcalca,\grcalcb) {\line(0,-1){\dfactor}} 
   \advance \grcalcb by -\factor
   \advance \grcalca by -\hfactor
   \advance \grcalcb by \dfactor
   \put(\grcalca,\grcalcb) {\line(1,0){\factor}} 
   \advance \grcalcb by \factor
   \advance \grcalcb by -\sfactor
   \put(\grcalca,\grcalcb) {\line(1,0){\factor}} 
   \advance \grcolumn by 1}

 \newcommand{\gcmpb}{    
   \grcalca = \grcolumn
   \multiply \grcalca by \factor
   \advance \grcalca by \hfactor
   \grcalcb = \grrow
   \multiply \grcalcb by \factor
   \advance \grcalcb by -\factor
   \put(\grcalca,\grcalcb) {\line(0,1){\dfactor}} 
   \advance \grcalca by -\hfactor
   \advance \grcalcb by \dfactor
   \put(\grcalca,\grcalcb) {\line(1,0){\factor}} 
   \advance \grcalcb by \factor
   \advance \grcalcb by -\sfactor
   \put(\grcalca,\grcalcb) {\line(1,0){\factor}} 
   \advance \grcolumn by 1}

 \newcommand{\gcmp}{    
   \grcalca = \grcolumn
   \multiply \grcalca by \factor
   \grcalcb = \grrow
   \multiply \grcalcb by \factor
   \advance \grcalcb by -\factor
   \advance \grcalcb by \dfactor
   \put(\grcalca,\grcalcb) {\line(1,0){\factor}} 
   \advance \grcalcb by \factor
   \advance \grcalcb by -\sfactor
   \put(\grcalca,\grcalcb) {\line(1,0){\factor}} 
   \advance \grcolumn by 1}

 \newcommand{\grmptb}{    
   \grcalca = \grcolumn
   \multiply \grcalca by \factor
   \advance \grcalca by \hfactor
   \grcalcb = \grrow
   \multiply \grcalcb by \factor
   \put(\grcalca,\grcalcb) {\line(0,-1){\dfactor}} 
   \advance \grcalcb by -\factor
   \put(\grcalca,\grcalcb) {\line(0,1){\dfactor}} 
   \advance \grcalca by \hfactor
   \advance \grcalca by -\dfactor
   \advance \grcalcb by \dfactor
   \put(\grcalca,\grcalcb) {\line(-1,0){\factor}} 
   \advance \grcalcb by \factor
   \advance \grcalcb by -\sfactor
   \put(\grcalca,\grcalcb) {\line(-1,0){\factor}} 
   \grcalcc = \factor
   \advance \grcalcc by -\sfactor
   \put(\grcalca,\grcalcb) {\line(0,-1){\grcalcc}} 
   \advance \grcolumn by 1}

 \newcommand{\grmpt}{    
   \grcalca = \grcolumn
   \multiply \grcalca by \factor
   \advance \grcalca by \hfactor
   \grcalcb = \grrow
   \multiply \grcalcb by \factor
   \put(\grcalca,\grcalcb) {\line(0,-1){\dfactor}} 
   \advance \grcalca by \hfactor
   \advance \grcalca by -\dfactor
   \advance \grcalcb by -\dfactor
   \put(\grcalca,\grcalcb) {\line(-1,0){\factor}} 
   \advance \grcalcb by -\factor
   \advance \grcalcb by \sfactor
   \put(\grcalca,\grcalcb) {\line(-1,0){\factor}} 
   \grcalcc = \factor
   \advance \grcalcc by -\sfactor
   \put(\grcalca,\grcalcb) {\line(0,1){\grcalcc}} 
   \advance \grcolumn by 1}

 \newcommand{\grmpb}{    
   \grcalca = \grcolumn
   \multiply \grcalca by \factor
   \advance \grcalca by \hfactor
   \grcalcb = \grrow
   \multiply \grcalcb by \factor
   \advance \grcalcb by -\factor
   \put(\grcalca,\grcalcb) {\line(0,1){\dfactor}} 
   \advance \grcalca by \hfactor
   \advance \grcalca by -\dfactor
   \advance \grcalcb by \dfactor
   \put(\grcalca,\grcalcb) {\line(-1,0){\factor}} 
   \advance \grcalcb by \factor
   \advance \grcalcb by -\sfactor
   \put(\grcalca,\grcalcb) {\line(-1,0){\factor}} 
   \grcalcc = \factor
   \advance \grcalcc by -\sfactor
   \put(\grcalca,\grcalcb) {\line(0,-1){\grcalcc}} 
   \advance \grcolumn by 1}

 \newcommand{\grmp}{    
   \grcalca = \grcolumn
   \multiply \grcalca by \factor
   \advance \grcalca by \factor
   \advance \grcalca by -\dfactor
   \grcalcb = \grrow
   \multiply \grcalcb by \factor
   \advance \grcalcb by -\dfactor
   \put(\grcalca,\grcalcb) {\line(-1,0){\factor}} 
   \advance \grcalcb by -\factor
   \advance \grcalcb by \sfactor
   \put(\grcalca,\grcalcb) {\line(-1,0){\factor}} 
   \grcalcc = \factor
   \advance \grcalcc by -\sfactor
   \put(\grcalca,\grcalcb) {\line(0,1){\grcalcc}} 
   \advance \grcolumn by 1}

 \newcommand{\gwmuh}[3]{    
   \grcalca = \grcolumn
   \multiply \grcalca by \factor
   \grcalcb = #2
   \advance \grcalcb by #3
   \multiply \grcalcb by \qfactor
   \advance \grcalca by \grcalcb
   \grcalcb = \grrow
   \multiply \grcalcb by \factor
   \grcalcc = #3
   \advance \grcalcc by -#2
   \multiply \grcalcc by \hfactor
   \grcalcd = \factor
   \advance \grcalcd by \hfactor
   \put(\grcalca,\grcalcb){\oval(\grcalcc,\grcalcd)[b]}
   \grcalca = \grcolumn
   \multiply \grcalca by \factor
   \grcalcc = #1
   \multiply \grcalcc by \hfactor
   \advance \grcalca by \grcalcc
   \advance \grcalcb by -\hfactor
   \advance \grcalcb by -\qfactor
   \put(\grcalca,\grcalcb) {\line(0,-1){\qfactor}} 
   \advance \grcolumn by #1}

 \newcommand{\gwcmh}[3]{   
   \grcalca = \grcolumn
   \multiply \grcalca by \factor
   \grcalcb = #2
   \advance \grcalcb by #3
   \multiply \grcalcb by \qfactor
   \advance \grcalca by \grcalcb
   \grcalcb = \grrow
   \advance \grcalcb by -1
   \multiply \grcalcb by \factor
   \grcalcc = #3
   \advance \grcalcc by -#2
   \multiply \grcalcc by \hfactor
   \grcalcd = \factor
   \advance \grcalcd by \hfactor
   \put(\grcalca,\grcalcb){\oval(\grcalcc,\grcalcd)[t]}
   \grcalca = \grcolumn
   \multiply \grcalca by \factor
   \grcalcc = #1
   \multiply \grcalcc by \hfactor
   \advance \grcalca by \grcalcc
   \advance \grcalcb by \factor
   \put(\grcalca,\grcalcb) {\line(0,-1){\qfactor}} 
   \advance \grcolumn by #1}

 \newcommand{\gsbox}[1]{
   \grcalca = \grcolumn
   \multiply \grcalca by \factor
   \grcalcb = \grrow
   \multiply \grcalcb by \factor
   \advance \grcalcb by -\factor
   \grcalcc = #1
   \multiply \grcalcc by \factor
   \grcalcd = \factor
   \put(\grcalca,\grcalcb){\framebox(\grcalcc,\grcalcd){}}}

\allowdisplaybreaks[4]

\begin{document}
\title[Entwined modules over cowreaths]
{Frobenius and separable functors for the category of entwined modules over cowreaths, I: General theory}
\author{D. Bulacu}
\address{Faculty of Mathematics and Informatics, University
of Bucharest, Str. Academiei 14, RO-010014 Bucharest 1, Romania}
\email{daniel.bulacu@fmi.unibuc.ro}
\author{S. Caenepeel}
\address{Faculty of Engineering, 
Vrije Universiteit Brussel, B-1050 Brussels, Belgium}
\email{scaenepe@vub.ac.be}
\author{B. Torrecillas}
\address{Department of Algebra and Analysis\\
Universidad de Almer\'{\i}a\\
E-04071 Almer\'{\i}a, Spain}
\email{btorreci@ual.es}
\thanks{The first author was supported by the UEFISCDI Grant PN-II-ID-PCE-2011-
3-0635, contract no. 253/5.10.2011 of CNCSIS. The second author was supported by research project G.0117.10  
``Equivariant Brauer groups and Galois deformations'' from
FWO-Vlaanderen. The third author was partially supported by FQM 211 from Junta 
Andaluc\'{\i}a and by research project MTM2014-54439 from MEC.
The first author thanks the Vrije Universiteit Brussel and the Universidad 
de Almer\'{\i}a for their support and warm hospitality. 
The authors also thank Bodo Pareigis for sharing his ``diagrams" program.}

\begin{abstract}
Entwined modules over cowreaths in a monoidal category are introduced. They can be identified to coalgebras in an appropriate monoidal category. It is investigated when such coalgebras are Frobenius (resp. separable), and when  the forgetful functor from entwined modules to representations of the underlying algebra is Frobenius (resp. separable). These properties are equivalent when the unit object of the category is a $\otimes$-generator. 
\end{abstract}

\maketitle
\section*{Introduction}
This paper is part of a series that has as the final aim the study of Frobenius and separable 
properties for forgetful functors defined on categories of entwined modules over cowreaths obtained from certain 
quasi-Hopf actions and coactions. In this paper we present a general theory that allows us not only to achieve the 
mentioned goal but also to unify similar results obtained so far for various generalizations of Hopf algebras. It can be seen as a sequel of 
\cite{bc4, dbbt2} and as the theoretical support for \cite{bctFSA}.  

Central elements in the enveloping algebra $A\ot A^{\rm op}$ of an algebra $A$ are often called
Casimir elements, and they play a crucial role in the theory of Frobenius and of separable algebras.
The fact that they appear in both theories is well understood, and has a categorical explanation
related to the properties that an algebra is Frobenius if the restriction of scalars functor $G$ is Frobenius,
that is, its right adjoint is also a left adjoint, and that it is separable if and only if $G$ is separable in the
sense of \cite{nbo}. This can be exploited in order to study Frobenius and separable functors simultaneously.
This idea originated in the study of separability and Frobenius properties for Doi-Hopf modules in 
\cite{cmz,cmz0,cmz2}, and was later refined and applied to entwined modules, see \cite{brFM}.

Entwined modules over entwining structures were introduced by Brzezi\'nski in \cite{br} in order to 
extend the Hopf-Galois theory to coalgebras. 
One of the attractive aspects is that many structures that appear in Hopf algebra theory, such as
relative Hopf modules, Doi-Hopf and Yetter-Drinfeld modules, turn out to be special cases. An entwining 
structure is a kind of local braiding between an algebra and a coalgebra. In fact an entwining structure 
with underlying algebra $A$ can be viewed as a coalgebra in the monoidal category $\Tc_A$ of transfer morphisms
through $A$ as introduced by Tambara in \cite{tambara}. Tambara's construction can be obtained from 
Street's formal theory of monads, see \cite{street}. Monads in a 2-category $\Cc$ can be organized into 
a new 2-category ${\rm Mnd}(\Cc)$. For an algebra (or monad) in a strict monoidal category $\Cc$ (a 2-category with single 0-cell),
Tambara's category $\Tc_A$ is the category ${\rm Mnd}(\Cc)(A,A)$ of endomorphisms of $A$ in ${\rm Mnd}(\Cc)$.

There is a second way to organize monads into a 2-category, see \cite{LackRoss}; the second 2-category is the Eilenberg-Moore 
2-category ${\rm EM}(\Cc)$. It coincides with ${\rm Mnd}(\Cc)$ at the level of 0-cells and 1-cells, but has different 
2-cells. A cowreath in $\Cc$ is a comonad in ${\rm EM}(\Cc)$, and consists of an algebra in $\Cc$ together with 
a coalgebra in $\Tc_A^\#={\rm EM}(\Cc)(A,A)$, the category of endomorphisms of $A$ in the Eilenberg-Moore 2-category. 
Note that, in the case where ${\cal K}$ is a $2$-category, a comonad in in $EM({\cal K})$ was called by Street a mixed wreath, 
see \cite{StreetMW}. So the cowreaths we are dealing with are nothing but mixed wreaths (or comonads) in $EM({\cal K})$ in the sense of 
Street, in the case where ${\cal K}$ is a $2$-category with a single $0$-cell. If this is the case, we can introduce 
entwined modules over a cowreath. The main aim of this paper is to study when
the forgetful functor from entwined modules to $A$-modules is Frobenius or separable. This is related to
the question when a coalgebra in $\Tc_A^\#$ is a Frobenius or a coseparable coalgebra.

Compared to the classical situation, we have a two-fold generalization: first of all, the category of vector
spaces is replaced by an arbitrary (strict) monoidal category $\Cc$. The best results
are obtained in the situation where the unit object $\un{1}$ is a $\ot$-generator of the monoidal category $\Cc$,
as introduced in \cite{dbbt2}. The following monoidal categories satisfy this condition: 
the category of vector spaces, the category of bimodules ${}_R{\cal M}_R$ over an Azumaya $k$-algebra $R$,
the category of finite dimensional  Hilbert complex vector spaces ${\rm FdHilb}$, and the category ${\cal Z}_k$
as introduced in \cite{cl}. We refer to \cite[Examples 3.2]{dbbt2}.

Secondly, we work over cowreaths which can be viewed as generalized entwining structures. Our motivation 
to investigate such cowreaths comes from the applications that we have in mind, namely the study of categories 
of Doi-Hopf modules, two-sided Hopf modules and Yetter-Drinfeld modules 
over a quasi-Hopf algebra, which can be defined as entwined modules over certain cowreaths that are not ordinary entwining 
structures. This study will be done in the forthcoming paper \cite{bctFSA}.

In Sections \ref{se:2}-\ref{se:6}, we present our general theory. In \seref{1}, we present preliminary results on 
monoidal categories and bimodules. In \seref{2}, we introduce cowreaths in monoidal categories, 
and entwined modules over them. In \seref{3}, we introduce generalized factorization structures; these are algebras 
in $\Tc_A^{\#}$, or, equivalently, wreaths in $\Cc$. Given a generalized factorization structure, we can define an algebra 
in $\Cc$, called the wreath product algebra or the generalized smash product. Duality arguments turn cowreaths into 
generalized factorization structures, and the category of entwined modules is isomorphic to the category of
modules over the generalized smash product, see \thref{gntasmod}. In \seref{4}, we discuss when the 
forgetful functor $F$ is Frobenius. $F$ always has a right adjoint $G$; in order to investigate when $G$ is also
a left adjoint, we need to investigate natural transformations from the identity functor to $FG$, and from $GF$ to
the identity functor. Propositions \ref{pr:FrobElem} and \ref{pr:CasMor} tell us that the necessary and sufficient information that
is needed to produce such natural transformations is encoded in the so-called Frobenius elements and
Casimir morphisms, at least in the case where $\un{1}$ is a $\ot$-generator. Using these results, it is straightforward to
prove the main \thref{FrobGenEntwMod}, stating that $F$ is a Frobenius functor if and only if the coalgebra corresponding
to the given cowreath is Frobenius. In \seref{5} it is shown that there is a strong monoidal
functor from the category of generalized transfer morphisms $\Tc_A^{\#}$ to the category of $A$-bimodules,
as introduced in the preliminary \seref{bimod}. Consequently, a cowreath produces
an $A$-coring, that is a coalgebra in the category of $A$-bimodules. The main result is that this $A$-coring is
Frobenius if and only if the corresponding coalgebra $(A, X)$ in $\Tc_A^{\#}$ is Frobenius, see \thref{coFrobTvscoring}.
Under the assumption that $X$ has a right adjoint $Y$, we have additional results, see \thref{Frobcaractfinitecase}.
Separability is investigated in \seref{6}. The main result is \thref{ForFunsepvsSepcoal} stating that
a coalgebra $(X, \psi)$ in $\Tc_A^{\#}$ is coseparable if and only if the forgetful functor is coseparable. Again, additional results
can be stated if $X$ has a right adjoint.

Our theory can be applied to various cowreaths coming from (co)actions of Hopf algebras and their generalizations, 
see Section 5 of the paper \cite{bc4}. But perhaps the most interesting are those cowreaths $(A, X)$ with $X$ regarded as 
an object in ${\cal T}_A^\#$ rather than ${\cal T}_A$. Such examples occur in the quasi-Hopf case, leading, for instance, 
to categories of Doi-Hopf modules, two-sided Hopf modules and Yetter-Drinfeld modules 
over a quasi-Hopf algebra, respectively. As we already mentioned above, when they are Frobenius or separable cowreaths 
will be the topic of the forthcoming paper \cite{bctFSA}. 
\section{Preliminaries}\selabel{1}
\subsection{Monoidal categories}\selabel{1.1}
\subsubsection*{Monoidal categories}
A monoidal category is a category $\Cc$ together with 
a functor $\ot:\ \Cc\times\Cc\to \Cc$, called the tensor product, an
object $\un{1}\in \Cc$, called the unit object, and natural isomorphisms 
$a:\ \ot\circ (\ot\times {\rm Id})\to \ot\circ ({\rm Id}\times \ot)$ (the associativity
constraint), $l:\ \ot\circ (\un{1}\times {\rm Id})\to {\rm Id}$ (the left unit constraint) and 
$r:\ \ot\circ ({\rm Id}\times \un{1})\to {\rm Id}$ (the right unit constraint) satisfying appropriate
coherence conditions, see for example \cite[XI.2]{k} for a detailed discussion.
$\Cc$ is called strict if $a$, $l$ and $r$ are the identity natural transformations.
It is well-known that every monoidal category is monoidally equivalent to a strict monoidal
category, and this enables us to assume without loss of generality that $\Cc$ is strict. 
We will often delete the tensor symbol $\ot$, and write $X\ot Y=XY$. We write $X^n$ for the
tensor product of $n$ copies of $X$.
The identity 
morphism of an object $X\in \Cc$ will be denoted by ${\rm Id}_X$ or simply $X$.
For morphisms ${\rm Id}_X=X :\ X\ra X$, 
$f:\ X\ra Y$, $g: X Y\ra Z$ and $h:\ X\ra Y Z$ in $\Cc$, we adopt the following graphical notation
\[{\rm Id}_X=X=
{\footnotesize
\gbeg{1}{3}
\got{1}{X}\gnl
\gcl{1}\gnl
\gob{1}{X}
\gend}
\hspace{2mm},\hspace{2mm}
f= {\footnotesize
\gbeg{1}{3}
\got{1}{X}\gnl
\gmp{f}\gnl
\gob{1}{Y}
\gend}
\mbox{\hspace{2mm},\hspace{2mm}}
g={\footnotesize
\gbeg{2}{5}
\got{1}{X}\got{1}{Y}\gnl
\gcl{1}\gcl{1}\gnl
\gsbox{2}\gnot{\hspace{5mm}g}\gnl
\gcn{1}{1}{2}{2}\gnl
\gob{2}{Z}
\gend}
\mbox{\hspace{2mm}and\hspace{2mm}}
h={\footnotesize 
\gbeg{2}{5}
\got{2}{X}\gnl
\gcn{1}{1}{2}{2}\gnl
\gsbox{2}\gnot{\hspace{5mm}h}\gnl
\gcl{1}\gcl{1}\gnl
\gob{1}{Y}\gob{1}{Z}
\gend}~~.
\]

\subsubsection*{Algebras and coalgebras}
An algebra  in $\Cc$ is a triple $(A,{m},\eta)$, where
$A$ is an object in $\Cc$ and ${m}:\ AA\ra A$ (the multiplication) and ${\eta}:\ \un{1}\ra A$
(the unit) are morphisms in $\Cc$ satisfying the associativity and unit conditions
${m}\circ mA= {m}\circ Am$ and
${m}\circ \eta A= {m}\circ A\eta =A$.
The graphical notation for ${m}$ and $\eta$ is the following:
$${m}=\gbeg{2}{3}
\got{1}{A}\got{1}{A}\gnl
\gmu\gnl
\gob{2}{A}
\gend~~{\rm and}~~
{\eta}= \gbeg{1}{3}
\got{1}{\un{1}}\gnl
\gu{1}\gnl
\gob{1}{A}
\gend.
$$
We use $A$ as a shorter notation for the algebra $(A,m,\eta)$; the multiplication on an algebra $A$
is typically denoted by $m$, and the unit by $\eta$; we put subscripts whenever convenient,
so that we can write $A=(A,m_A,\eta_A)$. Similar conventions are used for other structures,
such as coalgebras, modules over an algebra, adjunctions, entwining structures etc.\\
A coalgebra in $\Cc$ is a triple $C=(C,~\Delta:\ C\to C C,~\varepsilon:\ C\to \un{1})$, satisfying the
appropriate coassociativity and counit conditions.
The graphical notation takes the form
$$
\Delta=\gbeg{2}{3}
\got{2}{C}\gnl
\gcmu\gnl
\gob{1}{C}\gob{1}{C}
\gend~~{\rm and}~~
\varepsilon=\gbeg{1}{3}
\got{1}{C}\gnl
\gcu{1}\gnl
\gob{1}{\un{1}}
\gend.
$$

\subsubsection*{Adjunctions}
An adjunction $X\dashv Y$ in $\Cc$ is a quadruple $(X,Y,b,d)$, with $X,Y$ objects in $\Cc$ and morphisms
$b:\ \un{1}\to YX$ and $d:\ XY\to \un{1}$ satisfying
\begin{equation}\eqlabel{evcoev}
Yd\circ bY=Y~~{\rm and}~~dX\circ Xb=X.
\end{equation}
With the graphical notation
$$d=\gbeg{2}{3}
\got{1}{X}\got{1}{Y}\gnl
\gev\gnl
\gob{2}{\un{1}}
\gend,~~
b=\gbeg{2}{3}
\got{2}{\un{1}}\gnl
\gdb\gnl
\gob{1}{Y}\gob{1}{X}
\gend$$
\equref{evcoev} can be rewritten as
$$
{\footnotesize
\gbeg{3}{4}
\gvac{2}\got{1}{Y}\gnl
\gdb\gcl{1}\gnl
\gcl{1}\gev\gnl
\gob{1}{Y}
\gend}=
{\footnotesize
\gbeg{1}{3}
\got{1}{Y}\gnl
\gcl{1}\gnl
\gob{1}{Y}
\gend}~,~
{\footnotesize
\gbeg{3}{4}
\got{1}{X}\gnl
\gcl{1}\gdb\gnl
\gev\gcl{1}\gnl
\gvac{2}\gob{1}{X}
\gend}=
{\footnotesize
\gbeg{1}{3}
\got{1}{X}\gnl
\gcl{1}\gnl
\gob{1}{X}
\gend}~~.
$$
$Y$ is called a right adjoint of $X$, and $X$ a left adjoint of $Y$.
Right adjoints are unique in the following sense. If $(X,Y',b',d')$ is another adjunction,
then $\lambda=Y'd\circ b'Y:\ Y\to Y'$ is an isomorphism with inverse $\lambda^{-1}=Yd'\circ bY'$.
It is easy to show that $\lambda$ is designed in such a way that
\begin{equation}\eqlabel{1.1.1}
b'=\lambda X\circ b~~{\rm and}~~d=d'\circ X\lambda.
\end{equation}
For two adjunctions $(X,Y,b,d)$ and $(X',Y',b',d')$, we have a new adjunction
$$(XX',Y'Y,b\cdot b'=Y'bX'\circ b',d\cdot d'=d\circ Xd'Y).$$
In particular, we have an adjunction $(X^2,Y^2,b^2=b\cdot b,d^2=d\cdot d)$. Given a morphism $f:\ X\to X'$,
we have
$$g=Yd' \circ YfY' \circ bY':\ Y'\to X',$$
which reproduces $f=dX' \circ XgX' \circ Xb'$.\\
If every object in $\Cc$ has a right (resp. left) adjoint, then we say that $\Cc$ has right (resp. left) duality;
$\Cc$ is called rigid if it has left and right duality. Assume that $\Cc$ has right duality, and choose a right dual
${}^*\hspace*{-2pt}X$ for every object $X$. For every morphism $f:\ X\to X'$ in $\Cc$, we have
${}^*\hspace*{-1pt}f:\ {}^*\hspace*{-2pt}X'\to {}^*\hspace*{-2pt}X$, and this defines a functor ${}^*\hspace*{-1pt}(-):\ \Cc\to \Cc^{\rm op}$.
If $\Cc$ is rigid, then $\bigl( {}^*\hspace*{-1pt}(-), (-)^*\bigr)$ is a pair of inverse equivalences between $\Cc$ and $\Cc^{\rm op}$.\\
For $X,X'\in \Cc$, ${}^*\hspace*{-1pt}(XX')$ and ${}^*\hspace*{-2pt}X'{}^*\hspace*{-2pt}X$ are right duals of $XX'$,
so we have an isomorphism $\varphi_2(X,X'):\ {}^*\hspace*{-1pt}(XX')\to {}^*\hspace*{-2pt}X'{}^*\hspace*{-2pt}X$.
$(\un{1}, \un{1}, \un{1}, \un{1})$ is an adjunction, so we can put ${}^*\un{1}=\un{1}$, and define 
$\varphi_0=\un{1}:\ \un{1}\to \un{1}$. 
$({}^*\hspace*{-1pt}(-),\varphi_0,\varphi_2):\ \Cc\to \Cc^{\rm oprev}$ is a strong monoidal functor.\\
Let $C$ be a coalgebra in $\Cc$, and assume that we have an adjunction $C\dashv A$. Then $A$ is an algebra, with structure maps
\begin{equation}\eqlabel{dualalgstrofcoalg}
m=Ad^2\circ A\Delta AA \circ bAA:\ AA\to A~~{\rm and}~~\eta= A\varepsilon \circ b:\ \un{1}\to A.
\end{equation}
In this situation $C$ is a right $A$-module (the definition of an $A$-module is given below), with structure map
\begin{equation}\eqlabel{modstrcoFrob}
\mu=Cd\circ \Delta A.
\end{equation}
In a similar way, if a coalgebra $C$ has a left adjoint $A$, then $A$ is an algebra, and $C$ is a left $A$-module.

\subsubsection*{Module categories}
Let $\Cc$ be a monoidal category. A right $\Cc$-category 
is a quadruple $(\Dc , \diamond, \Psi, r)$, where $\Dc$ is a category, 
$\diamond:\ \Dc\times \Cc\ra \Dc$ is a functor, and 
$\Psi:\ \diamond\circ (\diamond\times\Id)\ra \diamond\circ(\Id\times\ot)$ 
and ${\bf r}:\ \diamond\circ ({\rm Id}\times \un{1})\ra {\rm Id}$ are 
natural isomorphisms such that the diagrams
\[
\footnotesize{
\xymatrix{
((M  X)  Y)  Z \ar[d]_-{\Psi_{M, X, Y}  {Z}}
                                           \ar[rr]^-{\Psi_{M  X, Y, Z}}
                                           && ~~(M  X)  (Y  Z)\ar[rr]^{\Psi_{M,X,YZ}}&&M(X(YZ)) \\
(M (X  Y))  Z ~~\ar[rr]^-{\Psi_{M, X  Y, Z}}&&
                                     M ((X  Y)  Z)\ar[rru]_-{~~{M}  a_{X, Y, Z}}&&
                                     }
                                     \hspace*{4mm}{\rm and}\hspace*{4mm}
\xymatrix{
({M}  \un{1})  X \ar[rr]^-{\Psi_{{M}, \un{1}, X}} 
                                  \ar[drr]_-{{\bf r}_{M}  {X}} 
                                   &&{M}  (\un{1}  X)\ar[d]^-{{{M}}  l_X}\\
                                    &&X     
}}
\]
commute, for all $M\in \Dc$ and $X, Y, Z\in \Cc$. Obviously $\Cc$ is itself a right
$\Cc$-category. As before, we deleted the diamond product symbols, and wrote $MX=M\diamond X$,
for $M\in \Dc$ and $X\in \Cc$. The above mentioned coherence
theorem for monoidal categories can be extended to $\Cc$-categories, enabling us to assume throughout
that $\Psi$ and ${\bf r}$ are natural identities. In the literature,
$\Cc$-categories are also named module categories.\\
Let $\Dc$ be a right $\Cc$-category, and consider an algebra $A$ in $\Cc$.
A right $A$-module in $\Dc$ is a pair $M=(M,\mu)$, with $M\in \Dc$ and $\mu:\ M A\to M$ satisfying
$\mu\circ M\eta =M$ and $\mu\circ \mu A= \mu\circ  M m$.
A morphism $f:\ M\to N$ between two right $A$-modules $M$ and $N$ in $\Dc$ is called right
$A$-linear if $f\circ \mu=\mu\circ f A$. 
$\Dc_A$ will be the category of right $A$-modules and right 
$A$-linear morphisms in $\Dc$. 
In a similar way, we can define left $A$-modules $N=(N,\nu)$ in a left $\Cc$-category $\Ec$ and the category
${}_{A}\Ec$.  We will typically use the notation $\mu$ for a right action and $\nu$ for a
left action. The next step is to introduce two-sided $\Cc$-categories,
and two-sided $A$-bimodules in a two-sided $\Cc$-category. We leave it to the reader to formulate
the precise definitions.\\
We can also define the notions of a right $C$-comodule $(M,\rho)$ in a right $\Cc$-category $\Dc$, and right
$C$-colinearity of a morphism between two right $C$-comodules in $\Dc$.
The category of right comodules and right $C$-colinear morphisms in $\Dc$ will be denoted as $\Dc^C$. 
We will use the following diagrammatic notation for actions and coactions:
$$\mu= \gbeg{2}{3}
\got{1}{M}\got{1}{A}\gnl
\grm\gnl
\gob{1}{M}
\gend~~~
,~~~\nu=\gbeg{2}{3}
\got{1}{A}\got{1}{N}\gnl
\glm\gnl
\gob{3}{N}
\gend~~~
,~~~
\rho=
{\footnotesize
\gbeg{2}{3}
\got{1}{M}\gnl
\grcm\gnl
\gob{1}{M}\gob{1}{C}
\gend}.
$$

\subsection{The category of bimodules}\selabel{bimod}
The results in this Subsection will be needed in Sections \ref{se:5} and \ref{se:6}. The results are well-known,
see for example \cite{par}, \cite{sch} or \cite{bc3}. What follows is an original reformulation, which is why we decided
to keep the details. 
Let $\Cc$ be a (strict) monoidal category with coequalizers. 
Recall that $X\in \Cc$ is called left coflat if the functor $- X:\ \Cc\to \Cc$ preserves coequalizers. Let $A$ be an algebra in $\Cc$.
For $X\in \Cc_A$ and $Y\in {}_A\Cc$, $(X\ot_A Y,q)$ is the coequalizer of the parallel morphisms
$\mu Y,~X\nu:\ XAY\to XY$:
$$
\xymatrix{
XAY\ar@<-.5ex>[rr]_(.54){\mu Y} 
\ar@<.5ex>[rr]^(.55){X\nu}&&
XY\ar[r]^(.40){q} &X\ot_AY .
}
$$
We compactify our notation by writing $X\ot_A Y=X\ota Y$. Now let $f:\ X\to X'$ in $\Cc_A$ and $g:\ Y\to Y'$ in ${}_A\Cc$.
The universal property of coequalizers tells us that there is a unique $f\ot_A g=f\ota g$ in $\Cc$ such that \equref{tensor}
commutes.
\begin{equation}\eqlabel{tensor}
\xymatrix{
XAY\ar[d]^{fAg}\ar@<-.5ex>[rr]_(.54){\mu Y} 
\ar@<.5ex>[rr]^(.55){X\nu}&&
XY\ar[d]^{fg}\ar[r]^{q} &X\ota Y\ar@{.>}[d]^{\exists ! f\ota g}\\
X'AY'\ar@<-.5ex>[rr]_(.54){\mu Y'} 
\ar@<.5ex>[rr]^(.55){X'\nu}&&
X'Y'\ar[r]^{q} &X'\ota Y'}
\end{equation}

\begin{proposition}\prlabel{bim1}
Let $X\in \Cc_A$ and $M\in \Cc$. Then $(AM, mM)\in {}_A\Cc$, and
$$
\xymatrix{
XAAM\ar@<-.5ex>[rr]_(.54){\mu AM} 
\ar@<.5ex>[rr]^(.55){XmM}&&
XAM\ar[r]^{\mu M} &XM 
}
$$
is a coequalizer in $\Cc$. If $M\in \Cc_A$ (resp. $X\in {}_A\Cc_A$), then this is also a coequalizer in $\Cc_A$
(resp. ${}_A\Cc$).
\end{proposition}

\begin{proof}
Let $f:\ XAM\to P$ be such that $f\circ \mu AM= f\circ XmM$. We have to prove the existence and uniqueness
of $g:\ XM\to P$ such that $f=g\circ \mu M$.\\
If $g$ exists, then it is unique since
\begin{equation}\eqlabel{bim1.1}
f\circ X\eta M= g\circ \mu M \circ X\eta M=g.
\end{equation}
$g= f\circ X\eta M$ is such that 
$g\circ \mu M= f\circ X\eta M\circ \mu M=f\circ \mu A M\circ XA\eta M=
f\circ XmM\circ XA\eta M=f$. Finally, let $M\in \Cc_A$, and assume that $f$ is right $A$-linear. The morphism
$g$ defined by \equref{bim1.1} is also right $A$-linear, and this shows that $(XM,\mu M)$ is also a coequalizer in
$\Cc_A$. Similar arguments hold in the case where $X\in {}_A\Cc_A$.
\end{proof}

It follows from \prref{bim1} that we have a unique isomorphism $\Upsilon:\ X\ota (AM)\to XM$ such that
$\Upsilon\circ q=\mu M$, and $\Upsilon^{-1}=q\circ X\eta M$. Otherwise stated, there is a unique
isomorphism of coequalizers $(X\ota (AM),q)\cong (XM, \mu M)$. Now coequalizers are defined only up to
isomorphisms, so we can go one step further, and declare $(X\ota (AM),q)= (XM, \mu M)$. This identification
will also be useful at the level of morphisms. Before we explain this, we state the following Lemma, which is
a well-known and basic fact.

\begin{lemma}\lelabel{bim1a}
For $M\in \Cc$ and $Y\in {}_A\Cc$, we have an isomorphism $\alpha: {}_A\Cc(AM,Y)\to \Cc(M,Y)$, given by
the formulas
$$\alpha(\un{f})=\un{f}\circ \eta N~~~;~~~\alpha^{-1}(f)=\nu_Y\circ Af.$$
\end{lemma}

Now take $f:\ X\to X'$ in $\Cc_A$, $M,M'\in \Cc$ and $\un{g}:\ AM\to AM'$ in ${}_A\Cc$. Let $g=\alpha(\un{g}):
M\to AM'$. Making the identification $X\ota (AM)=XM$ and $X'\ota (AM')=X'M'$, we have $f\ota \un{g}:\
XM\to X'M'$. According to \equref{tensor}, $f\ota \un{g}$ is determined by the commutativity of the diagram
$$\xymatrix{
XAM\ar[d]_{f\un{g}}\ar[rr]^{\mu M} && XM\ar[d]^{f\ota \un{g}}\\
X'AM'\ar[rr]^{\mu M'} &&X'M'}$$
It follows from \equref{bim1.1} that
\begin{equation}\eqlabel{bim1a.1}
f\ota \un{g}=\mu M' \circ f\un{g}\circ X\eta M=\mu M' \circ fg.
\end{equation}

\begin{definition}\delabel{bim2}
Let $A$ be an algebra in $\Cc$. $Y\in {}_A\Cc$ is called robust as a left $A$-module if
$$
\xymatrix{
MXAY\ar@<-.5ex>[rr]_(.54){M\mu Y} 
\ar@<.5ex>[rr]^(.55){MX\nu}&&
MXY\ar[r]^(.41){Mq } &M(X\ota Y) 
}
$$
is a coequalizer in $\Cc$, for all $M\in \Cc$ and $X\in \Cc_A$.
\end{definition}

This definition can be restated as follows: the universal property of coequalizers implies the existence
of a unique $\theta:\ (MX)\ota Y\to M(X\ota Y)$ such that $\theta\circ q=Mq$. $Y$ is robust if and only if $\theta$ is an isomorphism
for all $X$ and $M$.

\begin{proposition}\prlabel{bim3}
For all $N\in \Cc$, $AN\in {}_A\Cc$ is robust.
\end{proposition}

\begin{proof}
Take the coequalizer from \prref{bim1}, with $X$ replaced by $NX$. This is precisely the coequalizer in
\deref{bim2}.
\end{proof}

\begin{proposition}\prlabel{bim4}
\begin{enumerate}
\item Let $X\in \Cc_A$ and $Y\in {}_A\Cc_A$. If $A$ is left $A$-coflat, then $X\ota Y\in \Cc_A$,
and $(X\ota Y,q)$ is also a coequalizer in $\Cc_A$.
\item
Let $X\in {}_A\Cc_A$ and $Y\in {}_A\Cc_A$. If $Y$ is robust as a left $A$-module, then $X\ota Y\in{}_A\Cc$,
and $(X\ota Y,q)$ is also a coequalizer in ${}_A\Cc$.
\item 
If both $X$ and $Y$ are $A$-bimodules, $A$ is left coflat and $Y$ is left $A$-robust, then $X\ota Y\in
{}_A\Cc_A$, and $(X\ota Y,q)$ is also a coequalizer in ${}_A\Cc_A$.
\end{enumerate}
\end{proposition}

\begin{proof}
(1) Consider the diagram
\begin{equation}\eqlabel{bim4.1}
\xymatrix{
XAYA\ar[d]^{XA\mu}\ar@<-.5ex>[rr]_(.54){\mu YA} 
\ar@<.5ex>[rr]^(.55){X\nu A}&&
XYA\ar[d]^{X\mu }\ar[r]^(.42){qA} &(X\ota Y)A\ar@{.>}[d]^{\exists ! \mu}\\
XAY\ar@<-.5ex>[rr]_(.54){\mu Y} 
\ar@<.5ex>[rr]^(.55){X\nu}&&
XY\ar[r]^{q} &X\ota Y}
\end{equation}
The top row is a coequalizer since $A$ is left coflat, and this implies the
existence of $\mu$.  $\mu$ satisfies the unit property.
The diagram
$$\xymatrix{
XY\ar[d]_{XY\eta}\ar[rr]^{q}&&X\ota Y\ar[d]^{(X\ota Y)\eta}\\
XYA\ar[d]_{X\mu}\ar[rr]^{qA}&&(X\ota Y)A\ar[d]^{\mu}\\
XY\ar[rr]^{q}&&X\ota Y}$$
commutes. $X\mu\circ XY\eta=XY$, the identity, and it follows from the uniqueness
in the universal property of the coequalizer that $\mu\circ (X\ota Y)\eta=X\ota Y$.
We omit the proof of the associativity of $\mu$, since it is similar to the proof of
the compatibility of $\mu$ and $\nu$ in the third part of the proof.\\
Let $f:\ XY\to P$ be a morphism in $\Cc_A$ such that $f\circ \mu Y=f\circ X\nu$.
We know that there exists a unique $g:\ X\ota Y\to P$ such that $g\circ f=q$.
In follows that $(X\ota A,q)$ is a coequalizer in $\Cc_A$ if we can show that
$g$ is right $A$-linear. We first compute that
$$g\circ \mu \circ qA\equal{\equref{bim4.1}}g\circ q\circ X\mu= f\circ X\mu\equal{(*)}
\mu\circ fA=\mu\circ gA\circ qA.$$
At $(*)$, we used the fact that $f$ is right $A$-linear. $((X\ota Y)A,qA)$ is a coequalizer
since $A$ is left coflat, ant it follows that $g\circ \mu=\mu\circ gA$, which is precisely
what we need.\\

(2) If $Y$ is left $A$-robust, then the top row in the diagram
\begin{equation}\eqlabel{bim4.2}
\xymatrix{
AXAY\ar[d]^{\nu AY}\ar@<-.5ex>[rr]_(.54){A\mu Y} 
\ar@<.5ex>[rr]^(.55){AX\nu }&&
AXY\ar[d]^{\nu Y }\ar[r]^(.41){Aq} &A(X\ota Y)\ar@{.>}[d]^{\exists ! \nu}\\
XAY\ar@<-.5ex>[rr]_(.54){\mu Y} 
\ar@<.5ex>[rr]^(.55){X\nu}&&
XY\ar[r]^{q} &X\ota Y}
\end{equation}
is a coequalizer, and the universal property brings the left action $\nu$ on
$X\ota Y$. The rest of the proof of part (2) is left to the reader.\\

(3) Now we assume that both $X$ and $Y$ are bimodules. We show that the actions
$\mu$ and $\nu$ on $X\ota Y$ are compatible. To this end, consider the cubic diagram
$$\xymatrix{
AXYA\ar[rr]^{AqA}\ar[dd]^{\nu YA}\ar[dr]^{AX\mu}&&A(X\ota Y)A\ar'[d][dd]^{\nu A}
\ar[dr]^{A\mu}&\\
&AXY\ar[dd]^(.7){\nu Y}\ar[rr]^(.35){Aq}
&&A(X\ota Y)\ar[dd]^{\nu}\\
XYA\ar'[r][rr]^(.3){qA}\ar[dr]^{X\mu}&&(X\ota Y)A\ar[dr]^{\mu}&\\
&XY\ar[rr]^{q}
&&X\ota Y}$$
Commutativity of the top and bottom faces follows from the definition of $\mu$, and commutativity
of the front and back faces follows from the definition of $\nu$. It is obvious that the left face commutes.
From this we deduce that
$$
\mu\circ \nu A\circ AqA=\mu\circ qA\circ \nu YA=q\circ X\mu \circ \nu Y A
= q\circ \nu Y\circ AX\mu=\nu\circ Aq\circ AX\mu=\nu\circ A\mu\circ AqA.
$$
From the robustness of $Y$, we know that $(A(X\ota Y),Aq)$ is a coequalizer, and from
the left coflatness of $A$ that $(A(X\ota Y)A,AqA)$ is a coequalizer. It then follows that
$\mu\circ \nu A= \nu\circ A\mu$, which is the compatibility that we need. The proof of the 
associativity of $\mu$ and $\nu$ follows by similar arguments.
\end{proof}

Let $A$ be a left coflat algebra in $\Cc$, and let ${}_A^{~ !}\Cc_A^{\vbox{}}$ be the full subcategory of
${}_A\Cc_A$ consisting of bimodules that are left coflat as objects in $\Cc$, and robust as left
$A$-modules. Our aim is to show that ${}_A^{~ !}\Cc_A^{\vbox{}}$ is a monoidal category, with tensor product 
$\ot_A$ and unit object $A$.

\begin{lemma}\lelabel{bim5}
Let $A$ be left coflat. If $X,Y\in {}_A^{~ !}\Cc_A^{\vbox{}}$, then $X\ota Y$ is left coflat.
\end{lemma}

\begin{proof}
It is easy to show that the tensor product (in $\Cc$) of two left coflat objects is left coflat.
Let $(P,h)$ be the coequalizer of two parallel morphisms $f,g:\ M\to N$ in $\Cc$. We have to show
that $(P(X\ota Y),h(X\ota Y))$ is the coequalizer of $f(X\ota Y),g(X\ota Y)$. Take
$r:\ N(X\ota Y)\to R$ such that $r\circ f(X\ota Y)=r\circ g(X\ota Y)$ and
consider the diagram
$$\xymatrix{
MXAY \ar@<.5ex>[d]\ar@<-.5ex>[d] 
\ar@<.5ex>[rr]\ar@<-.5ex>[rr]&&
NXAY \ar@<.5ex>[d]\ar@<-.5ex>[d] 
\ar[rr]^{hXAY}&&
PXAY \ar@<.5ex>[d]\ar@<-.5ex>[d]\\
MXY\ar[d]^{Mq}\ar@<.5ex>[rr]\ar@<-.5ex>[rr]&&
NXY\ar[d]^{Nq}\ar[rr]^{hXY}&&
PXY\ar[d]^{Pq}\ar@{.>}@/^3pc/[dd]^{\exists  s}\\
M(X\ota Y)\ar@<.5ex>[rr]\ar@<-.5ex>[rr]&&
N(X\ota Y)\ar[rr]^{h(X\ota Y)}\ar[rrd]_r&&
P(X\ota Y)\ar@{.>}[d]_{\exists  t}\\
&&&&R}$$
The first two rows are coequalizers since $XY$ and $XAY$ are left coflat, and the three columns are
coequalizers since $Y$ is left $A$-robust. The rectangles in the diagram commute, and we easily compute
that
$$r\circ Nq\circ fXY=r\circ f(X\ota Y)\circ Mq =r\circ g(X\ota Y)\circ Mq=r\circ Nq\circ gXY,$$
and it follows that there exists  $s:\ PXY\to R$ such that $s\circ hXY=r\circ Nq$.
Then we compute that
$$
s\circ P\nu Y\circ hXAY=s\circ hXY\circ N\nu Y=r\circ Nq\circ N\nu Y
=r\circ Nq\circ NX\mu=s\circ PX\mu\circ hXAY,$$
hence $s\circ P\nu Y=s\circ PX\mu$, since $hXAY$ is an epimorphism. This implies the existence of 
 $t:\ P(X\ota Y)\to R$ such that $t\circ Pq=s$. It is now easy to see that
$$t\circ h(X\ota Y)\circ Nq=t\circ Pq\circ hXY= s\circ hXY=r\circ Nq,$$
hence  $t\circ h(X\ota Y)=r$, as needed. The uniqueness can be easily obtained as follows: if 
$t':\ P(X\ota Y)\to R$ is such that $t'\circ h(X\ota Y)=r$,
then $t'\circ Pq\circ hXY=r\circ Nq=t\circ Pq\circ hXY$, and $t'=t$.
\end{proof}

Recall that coequalizers are colimits, see for example \cite[III.3]{mclane}; in particular, the tensor product $X\ota Y$ of $X\in \Cc_A$ and
$Y\in {}_A\Cc$ is a colimit: consider the category $J$ with two objects labelled $xay$ and $xy$,
and two non-identity arrows $my,~xn:\ xay\to xy$, and let $F:\ J\to \Cc$ be the following functor:
$$
F(xay)=XAY,~~~F(xy)=XY,~~~F(my)=\mu Y,~~~F(xn)=X\nu.
$$
Cones from $F$ to the vertex $P$ in $\Cc$ correspond to morphisms $f:\ XY\to P$ such that $f\circ \mu Y=f\circ X\nu$,
and the colimit $\Colim F=(X\ota Y,q)$ consists of an object $X\ota Y\in \Cc$ and a universal cone $q$ from $F$ to
$X\ota Y$.\\
We now generalize this construction. Let $J_2$ be the category with four objects
$xayaz$, $xyaz$, $xayz$ and $xyz$, and morphisms
$$\xymatrix{xayaz
\ar@<.5ex>[d]_{xamz~}\ar@<-.5ex>[d]^{~xayn}
\ar@<.5ex>[rr]^{myaz}\ar@<-.5ex>[rr]_{xnaz}&&
xyaz
\ar@<.5ex>[d]_{xmz~}\ar@<-.5ex>[d]^{~xyn}\\
xayz
\ar@<.5ex>[rr]^{myz}\ar@<-.5ex>[rr]_{xnz}&&
xyz}$$
and their compositions, subject to the relations
$$\begin{array}{ccc}
xmz\circ myaz=myz\circ xamz=mmz&;&
xyn\circ xmaz=myz\circ xayn=myn;\\
xmz\circ xnaz=xnz\circ xamz= xmnz&;&
xyn\circ xnaz=xnz\circ xayn=xnn.
\end{array}$$
Consider $X\in \Cc_A$, $Y\in {}_A\Cc_A$ and $Z\in {}_A\Cc$.
$F_2:\ J_2\to \Cc$ is defined in the following way:
$F_2(xayaz)=XAYAZ$, $F_2(xyaz)=XYAZ$, $\cdots$, $F_2(xmaz)=X\mu AZ$ etc.
We can also consider the full subcategory $J'_2$ of $J_2$, with objects $xyaz$, $xayz$ and $xyz$,
and the restriction $F'_2$ of $F_2$ to $J'_2$. It is easy to establish that cones from $F_2$ to $P\in  \Cc$
correspond bijectively to cones from $F'_2$ to $P$, so that $F_2$ and $F'_2$ have the same colimit. We now define
$$\Colim F'_2=(X\ota Y\ota Z, q_2).$$

\begin{proposition}\prlabel{bim6}
Let $A$ be a coflat algebra in $\Cc$, and consider $X\in \Cc_A$ and $Y,Z\in {}_A^{~ !}\Cc_A^{\vbox{}}$. Then we have isomorphisms
of cones
$$(X\ota Y\ota Z, q_2)\cong (X\bullet (Y\bullet Z), q\circ Xq)
\cong ((X\ota Y)\ota Z,q\circ qZ).$$
If $X$ is an $A$-bimodule, then
$F_2$ and $F'_2$ corestrict to functors with values in ${}_A\Cc_A$, and the above cones are also the colimits of
these corestrictions. 
\end{proposition}

\begin{proof}
We will show that $(X\bullet (Y\bullet Z), q\circ Xq)$ satisfies the universal property of cones.
Assume that $f:\ XYZ\to P$ in $\Cc$ is such that
\begin{equation}\eqlabel{bim6}
f\circ \mu YZ=f\circ X\nu Z~~{\rm and}~~f\circ x\mu Z=f\circ XY\nu.
\end{equation}
Consider the diagram
$$\xymatrix{
XAYAZ
\ar@<.5ex>[d]\ar@<-.5ex>[d] 
\ar@<.5ex>[rr]\ar@<-.5ex>[rr]&&
XAYAZ
\ar@<.5ex>[d]_{\mu YZ~}\ar@<-.5ex>[d]^{~X\nu Z}
\ar[rr]^{XAq}&&
XA(Y\ota Z)
\ar@<.5ex>[d]_{\mu(Y\ota Z)~}\ar@<-.5ex>[d]^{~X\nu}\\
XYAZ\ar@<.5ex>[rr]\ar@<-.5ex>[rr]&&
XYZ\ar[d]_f\ar[rr]^{Xq}&&
X(Y\ota Z)\ar[d]_q\ar@{.>}[dll]_{\exists f_1}\\
&&P&&X\ota(Y\ota Z)\ar@{.>}[ll]_{\exists f_2}}$$
The two top rows are coequalizers since $Z$ is robust as a left $A$-module. The second equation in
\equref{bim6} implies the existence of $f_1:\ X (Y\ota Z)\to P$ such that $f_1\circ Xq=f$.\\
The two squares in the top right corner of the diagram commute. The commutativity of the one on the left
is obvious, and the commutativity of the one on the right 
is a consequence of the definition of the left action $\nu$ on $Y\ota Z$, see \equref{bim4.2}. 
We now easily find that
$$f_1\circ X\nu \circ XAq=f_1\circ Xq\circ X\nu Z=f\circ X\nu Z
\equal{\equref{bim6}} f\circ \mu YZ=f_1\circ Xq\circ \mu YZ=f_1\circ \mu(Y\ota Z)\circ XAq,$$
hence $f_1\circ X\nu= f_1\circ \mu(Y\ota Z)$, so there exists $f_2:\ X\bullet (Y\bullet Z)\to P$
such that $f_2\circ q=f_1$, and $f_2\circ q\circ Xq=f$, as needed. The uniqueness of $f_2$ follows
from the fact that $q$ and $Xq$ are epimorphisms. If $f$ is a morphism in ${}_A\Cc_A$, then
it follows from \prref{bim4} that $f_1$ and $f_2$ are also in ${}_A\Cc_A$, and this shows that
$(X\bullet (Y\bullet Z), q\circ Xq)$ is the colimit of the corestriction of $F'_2$ to ${}_A\Cc_A$.\\
Similar arguments show that $((X\ota Y)\ota Z,q\circ qZ)$ is a colimit of $F'_2$.
\end{proof}

Take $X,Y,Z\in {}_A^{~ !}\Cc_A^{\vbox{}}$.
It follows from the universal property of colimits that there exists a unique isomorphism
$$\alpha=\alpha_{X,Y,Z}:\ X\ota (Y\ota Z)\to (X\ota Y)\ota Z$$
in ${}_A\Cc_A$
such that the diagram
\begin{equation}\eqlabel{bim7.1}
\xymatrix{
XYZ\ar[drr]_{qZ}\ar[rr]^{Xq}&& X (Y\ota Z)\ar[rr]^q&& X\ota (Y\ota Z)\ar[d]^{\alpha}\\
&&(X\ota Y) Z \ar[rr]^q&& (X\ota Y)\ota Z}
\end{equation}
commutes. The diagram
\begin{equation}\eqlabel{bim7.2}
\xymatrix{
XYZ\ar[drr]_q\ar[rr]^{Xq}&&X(Y\ota Z)\ar[d]^{\theta^{-1}}\ar[rr]^q&&X\ota(Y\ota Z)\ar[d]^{\alpha}\\
&&(XY)\ota Z\ar[rr]^{q\ota Z}&&(X\ota Y)\ota Z}
\end{equation}
commutes. Indeed, the commutativity of the pentangle follows from \equref{bim7.1} combined with
\equref{tensor}; the triangle commutes: this is the definition of $\theta$. Then we compute that
$$\alpha\circ q\circ \theta\circ q= \alpha\circ q\circ Xq=q\ota Z\circ q,$$
hence $\alpha\circ q\circ \theta=Z\circ q$, so the rectangle commutes.

\begin{proposition}\prlabel{bim7}
If $A$ is left coflat, then ${}_A^{~ !}\Cc_A^{\vbox{}}$ is closed under the tensor product over $A$.
\end{proposition} 

\begin{proof}
Take $Y,Z\in {}_A^{~ !}\Cc_A^{\vbox{}}$. We know from \leref{bim5} that $Y\ota Z$ is left coflat.
We are left to show that $Y\ota Z$ is robust as a left $A$-module. Take $M\in \Cc$ and $X\in \Cc_A$
and consider the diagram
\begin{equation}\eqlabel{bim7.3}
\xymatrix{
&&MX(Y\ota Z)\ar[d]^{\theta^{-1}}\ar[rr]^q
&&(MX)\ota(Y\ota Z)\ar[d]^{\alpha}\\
&&(MXY)\ota Z\ar[rrd]^{(Mq)\ota Z}\ar[rr]^{q\ota Z}
&&((MX)\ota Y)\ota Z\ar[d]^{\theta\ota Z}\\
MXYZ
\ar[uurr]^{MXq}\ar[urr]^q\ar[rr]^{MqZ}\ar[drr]^{Mq}\ar[ddrr]^{MXq}
&&M(X\ota Y)Z\ar[rr]^q\ar[drr]^{Mq}
&&(M(X\ota Y))Z\ar[d]^{\theta}\\
&&M((XY)\ota Z)\ar[d]^{M\theta}\ar[rr]^{M(q\ota Z)}
&&M((X\ota Y)\ota Z)\ar[d]^{M\alpha^{-1}}\\
&&MX(Y\ota Z)\ar[rr]^{Mq}
&&M(X\ota (Y\ota Z))}
\end{equation}
Commutativity of the top and bottom triangles and rectangles follows from \equref{bim7.2}.
The commutativity of the two remaining triangles follows from the definition of $\theta$,
and the commutativity of the two remaining quadrangles follows from \equref{tensor}.
We conclude that the whole diagram commutes. Now let
$$\Theta= M\alpha^{-1}\circ \theta\circ \theta\ota Z\circ \alpha:\ (MX)\ota(Y\ota Z)\to M(X\ota (Y\ota Z)).$$
$\Theta$ is an isomorphism, and $\Theta\circ q\circ MXq=Mq\circ MXq$, so that $\Theta\circ q=Mq$.
It follows from the (reformulation of) \deref{bim2} that $Y\ota Z$ is robust as a left $A$-module.
\end{proof}

\begin{theorem}\thlabel{bim8}
Let $A$ be a left coflat algebra in $\Cc$. Then we have a monoidal category
$({}_A^{~ !}\Cc_A^{\vbox{}}, \ot_A=\ota, A, \alpha,\lambda,\rho)$. The category $\Cc_A$ is a right
$\ACA$-category.
\end{theorem}

\begin{proof}
We have shown in \prref{bim7} that the tensor product over $A$ of two objects in ${}_A^{~ !}\Cc_A^{\vbox{}}$
is again in ${}_A^{~ !}\Cc_A^{\vbox{}}$. The associativity constraint $\alpha$ was defined as an application
of \prref{bim6}.\\
The unit constraint follows as an application of \prref{bim1}.
 $(X,\mu)$ and $(X\ota A,q)$ are both coequalizers in $\Cc$ (and in ${}_A\Cc_A$) of $Xm,\mu A:\ XAA\to XA$,
 so there exists a unique isomorphism $\rho_X:\ X\ota A\to X$
in ${}_A\Cc_A$ such that $\rho_X\circ q=\mu$, with inverse $\rho^{-1}_X=q\circ X\eta$. In a similar way, we have
a unique isomorphism $\lambda_X:\ A\ota X\to X$ in ${}_A\Cc_A$ such that $\lambda_X\circ q=\nu$, with inverse 
$\lambda^{-1}_X=q\circ \eta X$. We are left to show that the coherence conditions are satisfied.\\
Take
$X,Y,Z,T\in {}_A^{~ !}\Cc_A^{\vbox{}}$. We have to show that the following diagrams commute.
\begin{equation}\eqlabel{bim7.4}
\xymatrix{
X\ota (A\ota Z)\ar[rr]^{\alpha_{X,A,Z}}\ar[dr]_{X\ota \lambda_Z}
&& (X\ota A)\ota Z\ar[dl]^{\rho_X\ota Z}\\
&X\ota Z&}
\end{equation}
\begin{equation}\eqlabel{bim7.5}
\xymatrix{
&(X\ota Y)\ota (Z\ota T)\ar[dr]^{\alpha_{X\ota Y,Z,T}}&\\
X\ota (Y\ota (Z\ota T))\ar[ur]^{\alpha_{X,Y,Z\ota T}}\ar[d]_{X\ota \alpha_{Y,Z,T}}&&
((X\ota Y)\ota Z)\ota T\\
X\ota((Y\ota Z)\ota T)\ar[rr]^{\alpha_{X,Y\ota Z,T}}&&(X\ota (Y\ota Z))\ota T\ar[u]_{\alpha_{X,Y,Z}}\ota T}
\end{equation}
$\alpha_{X,A,Z}$ is the unique morphism that makes the diagram \equref{bim7.1} commutative. If we can
show that $(\rho_X\ota Z)^{-1}\circ X\ota \lambda_Z$ has the same property, then it follows that
\equref{bim7.4} commutes. This means that we have to show that the diagram
$$
\xymatrix{
XAZ\ar[d]_{qZ}\ar[r]^{Xq}&X(A\ota Z)\ar[r]^q&X\ota (A\ota Y)\ar[d]^{X\lambda_Y}\\
(X\ota A)Z\ar[r]^q&(X\ota A)\ota Z\ar[r]^{\rho_X\ota Y}& X\ota Y}$$
commutes. This is an easy computation:
$$X\ota \lambda_Y\circ q\circ Xq\equal{\equref{tensor}}q\circ X\lambda_Y\circ Xq=q\circ X\nu
=q\circ \mu Y= q\circ \rho_XY\circ qY\equal{\equref{tensor}}\rho_X\ota Y\circ q\circ qZ.$$
Now consider the category $J'_3$, consisting of four objects and six morphisms that are not identities:
$$\xymatrix{
&&xyazt
\ar@<.5ex>[d]_{xmzt~}\ar@<-.5ex>[d]^{~xynt}&&\\
xayzt\ar@<.5ex>[rr]^{myzt}\ar@<-.5ex>[rr]_{xnzt}&&
xyzt
&&xyzat
\ar@<.5ex>[ll]^{xyzn}\ar@<-.5ex>[ll]_{xymt}}$$
We define $F'_3:\ J'_3\to \Cc$ in the obvious way: $F'_3(xyazt)=XYAZT$, $F'_3(xmzt)=X\mu ZT$, etc.
The fourfold tensor product is defined as the colimit of $F'_3$:
$\Colim F'_3= (X\ota Y\ota Z\ota T,q_3)$. Proceeding as in \prref{bim6}, we can show that
$\bigl(X\ota (Y\ota (Z\ota T)),q\circ Xq\circ XYq\bigr)$,
$\bigl((X\ota Y)\ota (Z\ota T), q\circ qq\bigr)$,
$\bigl(((X\ota Y)\ota Z)\ota T, q\circ qT\circ qZT\bigr)$,
$\bigl(X\ota((Y\ota Z)\ota T),q\circ Xq\circ XqT\bigr)$ and
$\bigl((X\ota (Y\ota Z))\ota T,q\circ qT\circ XqT\bigr)$
are all colimits of $F'_3$ (and of the corestriction of $F'_3$ to ${}_A\Cc_A$). This means that these five cones
are isomorphic. For example, the isomorphism between the first two cones is the unique morphism that makes
the diagram
$$\xymatrix{
XYZT\ar[rr]^{XYq}\ar[drr]_{XYq}&&
XY(Z\ota T)\ar[d]^=\ar[rr]^{Xq}&&
X(Y\ota(Z\ota T))\ar[rr]^q&&
X\ota(Y\ota(Z\ota T))\ar@{.>}[d]^{\exists !}\\
&&XY(Z\ota T)\ar[rr]^{q(Z\ota T)}&&
(X\ota Y)(Z\ota T)\ar[rr]^q&&
(X\ota Y)\ota(Z\ota T)}$$
commutative. Here we used the equality $qq=q(Z\ota T)\circ XYq$. In view of \equref{bim7.1},
this morphism is $\alpha_{X,Y,Z\ota T}$. In a similar way, we can prove that the maps in the diagram \equref{bim7.5}
establish isomorphisms between the 5 coequalizers above. Therefore the two compositions in the diagram
also establish isomorphisms between coequalizers, hence they are equal, since these isomorphisms
are unique. This tells us that \equref{bim7.5} commutes.
\end{proof}

A coalgebra $C$ in $\ACA$ is called an $A$-coring. The category of right $C$-comodules 
and right $C$-colinear morphisms in $\Cc_A$ is denoted by $\Cc^C$. 

\section{Entwined modules over cowreaths}\selabel{2}
\setcounter{equation}{0}
A (strict) monoidal category $\Cc$ can be viewed as a 2-category with a single 0-cell, hence
we can consider the 2-categories ${\rm Mnd}(\Cc)$ \cite{street} and ${\rm EM}(\Cc)$ \cite{LackRoss}.
These have the same 0-cells and 1-cells, but are different at the level of 2-cells. The 0-cells are
algebras (or monads) in $\Cc$. Fix an algebra $A$ in $\Cc$ and consider the endomorphism categories
$$
\Tc(\Cc)_A=\Tc_A={\rm Mnd}(\Cc)(A,A)~~{\rm and}~~\Tc(\Cc)_A^\#=\Tc_A^\#={\rm EM}(\Cc)(A,A).
$$
The notation $\Tc(\Cc)_A$ is taken from \cite{tambara}, where $\Tc(\Cc)_A$ appears in a different
context, and where it is called the category of right transfer morphisms through $A$. 
$\Tc_A$ and $\Tc_A^\#$ are (strict) monoidal categories. A monad in ${\rm Mnd}(\Cc)$ 
is called a distributive law \cite[Sec. 6]{street} or a factorization structure. A comonad in ${\rm Mnd}(\Cc)$
is called a mixed distributive law or an entwining structure \cite{brmaj}. A monad in ${\rm EM}(\Cc)$
is called a wreath in $\Cc$ \cite{LackRoss}, an alternative name suggested in \cite{LackRoss}
is generalized distributive law. A comonad in ${\rm EM}(\Cc)$
was called in \cite{bc4} a cowreath in $\Cc$, or a mixed wreath in \cite{StreetMW}; we can also refer to it as 
a generalized entwining structure. 

A cowreath in $\Cc$ consists of  an algebra $A$ in $\Cc$ and 
a coalgebra in $\Tc_A^\#$. In a similar way, a wreath in $\Cc$ consists of 
an algebra $A$ in $\Cc$ and an algebra in $\Tc_A^\#$. For later use, we spell out the 
explicit definition of a cowreath.

\subsection{The monoidal categories $\Tc_A$ and $\Tc_A^\#$}\selabel{2.1}
Let $A$ be an algebra in $\Cc$. A (right) transfer morphism through $A$ is a pair
$X=(X,\psi)$, with  $X\in \Cc$ and $\psi:\ X A\ra A X$  in $\Cc$ such that 
$\psi\circ Xm=Xm\circ A\psi\circ \psi A$ and $\psi\circ X\eta=\eta X$; in diagrammatic notation
\begin{equation}\eqlabel{ta}
\psi=
{\footnotesize{
\gbeg{2}{3}
\got{1}{X}\got{1}{A} \gnl
\gbrc \gnl
\gob{1}{A}\gob{1}{X}
\gend}}\hspace*{5mm}{\rm satisfies}\hspace*{5mm}
(a)~~
{{\footnotesize
\gbeg{3}{5}
\got{1}{X}\got{1}{A}\got{1}{A}\gnl
\gbrc\gcl{1}\gnl
\gcl{1}\gbrc\gnl
\gmu\gcl{1}\gnl
\gob{2}{A}\gob{1}{X}
\gend}} =
{{\footnotesize
\gbeg{3}{5}
\got{1}{X}\got{1}{A}\got{1}{A}\gnl
\gcl{2}\gmu\gnl
\gvac{1}\gcn{1}{1}{2}{1}\gnl
\gbrc\gnl
\gob{1}{A}\gob{1}{X}
\gend}}
\hspace*{5mm}{\rm and}\hspace*{5mm}(b)~~
{{\footnotesize
\gbeg{2}{4}
\got{1}{X}\gnl
\gcl{1}\gu{1}\gnl
\gbrc\gnl
\gob{1}{A}\gob{2}{X}
\gend}}
=
{{\footnotesize
\gbeg{2}{3}
\got{3}{X}\gnl
\gu{1}\gcl{1}\gnl
\gob{1}{A}\gob{1}{X}
\gend}}.
\end{equation}
The categories $\Tc_A$ and $\Tc_A^\#$ coincide at the level of objects; their objects
are right transfer morphisms through $A$. 
A morphism $X\to Y$ in $\Tc_A$ is a morphism $f:\ X\to Y$ in $\Cc$ such that
$\psi \circ fA=A f\circ \psi$. 
A morphism $X\to Y$ in $\Tc_A^\#$ is a morphism $f:\ X\to A Y$ in $\Cc$ such that
\begin{equation}\eqlabel{2.1.0}
mY \circ Af \circ \psi=mY\circ A \psi\circ fA.
\end{equation}

The composition of two morphisms $f:\ X\to Y$ 
and $g:\ Y\to Z$ in $\Tc_A^\#$ is
$g\bullet f= mZ \circ Ag \circ f$. The identity on $(X, \psi)$ is $\eta X$.
The tensor product of $X$ and $Y$ is 
$XY=(X Y, \psi_X\cdot \psi_Y=\psi_X Y\circ X\psi_Y)$. The tensor product of
$f:\ X\to X'$ and $g:\ Y\to Y'$ in $\Tc_A^\#$ is given by the composition 
$mXY\circ A\psi Y\circ fg$. 
The unit object is $(\un{1}, A)$. $\Tc_A$ and $\Tc_A^\#$ are strict monoidal categories,
and we have a strong monoidal functor 
$F:\ \Tc_A\to \Tc_A^\#$, which is the identity on objects, and $F(f)=\eta f$, for
$f:\ X\to Y$ in $\Tc_A$. If a morphism in $\Tc_A^\#$ is of the form $\eta f$,
with $f:\ X\to Y$ in $\Cc$, then $f$ is a morphism in $\Tc_A$.\\

In a similar way we introduce left transfer morphisms through $A$, consisting of pairs $X=(X, \varphi)$, with
$\varphi=\footnotesize{
\gbeg{2}{3}
\got{1}{A}\got{1}{X}\gnl
\gbrcb\gnl
\gob{1}{X}\gob{1}{A}
\gend}:\ A X\to X A$. 
We leave it to the reader to write down the precise definition of 
the categories ${}_A{\cal T}$ and ${}_A^\#{\cal T}$. In fact
${}_A\Tc={\rm Mnd}(\Cc^{\rm op})(A,A)~~{\rm and}~~{}_A^\#\Tc={\rm EM}(\Cc^{\rm op})(A,A)$.

The tensor product in ${}_A{\cal T}$ and ${}_A^\#{\cal T}$ is given by the formula
$XY=(X Y, \varphi_X\cdot \varphi_Y=X \varphi_Y\circ \varphi_X Y)$.\\

\subsection{Cowreaths}\selabel{2.2}
A cowreath (mixed wreath or generalized entwining structure) in $\Cc$ is a triple $(A, X, \psi)$, where $A$ is an algebra in $\Cc$,
and $(X, \psi)$ is a coalgebra in ${\cal T}_A^{\#}$, which is  an
object $(X, \psi)\in \Tc_A^\#$ together with morphisms 
$$
\delta=\footnotesize{
\gbeg{3}{5}
\got{3}{X}\gnl
\gvac{1}\gcl{1}\gnl
\gsbox{3}\gnl
\gcl{1}\gcl{1}\gcl{1}\gnl
\gob{1}{A}\gob{1}{X}\gob{1}{X}
\gend}~: X\ra AXX,\hspace*{1cm}
\epsilon=\footnotesize{
\gbeg{1}{3}
\got{1}{X}\gnl
\gmp{\epsilon}\gnl
\gob{1}{A}
\gend}~~:X\ra A$$
in $\Cc$ such that the following relations hold:
\begin{equation}\eqlabel{pdf}
{\rm (a)}~~{\footnotesize 
\gbeg{4}{8}
\got{3}{X}\got{1}{A}\gnl
\gvac{1}\gcl{1}\gvac{1}\gcl{2}\gnl
\gsbox{3}\gnl
\gcl{1}\gcl{1}\gcl{1}\gcl{1}\gnl
\gcl{1}\gcl{1}\gbrc\gnl
\gcl{1}\gbrc\gcl{1}\gnl
\gmu\gcl{1}\gcl{1}\gnl
\gob{2}{A}\gob{1}{X}\gob{1}{X}
\gend} 
=
{\footnotesize 
\gbeg{4}{7}
\got{1}{X}\got{1}{A}\gnl
\gbrc\gnl
\gcl{1}\gcn{1}{1}{1}{3}\gnl
\gcl{1}\gsbox{3}\gnl
\gcl{1}\gcl{1}\gcl{1}\gcl{1}\gnl
\gmu\gcl{1}\gcl{1}\gnl
\gob{2}{A}\gob{1}{X}\gob{1}{X}
\gend}
~;\hspace*{1cm}{\rm (b)}~~
{\footnotesize
\gbeg{5}{8}
\got{5}{X}\gnl
\gvac{2}\gcl{1}\gnl
\gvac{1}\gsbox{3}\gnl
\gcn{1}{1}{3}{1}\gvac{1}\gcl{1}\gcn{1}{1}{1}{3}\gnl
\gcl{1}\gsbox{3}\gvac{3}\gcl{1}\gnl
\gcl{1}\gcl{1}\gcl{1}\gcl{1}\gcl{1}\gnl
\gmu\gcl{1}\gcl{1}\gcl{1}\gnl
\gob{2}{A}\gob{1}{X}\gob{1}{X}\gob{1}{X}
\gend} 
=
{\footnotesize 
\gbeg{5}{9}
\got{3}{X}\gnl
\gvac{1}\gcl{1}\gnl
\gsbox{3}\gnl
\gcl{1}\gcl{1}\gcn{1}{1}{1}{3}\gnl
\gcl{1}\gcl{1}\gsbox{3}\gnl
\gcl{1}\gcl{1}\gcl{1}\gcl{1}\gcl{1}\gnl
\gcl{1}\gbrc\gcl{1}\gcl{1}\gnl
\gmu\gcl{1}\gcl{1}\gcl{1}\gnl
\gob{2}{A}\gob{1}{X}\gob{1}{X}\gob{1}{X}
\gend}
\end{equation}
$${\rm (c)}~~{\footnotesize
\gbeg{2}{5}
\got{1}{X}\got{1}{A}\gnl
\gbrc\gnl
\gcl{1}\gmp{\epsilon}\gnl
\gmu\gnl
\gob{2}{A}
\gend} 
= 
{\footnotesize
\gbeg{2}{4}
\got{1}{X}\got{1}{A}\gnl
\gmp{\epsilon}\gcl{1}\gnl
\gmu\gnl
\gob{2}{A}
\gend}
~;\hspace*{1cm}{\rm (d)}~~
{\footnotesize 
\gbeg{3}{6}
\got{3}{X}\gnl
\gvac{1}\gcl{1}\gnl
\gsbox{3}\gnl
\gcl{1}\gmp{\epsilon}\gcl{1}\gnl
\gmu\gcl{1}\gnl
\gob{2}{A}\gob{1}{X}
\gend} 
=
{\footnotesize
\gbeg{2}{3}
\got{3}{X}\gnl
\gu{1}\gcl{1}\gnl
\gob{1}{A}\gob{1}{X}
\gend}
~;\hspace*{1cm}{\rm (e)}~~
{\footnotesize 
\gbeg{3}{7}
\got{3}{X}\gnl
\gvac{1}\gcl{1}\gnl
\gsbox{3}\gnl
\gcl{1}\gcl{1}\gmp{\epsilon}\gnl
\gcl{1}\gbrc\gnl
\gmu\gcl{1}\gnl
\gob{2}{A}\gob{1}{X}
\gend} 
=
{\footnotesize  
\gbeg{2}{3}
\got{3}{X}\gnl
\gu{1}\gcl{1}\gnl
\gob{1}{A}\gob{1}{X}
\gend}
\hspace{2mm}.$$ 
Conditions (a) and (c) mean that $\delta$ and $\epsilon$ define morphisms $X\to XX$
and $X\to \un{1}$ in $\Tc_A^\#$. (b) is the coassociativity of the comultiplication $\delta$
and (d) and (e) are the left and right counit property.

\subsection{Entwined modules over cowreaths}\selabel{2.3}
Let $\Dc$ be a right $\Cc$-category, and let $A$ be an algebra in $\Cc$. Then
$\Dc_A$ is a right $\Tc_A$-category, see \cite[Prop. 4.3]{dbbt}. We will
now show that it is also a right $\Tc_A^\#$-category.

\begin{proposition}\prlabel{4.4}
Let $A$ be an algebra in a (strict) monoidal category $\Cc$, and let 
$\Dc$ be a (strict) right $\Cc$-category. Then $\Dc_A$ is a right $\Tc_A^\#$-category.
The tensor product of $N\in \Dc_A$ and $X\in \Tc_A^\#$ is given by the formula
$N \diamond X=(NX, \mu_{NX}=\mu X \circ N \psi)$.
The tensor product of 
$f:\ N\to M$ in $\Dc_A$ and $g:\ X\to Y$ in $\Tc_A^\#$
is given by the formula
$f\diamond g= \mu Y\circ fg$.
\end{proposition}

\begin{proof}
It is an easy exercise left to the reader. A more conceptual proof is the following. As $\Cc$ acts on the right 
on $\Dc$, we can view both as forming a bicategory $\cal{B}$ with two objects, say $0$ and $1$, with ${\cal B}(1, 1)=\Cc$, 
${\cal B}(1, 0)=\Dc$ and ${\cal B}(0, 0)=1$. For $A$ an algebra in $\Cc$, $(1, A)$ can be regarded as an object of $EM(\cal{B})$, 
and we have that $EM({\cal C})(A, A)=EM(\cal{B})((1, A), (1, A))$ and 
$\Dc_A=EM(\cal{B})((0, \un{1}), (1, A))$. Thus the above action of $\Tc_A^\#$ on $\Dc_A$ is just the composition functor 
in $EM(\cal{B})$.  
\end{proof}

\prref{4.4} justifies the following definition.

\begin{definition}\delabel{4.5}
Let $(A, X, \psi)$ be a cowreath in $\Cc$. An entwined module in $\Dc$ over $(A, X, \psi)$  is a right $(X,\psi)$-comodule in $\Dc_A$.
\end{definition}

An entwined module over $(A, X, \psi)$ consists of an object $M\in \Dc_A$ and a morphism
$\rho:\ M\to M X$ in $\Dc_A$ satisfying
\begin{eqnarray}
&&\hspace*{-2cm}\rho X \circ \rho= \mu X X\circ M \delta\circ \rho;\eqlabel{c1}\\
&&\hspace*{-2cm}\mu\circ M \epsilon\circ \rho= M.\eqlabel{c2}
\end{eqnarray}
\equref{c1} is the coassociativity of the coaction, and \equref{c2} is the counit property.
The fact that $\rho$ is right $A$-linear is expressed by the formula
\begin{equation}\eqlabel{c3}
\rho\circ \mu=\mu X \circ M\psi \circ \rho A.
\end{equation}
A mixed distributive law (or entwining structure) $(A,X,\psi)$ can be considered as a cowreath (or generalized entwining structure): 
take $\delta=\eta\Delta$ and $\epsilon=\eta\varepsilon$. It is easy to see that entwined modules over  $(A, X, \psi)$ considered as
a mixed distributive law coincide with entwined modules over $(A, X, \psi)$ considered as
a monoidal cowreath.\\
A morphism between two entwined modules $M$ and $N$ is a right $A$-linear morphism
$f:\ M\to N$ such that $fX \circ \rho=\rho\circ f$. The category of entwined modules 
in $\Dc_A$ over $(X, \psi, \delta, \epsilon)$ will be denoted as $\Dc(\psi)_A^X$.

\section{Wreaths, wreath product algebras and duality}\selabel{3}
\setcounter{equation}{0}
\subsection{Duality between left and right transfer morphisms}\selabel{3.1}
\begin{theorem}\thlabel{roots}
Let $A$ be an algebra in a (strict) monoidal category $\Cc$.
Take $X\in {\cal T}_A$, and assume that $X\dashv Y$ in $\Cc$. Consider
\begin{equation}\eqlabel{psidual}
\ov{\psi}=\varphi=Y  A   d  \circ Y  \psi  Y \circ  b  A  Y=
{\footnotesize
\gbeg{4}{5}
\gvac{2}\got{1}{A}\got{1}{Y}\gnl
\gdb\gvac{0}\gcl{1}\gcl{1}\gnl
\gcl{1}\gbrc\gcl{1}\gnl
\gcl{1}\gcl{1}\gev\gnl
\gob{1}{Y}\gob{1}{A}
\gend}~:~A  Y\to Y  A.
\end{equation}
Then $(Y,\varphi)\in {}_A{\cal T}$.
If $\Cc$ has right duality, then we have strong monoidal functors
${}^*(-):\ {\cal T}_A^\#\ra {}_A^\#{\cal T}^{\rm oprev}$
and ${}^*(-):\ {\cal T}_A\ra {}_A{\cal T}^{\rm oprev}$.
\end{theorem}

\begin{proof}
We first compute that 
\[
{\footnotesize
\gbeg{3}{5}
\got{1}{A}\got{1}{A}\got{1}{Y}\gnl
\gbrcb\gvac{2}\gcl{1}\gnl
\gcl{1}\gbrcb\gnl
\gmu\gcl{1}\gnl
\gob{2}{A}\gob{1}{Y}
\gend}
\equal{(\ref{eq:psidual})}
{\footnotesize
\gbeg{7}{6}
\gvac{2}\got{1}{A}\gvac{2}\got{1}{A}\got{1}{Y}\gnl
\gdb\gvac{0}\gcl{1}\gdb\gvac{0}\gcl{1}\gcl{1}\gnl
\gcl{1}\gbrc\gcl{1}\gbrc\gcl{1}\gnl
\gcl{1}\gcn{1}{1}{1}{3}\gev\gvac{0}\gcn{1}{1}{1}{-1}\gev\gnl
\gcl{1}\gvac{1}\gmu\gnl
\gob{1}{Y}\gvac{1}\gob{2}{A}
\gend}
\equal{\equref{evcoev}}
{\footnotesize
\gbeg{6}{6}
\gvac{2}\got{1}{A}\got{1}{A}\got{1}{Y}\gnl
\gdb\gvac{0}\gcl{1}\gcl{1}\gcl{1}\gnl
\gcl{1}\gbrc\gcl{1}\gcl{1}\gnl
\gcl{1}\gcl{1}\gbrc\gcl{1}\gnl
\gcl{1}\gmu\gev\gnl
\gob{1}{Y}\gob{2}{A}
\gend}
\equal{(\ref{eq:ta}.a)}
{\footnotesize
\gbeg{6}{6}
\gvac{2}\got{1}{A}\got{1}{A}\got{1}{Y}\gnl
\gdb\gvac{0}\gmu\gcl{1}\gnl
\gcl{1}\gcl{1}\gcn{1}{1}{2}{1}\gvac{1}\gcl{1}\gnl
\gcl{1}\gbrc\gvac{1}\gcn{1}{1}{1}{-1}\gnl
\gcl{1}\gcl{1}\gev\gnl
\gob{1}{Y}\gob{1}{A}
\gend}
\equal{(\ref{eq:psidual})}
{\footnotesize
\gbeg{3}{5}
\got{1}{A}\got{1}{A}\got{1}{Y}\gnl
\gmu\gcl{1}\gnl
\gcn{1}{1}{2}{3}\gvac{1}\gcl{1}\gnl
\gvac{1}\gbrcb\gnl
\gvac{1}\gob{1}{Y}\gob{1}{A}
\gend}~.
\]
It follows immediately from (\ref{eq:evcoev},~\ref{eq:ta}) that 
$\varphi \circo \eta  Y= Y  \eta$, hence $(Y,\varphi)\in {}_A{\cal T}$.\\
$\varphi$ is independent of the choice of the right adjoint $Y$ of $X$ in the following sense. If
$(X,Y,b',d')$ is another adjunction, leading to $\varphi':\ AY'\to Y'A$, then it follows from \equref{1.1.1}. that
\begin{equation}\eqlabel{roots.1}
\lambda A\circ \varphi=\varphi'\circ A\lambda,
\end{equation}
where $\lambda=Y'd\circ b'Y:\ Y\to Y'$. \\
Let $(X,\psi),~(X',\psi')\in \Tc_A$, $X\dashv Y$ and $X'\dashv Y'$. 
For $f:\ X\to X'$ in ${\cal T}_A^\#$, 
\begin{equation}\eqlabel{roots.2}
g=Y  A  d' \circo Y  f   Y' \circo b  Y':\ Y'\to Y  A
\end{equation}
is a morphism  $g:\ Y'\to Y$ in ${}_A^\#{\cal T}$. \\
Assuming that $\Cc$ has right duality, and fixing a right dual
${}^*\hspace*{-2pt}X$ for every $X\in \Cc$, we obtain a functor ${}^*(-):\ {\cal T}_A^\#\ra {}_A^\#{\cal T}^{\rm op}$,
putting ${}^*\hspace*{-1pt}(X,\psi)=({}^*\hspace*{-2pt}X,\varphi)$ and ${}^*f=g$. We leave it to the reader to verify
that ${}^*\hspace*{-1pt}(f'\bullet f)={}^*f\bullet {}^*f'$ and ${}^*\Id_{(X,\psi)}=\Id_{({}^*\hspace*{-2pt}X,\varphi)}$.\\
Let us finally show that ${}^*(-)$ is strong monoidal. It suffices to show that, for $X,X'\in \Tc_A$,
$\varphi_2(X,X'):\ {}^*\hspace*{-1pt}(XX')\to {}^*\hspace*{-2pt}X'{}^*\hspace*{-2pt}X$
defines an isomorphism 
$({}^*\hspace*{-1pt}(XX'), \ov{\psi\cdot \psi'})\to ({}^*\hspace*{-2pt}X'{}^*\hspace*{-2pt}X, \varphi'\cdot \varphi)$
in ${}_A\Tc$, and, a fortiori, in ${}_A^\#\Tc$. We have $X\dashv Y={}^*\hspace*{-2pt}X$, $X'\dashv Y'={}^*\hspace*{-2pt}X'$
and $XX'\dashv Y'Y$, and we claim that $\varphi'\cdot \varphi=\ov{\psi\cdot \psi'}$. To this end it suffices to observe that
the following diagram commutes.
$$\xymatrix{
AY'Y\ar[d]_{b'AY'Y}
\ar@/^3pc/[dddrrrrrr]^{\varphi'\cdot \varphi}
\ar@/_7pc/[dddrrrrrr]_(.40){\ov{\psi\cdot\psi'}}
&&&&&&\\
Y'X'AY'Y\ar[d]_{Y'bX'AY'Y}\ar[rr]^{Y'\psi'Y'Y}&&
Y'AX'Y'Y\ar[d]^{Y'bAX'Y'Y}&&&&\\
Y'YXX'AY'Y\ar[rr]^{Y'YX\psi' Y'Y}&&
Y'YXAX'Y'Y\ar[d]_{Y'Y\psi X'Y'Y}\ar[rr]^{Y'YXAd'Y}&&
Y'YXAY\ar[d]^{Y'Y\psi Y}&&\\
&&Y'YAXX'Y'Y\ar[rr]^{Y'YAXd'Y}&&
Y'YAXY\ar[rr]^{Y'YAd}
&&Y'YA}$$
Combining this formula with \equref{roots.1}, we find that the diagram
$$\xymatrix{
A{}^*\hspace*{-1pt}(XX')\ar[rr]^{\ov{\psi\cdot \psi'}}\ar[d]_{A\varphi_2(X,X')}&&
{}^*\hspace*{-1pt}(XX')A\ar[d]^{\varphi_2(X,X')A}\\
A{}^*\hspace*{-2pt}X'{}^*\hspace*{-2pt}X\ar[rr]^{\varphi'\cdot \varphi}&&{}^*\hspace*{-2pt}X'{}^*\hspace*{-2pt}XA}$$
commutes, which is precisely what we need.
\end{proof}

\subsection{Factorization structures}\selabel{3.2}
\begin{definition}\delabel{3.2}
Let $\Cc$ be a (strict) monoidal category. A left wreath (or left generalized factorization structure) in $\Cc$
is a triple $(A,X,\psi)$, where $A$ is an algebra in $\Cc$, and $(X,\psi)$ is an algebra in $\Tc_A^\#$.
A right wreath is a triple $(A,Y,\varphi)$, where $A$ is an algebra in $\Cc$ and $(Y,\varphi)$ is an algebra in ${}_A^\#\Tc$.
\end{definition}

Explicitly, a right wreath is a triple 
$(A,Y,\varphi)$, where $A$ is an algebra in $\Cc$, and $(Y,\varphi)\in  {}_A^\#\Tc$, together with morphisms
$$m_Y
={\footnotesize
\gbeg{2}{5}
\got{1}{Y}\got{1}{Y}\gnl
\gcl{1}\gcl{1}\gnl
\gsbox{2}\gnl
\gcl{1}\gcl{1}\gnl
\gob{1}{Y}\gob{1}{A}
\gend}: Y Y\ra Y A~~{\rm and}~~
{\eta_Y}
={\footnotesize
\gbeg{2}{5}
\got{2}{\un{1}}\gnl
\gnl
\gsbox{2}\gnl
\gcl{1}\gcl{1}\gnl
\gob{1}{Y}\gob{1}{A}
\gend}: \un{1}\ra Y A
$$
in $\Cc$ such that
\begin{eqnarray}
&&(a)~~\footnotesize{
\gbeg{3}{6}
\got{1}{A}\got{1}{Y}\got{1}{Y}\gnl
\gcl{1}\gcl{1}\gcl{1}\gnl
\gcl{1}\gsbox{2}\gnl
\gbrcb\gvac{2}\gcl{1}\gnl
\gcl{1}\gmu\gnl
\gob{1}{Y}\gob{2}{A}
\gend}
={\footnotesize
\gbeg{3}{6}
\got{1}{A}\got{1}{Y}\got{1}{Y}\gnl
\gbrcb\gvac{2}\gcl{1}\gnl
\gcl{1}\gbrcb\gnl
\gsbox{2}\gvac{2}\gcl{1}\gnl
\gcl{1}\gmu\gnl
\gob{1}{Y}\gob{2}{A}
\gend}
~~,~~(b)~~{\footnotesize 
\gbeg{3}{7}
\got{1}{Y}\got{1}{Y}\got{1}{Y}\gnl
\gcl{1}\gcl{1}\gcl{1}\gnl
\gsbox{2}\gvac{2}\gcl{1}\gnl
\gcl{1}\gbrcb\gnl
\gsbox{2}\gvac{2}\gcl{1}\gnl
\gcl{1}\gmu\gnl
\gob{1}{Y}\gob{2}{A}
\gend}
={\footnotesize 
\gbeg{3}{7}
\got{1}{Y}\got{1}{Y}\got{1}{Y}\gnl
\gcl{1}\gcl{1}\gcl{1}\gnl
\gcl{1}\gsbox{2}\gnl
\gcl{1}\gcl{1}\gcl{1}\gnl
\gsbox{2}\gvac{2}\gcl{1}\gnl
\gcl{1}\gmu\gnl
\gob{1}{Y}\gob{2}{A}
\gend}
~~,\nonumber\\
&&\eqlabel{rightwreath}\\
&&(c)~~{\footnotesize
\gbeg{3}{5}
\gvac{2}\got{1}{A}\gnl
\gvac{2}\gcl{1}\gnl
\gsbox{2}\gvac{2}\gcl{1}\gnl
\gcl{1}\gmu\gnl
\gob{1}{Y}\gob{2}{A}
\gend}
={\footnotesize
\gbeg{3}{6}
\got{1}{A}\gnl
\gcl{1}\gnl
\gcl{1}\gsbox{2}\gnl
\gbrcb\gvac{2}\gcl{1}\gnl
\gcl{1}\gmu\gnl
\gob{1}{Y}\gob{2}{A}
\gend}
~~,~~(d)~~{\footnotesize 
\gbeg{3}{7}
\gvac{2}\got{1}{Y}\gnl
\gvac{2}\gcl{1}\gnl
\gsbox{2}\gvac{2}\gcl{1}\gnl
\gcl{1}\gbrcb\gnl
\gsbox{2}\gvac{2}\gcl{1}\gnl
\gcl{1}\gmu\gnl
\gob{1}{Y}\gob{2}{A}
\gend}
={\footnotesize 
\gbeg{2}{3}
\got{1}{Y}\gnl
\gcl{1}\gu{1}\gnl
\gob{1}{Y}\gob{1}{A}
\gend}
~~,~~(e)~~{\footnotesize 
\gbeg{3}{7}
\got{1}{Y}\gnl
\gcl{1}\gnl
\gcl{1}\gsbox{2}\gnl
\gcl{1}\gcl{1}\gcl{1}\gnl
\gsbox{2}\gvac{2}\gcl{1}\gnl
\gcl{1}\gmu\gnl
\gob{1}{Y}\gob{2}{A}
\gend}
={\footnotesize 
\gbeg{2}{3}
\got{1}{Y}\gnl
\gcl{1}\gu{1}\gnl
\gob{1}{Y}\gob{1}{A}
\gend}~~.\nonumber
\end{eqnarray} 
$m_A$ and $\eta_A$ are the multiplication and unit of the algebra $(Y,\varphi)$. (a) and (c)
express the fact that $m_A:\ YY\to Y$ and $\eta_A:\ \un{1}\to Y$
are morphisms in ${}_A^\#\Tc$; (b) is the associativity and (d) and (e) are the unit conditions.\\

If $(A,Y,\varphi)$ is a right wreath, then $YA$ is an algebra in $\Cc$
with multiplication
$$
{m_{\#}=\footnotesize 
\gbeg{4}{6}
\got{1}{Y}\got{1}{A}\got{1}{Y}\got{1}{A}\gnl
\gcl{1}\gbrc\gcl{1}\gnl
\gsbox{2}\gsbox{2}\gvac{2}\gmu\gnl
\gcl{1}\gcl{1}\gcn{1}{1}{2}{1}\gnl
\gcl{1}\gmu\gnl
\gob{1}{Y}\gob{2}{A}
\gend}
$$
and unit $\eta_{\#}=\eta_Y:\ \un{1}\to YA$, see for example \cite{bc4}. In the literature, this algebra is called
the wreath product or generalized smash product, and is denoted as $Y \#_{\varphi} A$.\\

If $F:\ \Cc\to \Dc$ is strong monoidal, and $C$ is a coalgebra
in $\Cc$, then $F(C)$ is a coalgebra in $\Dc$ with comultiplication and counit given by the formulas
\begin{equation}\eqlabel{roots3}
\Delta_{F(C)}=\varphi_2^{-1}(C,C) \circo F(\Delta)~~;~~\varepsilon_{F(C)}=\varphi_0^{-1} \circo F(\varepsilon).
\end{equation}
Let $(A,X,\psi)$ be a cowreath, and assume that $X\dashv Y={}^*\hspace*{-2pt}X$ in $\Cc$. Then $(X,\psi)$ is a coalgebra in $\Tc_A^\#$, and $(Y,\varphi)$ is coalgebra in ${}^\#_A\Tc^{\rm oprev}$, by \thref{roots}, and therefore an algebra in 
${}^\#_A\Tc$, so that $(A,Y,\varphi)$ is a right wreath. We compute the multiplication and unit
using \equref{roots.2}, with $f:\ X\to X'$ in $\Tc_A^{\#}$ replaced by $\delta:\ X\to XX$ and $\epsilon:\ X\to \un{1}$. We find that
$$m_Y=YAd^2\circ Y\delta YY\circ bYY:\ YY\to YA~~{\rm and}~~\eta_Y=Y\epsilon\circ b.$$
This proves the first part of \prref{duality(co)wreaths}. 
The proof of the second part is similar and is left to the reader. Note also that a different proof can be given by using the 
techniques used in \cite{StreetMW}. 

\begin{proposition}\prlabel{duality(co)wreaths}
Let $\Cc$ be a (strict) monoidal category.

(i) If $(A,X, \psi)$ is a cowreath and $X\dashv Y$ in $\Cc$, then $(A,Y,\varphi)$, with $\varphi$ given by \equref{psidual},
is a right wreath, with multiplication $m_Y$ and unit $\eta_Y$ given by the formulas
\[
m_Y={\footnotesize
\gbeg{6}{7}
\gvac{4}\got{1}{Y}\got{1}{Y}\gnl
\gdb\gvac{2}\gcl{1}\gcl{1}\gnl
\gcl{1}\gcn{1}{1}{1}{3}\gvac{2}\gcl{1}\gcl{1}\gnl
\gcl{1}\gsbox{3}\gvac{3}\gcl{1}\gcl{1}\gnl
\gcl{1}\gcl{1}\gcn{1}{1}{1}{3}\gev\gvac{0}\gcn{1}{1}{1}{-1}\gnl
\gcl{1}\gcl{1}\gvac{1}\gev\gnl
\gob{1}{Y}\gob{1}{A}
\gend
}
~~\mbox{\rm and}~~
\eta_Y={\footnotesize
\gbeg{2}{4}
\got{2}{\un{1}}\gnl
\gdb\gnl
\gcl{1}\gmp{\epsilon}\gnl
\gob{1}{Y}\gob{1}{A}
\gend~}.
\] 
The wreath product $YA$ is an algebra in $\Cc$, with structure maps
\begin{equation}\eqlabel{3.4.1}
m_{\#}=
{\footnotesize
\gbeg{8}{11}
\gvac{4}\got{1}{Y}\got{1}{A}\got{1}{Y}\got{1}{A}\gnl
\gdb\gvac{2}\gcl{1}\gcl{1}\gcl{1}\gcl{1}\gnl
\gcl{1}\gcn{1}{1}{1}{3}\gvac{2}\gcl{1}\gcl{1}\gcl{1}\gcl{1}\gnl
\gcl{1}\gsbox{3}\gvac{3}\gcl{1}\gcl{1}\gcl{1}\gcl{1}\gnl
\gcl{1}\gcl{1}\gcl{1}\gev\gvac{0}\gcl{1}\gcl{1}\gcl{1}\gnl
\gcl{1}\gcn{1}{1}{1}{3}\gcn{1}{1}{1}{3}\gvac{1}\gcn{1}{1}{3}{1}\gcn{1}{1}{3}{1}\gcn{1}{1}{3}{1}\gnl
\gcl{1}\gvac{1}\gcl{1}\gbrc\gcl{1}\gcl{1}\gnl
\gcl{1}\gvac{1}\gmu\gev\gvac{0}\gcl{1}\gnl
\gcl{1}\gvac{1}\gcn{1}{1}{2}{5}\gvac{2}\gcn{1}{1}{3}{1}\gnl
\gcl{1}\gvac{3}\gmu\gnl
\gob{1}{Y}\gvac{3}\gob{2}{A}
\gend}
\hspace{3mm}\mbox{\rm ~~and~~}\hspace{3mm}\eta_{\#}=
{\footnotesize
\gbeg{2}{4}
\got{2}{\un{1}}\gnl
\gdb\gnl
\gcl{1}\gmp{\epsilon}\gnl
\gob{1}{Y}\gob{1}{A}
\gend~~.}
\end{equation}
(ii) If $(A,X, \psi)$ is a left wreath then $(A,Y,\varphi)$ is a left cowreath (a coalgebra in
${}_A^\#\Tc$), 
with comultiplication and counit given by the formulas
\[
\ov{\delta}={\footnotesize
\gbeg{6}{6}
\gvac{4}\got{1}{Y}\gnl
\gvac{1}\gdb\gvac{1}\gcl{1}\gnl
\gvac{1}\gcn{1}{1}{1}{-1}\gvac{-1}\gdb\gvac{-1}\gcn{1}{1}{1}{3}\gvac{1}\gcl{1}\gnl
\gcl{1}\gcl{1}\gsbox{2}\gvac{2}\gcl{1}\gnl
\gcl{1}\gcl{1}\gcl{1}\gev\gnl
\gob{1}{Y}\gob{1}{Y}\gob{1}{A}
\gend
}
~~\mbox{\rm and}~~
\ov{\epsilon}={\footnotesize
\gbeg{3}{5}
\gvac{2}\got{1}{Y}\gnl
\gvac{2}\gcl{1}\gnl
\gsbox{2}\gvac{2}\gcl{1}\gnl
\gcl{1}\gev\gnl
\gob{1}{A}
\gend}~.
\]
\end{proposition}

\subsection{Modules versus entwined modules}\selabel{3.3}
\thref{gntasmod} is the main result of this Subsection.
It is a generalization of \cite[Cor. 6.3]{hp} and its proof follows from \prref{duality(co)wreaths} and an old result 
of Eilenberg-Moore recalled in Section 1 of the paper \cite{street}. This is why we only define the functors that provide 
the desired isomorphism of categories, leaving the details to the reader.

Take  $M \in \Dc(\psi)_A^X$. The coaction $\rho:\ M\to MX$ is right $A$-linear, and satisfies 
\equref{c1} and \equref{c2}. In diagrammatic notation, these conditions take the form
\begin{equation}\label{entwined}
{\footnotesize
\gbeg{2}{4}
\got{1}{M }\got{1}{A}\gnl
\grm\gnl
\grcm\gnl
\gob{1}{M }\gob{1}{C}
\gend} =
{\footnotesize 
\gbeg{3}{5}
\got{1}{M }\got{3}{A}\gnl
\grcm\gcl{1}\gnl
\gcl{1}\gbrc\gnl
\grm\gcl{1}\gnl
\gob{1}{M }\gob{3}{C}
\gend}~~,~~~~~
{\footnotesize
\gbeg{3}{5}
\got{1}{M }\gnl
\grcm\gnl
\gcl{1}\gcn{1}{1}{1}{3}\gnl
\grcm\gcl{1}\gnl
\gob{1}{M }\gob{1}{X}\gob{1}{X}
\gend} =
{\footnotesize 
\gbeg{4}{6}
\got{1}{M }\gnl
\grcm\gnl
\gcl{1}\gcn{1}{1}{1}{3}\gnl 
\gcl{1}\gsbox{3}\gnl
\grm\gcl{1}\gcl{1}\gnl
\gob{1}{M }\gvac{1}\gob{1}{X}\gob{1}{X}
\gend}
\hspace{3mm}{\rm and}\hspace{3mm}
{\footnotesize 
\gbeg{3}{5}
\got{1}{M }\gnl
\grcm\gnl
\gcl{1}\gmp{\epsilon}\gnl
\grm\gnl
\gob{1}{M }
\gend} =
{\footnotesize 
\gbeg{1}{3}
\got{1}{M }\gnl
\gcl{1}\gnl
\gob{1}{M }
\gend}\hspace{2mm}.
\end{equation}

\begin{theorem}\thlabel{gntasmod}
Let $A$ be an algebra in a (strict) monoidal category $\Cc$, let
$\Dc$ be a right $\Cc$-category and let $(A, X, \psi)$ be a cowreath in $\Cc$. 
If $X\dashv Y$ in $\Cc$ then the categories $\Dc(\psi)_A^X$ and $\Dc_{YA}$
are isomorphic. 
\end{theorem}
\begin{proof}
We have a functor $F: \Dc(\psi)_A^X\ra \Dc_{YA}$. 
For $M\in \Dc(\psi)_A^X$, $F(M)=M\in \Cc_{YA}$ via
\begin{equation}\eqlabel{algstr}
\ov{\mu}={\footnotesize
\gbeg{4}{6}
\got{1}{M}\gvac{1}\got{1}{Y}\got{1}{A}\gnl
\grcm\gcl{1}\gcl{2}\gnl
\gcl{1}\gev\gnl
\gcl{1}\gvac{1}\gcn{1}{1}{3}{-1}\gnl
\grm\gnl
\gob{1}{M}
\gend\hspace{2mm}.}
\end{equation}   

We cal also define a functor $G: \Dc_{YA}\ra \Dc(\psi)_A^X$. 
$G(M )=M \in \Dc(\psi)_A^X$ 
via
\begin{equation}\label{comodstr}
\mu ={\footnotesize
\gbeg{2}{3}
\got{1}{M }\got{1}{A}\gnl
\grm\gnl
\gob{1}{M } 
\gend} 
={\footnotesize 
\gbeg{4}{8}
\got{1}{M }\gvac{2}\got{1}{A}\gnl
\gcl{1}\gdb\gvac{0}\gcl{1}\gnl
\gcl{1}\gcl{1}\gmp{\epsilon}\gcl{1}\gnl
\gcl{1}\gcl{1}\gmu\gnl
\gcl{1}\gcl{1}\gcn{1}{1}{2}{1}\gnl
\gsbox{3}\gnl
\gvac{1}\gcl{1}\gnl
\gvac{1}\gob{1}{M }
\gend}
\hspace{3mm}{\rm and}\hspace{3mm}
\rho=
{\footnotesize 
\gbeg{2}{3}
\got{1}{M }\gnl
\grcm\gnl
\gob{1}{M }\gob{1}{X}
\gend} 
= {\footnotesize 
\gbeg{4}{7}
\got{1}{M }\gnl
\gcl{1}\gdb\gnl
\gcl{1}\gcl{1}\gcn{1}{1}{1}{3}\gnl
\gcl{1}\gcl{1}\gu{1}\gcl{1}\gnl
\gsbox{3}\gvac{3}\gcl{1}\gnl
\gvac{1}\gcl{1}\gvac{1}\gcl{1}\gnl
\gvac{1}\gob{1}{M }\gvac{1}\gob{1}{X}
\gend}\hspace{2mm}.
\end{equation}

It can be seen easily that the functors $F$ and $G$ are inverses, and this completes the proof. 
\end{proof}

\section{Frobenius functors versus Frobenius coalgebras}\selabel{4}
\setcounter{equation}{0}
\subsection{Frobenius functors}\selabel{4.1}
Throughout this Section $(A, X, \psi)$ is a cowreath in a (strict) monoidal category $\Cc$. 
Recall that a Frobenius functor is a functor that has a right adjoint which is also a left adjoint.
The aim of this Section is to investigate when the forgetful functor $F : \Cc(\psi)_A^X\ra \Cc_A$ is Frobenius.
\leref{rightadj} tells us that $F$ always has a right adjoint $G$, so that our problem reduces to
examining whether $G$ is a left adjoint of $F$. 

\begin{lemma}\lelabel{rightadj} 
Let $\Dc$ be a right $\Cc$-category. The forgetful functor $F$ has a right adjoint $G : \Dc_A\ra \Dc(\psi)_A^X$, defined as follows:
$G(N)=N X$ is an object of $\Dc(\psi)_A^X$ via  
\[
\mu={\footnotesize
\gbeg{3}{4}
\got{1}{N }\got{1}{X}\got{1}{A}\gnl
\gcl{1}\gbrc\gnl
\grm\gcl{1}\gnl
\gob{1}{N }\gvac{1}\gob{1}{X}
\gend
}~~{\rm and}~~
\rho={\footnotesize
\gbeg{4}{5}
\got{1}{N }\gvac{1}\got{1}{X}\gnl
\gcl{1}\gvac{1}\gcl{1}\gnl
\gcl{1}\gsbox{3}\gnl
\grm\gcl{1}\gcl{1}\gnl
\gob{1}{N }\gvac{1}\gob{1}{X}\gob{1}{X}
\gend
}~.
\]
\end{lemma}

\begin{proof}
The unit and the counit of the adjunction 
are given by the formulas, 
$\eta_M =\rho : M \ra G F (M )=M  X$ and
$\varepsilon_N =\mu\circ N\epsilon:\ F G (N )=N X\ra N$, 
 for all $M \in \Dc(\psi)_A^X$ and $N \in \Dc_A$.
\end{proof}

\subsection{Frobenius coalgebras}\selabel{4.2}
The notion of Frobenius algebra in a monoidal category (as introduced in \cite{street2}, see
also \cite[Def. 4.1]{dbbt2}) can be dualized: a coalgebra in a monoidal category is Frobenius if and only
if the corresponding algebra in the opposite category is Frobenius. This leads to the following definition.

\begin{definition}\delabel{defcoFrobCoalgMon}
A coalgebra $C$ in $\Cc$ is called Frobenius if there exists a Frobenius system $(t,B)$ consisting of
morphisms
$t:\ \un{1}\to C$ (the Frobenius element) and $B:\ CC\to \un{1}$ (the Casimir morphism) in $\Cc$ such that
\begin{equation}\eqlabel{coFrobCoalg}
(a)~~{\footnotesize
\gbeg{3}{5}
\got{2}{C}\got{1}{C}\gnl
\gcmu\gcl{1}\gnl
\gcl{1}\gsbox{2}\gnot{\hspace*{5mm}B}\gnl
\gcl{1}\gnl
\gob{1}{C}
\gend}=
{\footnotesize
\gbeg{3}{5}
\got{1}{C}\got{2}{C}\gnl
\gcl{1}\gcmu\gnl
\gsbox{2}\gnot{\hspace{5mm}B}\gvac{2}\gcl{2}\gnl
\gvac{2}\gob{1}{C}\gnl
\gend}~~,~~(b)~~\hspace{2mm}
{\footnotesize
\gbeg{2}{5}
\gvac{1}\got{1}{C}\gnl
\gmpu{t}\gcl{1}\gnl
\gsbox{2}\gnot{\hspace{5mm}B}\gnl
\gob{2}{\un{1}}
\gend}={\va}_C=
{\footnotesize
\gbeg{2}{5}
\got{1}{C}\gnl
\gcl{1}\gmpu{t}\gnl
\gsbox{2}\gnot{\hspace{5mm}B}\gnl
\gob{2}{\un{1}}
\gend}~.
\end{equation}
\end{definition}

\begin{remark}\relabel{equivcoFrob}
Several equivalent characterizations of a Frobenius algebra are known, see for example
\cite[Theorem 5.1]{dbbt2}. Now a Frobenius coalgebra is a Frobenius algebra in the opposite
category, and this leads to the following equivalent characterizations of a Frobenius coalgebra.
\begin{itemize}
\item[(i)] There is an adjunction $C\dashv A$, and $C$ is isomorphic to $A$ as a right $A$-module. 
Recall from (\ref{eq:dualalgstrofcoalg}-\ref{eq:modstrcoFrob}) that $A$ is an algebra, and that $C$ is a right $A$-module. 
\item[(ii)] There is an adjunction $A'\dashv C$ and $C$ is isomorphic to $A'$ as a left $A'$-module. 
\item[(iii)] $C$ is an algebra in $\Cc$ with $C$-bicolinear multiplication. 
\item[(iv)] There is an adjunction $A'\dashv C$ and there exists a balanced right non-degenerate morphism 
$B_r:\ C  C\ra \un{1}$ in $\Cc$. This means that 
$C  B_r \circ \Delta  C=B_r  C \circ C  \Delta_C$ (balanced)
and that $\Phi_{B_r}=B_r  A'\circ C  b:\ C\ra A'$ 
is an isomorphism (non-degenerate).
\item[(v)] There is an adjunction $C\dashv A$ and there exists a balanced left non-degenerate morphism
$B_l:\ C  C\ra \un{1}$ in $\Cc$. The fact that  $B_l$ is left non-degenerate means that 
$\Psi_{B_l}=A  B_l\circ b'  C:\ C\ra A$ 
is an isomorphism. 
\item[(vi)] There is an adjunction $(b,d):\ C\dashv C$ such that 
$C  \Delta \circ b=\Delta C\circ b$.
\item[(vii)] There is an adjunction $(b,d):\ C\dashv C$ such that $b=\Delta \circ t$, for 
some $t: \un{1}\ra C$ in $\Cc$.
\end{itemize}
A Frobenius coalgebra $C$ is also a Frobenius algebra (and vice versa), with multiplication
$m=CB\circ \Delta C= BC\circ C\Delta$, unit $\eta=t$, and Frobenius system $(\varepsilon, \Delta\circ \eta)$.
\end{remark}

Specializing \deref{defcoFrobCoalgMon} to coalgebras in ${\cal T}_A^\#$, we obtain the following result.

\begin{lemma}\lelabel{CoFrobCoTsha}
A coalgebra $(X, \psi)$ in $ {\cal T}_A^\#$ is Frobenius if and only if there exist morphisms $t: \un{1}\ra A X$ 
and $B: X X\ra A$ in $\Cc$ such that
\begin{equation}\eqlabel{coFrobCoTsh}
{\rm (a)}~~~~{\footnotesize
\gbeg{3}{5}
\got{1}{A}\gnl
\gcl{1}\gnl
\gcl{1}\gsbox{2}\gnot{\hspace{5mm}t}\gnl
\gmu\gcl{1}\gnl
\gob{2}{A}\gob{1}{X}
\gend}=
{\footnotesize
\gbeg{3}{6}
\gvac{2}\got{1}{A}\gnl
\gvac{2}\gcl{2}\gnl
\gsbox{2}\gnot{\hspace{5mm}t}\gnl
\gcl{1}\gbrc\gnl
\gmu\gcl{1}\gnl
\gob{2}{A}\gob{1}{X}
\gend
}~;\hspace*{1cm}{\rm (b)}~~~~
{\footnotesize
\gbeg{3}{6}
\got{1}{X}\got{1}{X}\got{1}{A}\gnl
\gcl{1}\gcl{1}\gcl{3}\gnl
\gsbox{2}\gnot{\hspace{5mm}B}\gnl
\gcn{1}{1}{2}{3}\gnl
\gvac{1}\gmu\gnl
\gvac{1}\gob{2}{A}
\gend
}=
{\footnotesize
\gbeg{3}{7}
\got{1}{X}\got{1}{X}\got{1}{A}\gnl
\gcl{1}\gbrc\gnl
\gbrc\gcl{1}\gnl
\gcl{1}\gsbox{2}\gnot{\hspace{5mm}B}\gnl
\gcl{1}\gcn{1}{1}{2}{1}\gnl
\gmu\gnl
\gob{2}{A}
\gend
}~;
\end{equation}
$$
{\rm (c)}~~~~
{\footnotesize
\gbeg{4}{9}
\gvac{1}\got{1}{X}\gvac{1}\got{1}{X}\gnl
\gvac{1}\gcl{1}\gvac{1}\gcl{3}\gnl
\gsbox{3}\gnot{\hspace{1cm}\d}\gnl
\gcl{1}\gcl{1}\gcl{1}\gnl
\gcl{1}\gcl{1}\gsbox{2}\gnot{\hspace{5mm}B}\gnl
\gcl{1}\gcl{1}\gcn{1}{1}{2}{1}\gnl
\gcl{1}\gbrc\gnl
\gmu\gcl{1}\gnl
\gob{2}{A}\gob{1}{X}
\gend
}=
{\footnotesize
\gbeg{4}{8}
\got{1}{X}\gvac{1}\got{1}{X}\gnl
\gcl{1}\gvac{1}\gcl{1}\gnl
\gcl{1}\gsbox{3}\gnot{\hspace{1cm}\d}\gnl
\gbrc\gcl{1}\gcl{4}\gnl
\gcl{1}\gsbox{2}\gnot{\hspace{5mm}B}\gnl
\gcl{1}\gcn{1}{1}{2}{1}\gnl
\gmu\gnl
\gob{2}{A}\gvac{1}\gob{1}{X}
\gend
}~;\hspace*{1cm}{\rm (d)}~~~~
{\footnotesize
\gbeg{3}{8}
\gvac{2}\got{1}{X}\gnl
\gvac{2}\gcl{3}\gnl
\gsbox{2}\gnot{\hspace{5mm}t}\gnl
\gcl{1}\gcl{1}\gnl
\gcl{1}\gsbox{2}\gnot{\hspace{5mm}B}\gnl
\gcl{1}\gcn{1}{1}{2}{1}\gnl
\gmu\gnl
\gob{2}{A}
\gend
}=\epsilon=
{\footnotesize
\gbeg{3}{8}
\got{1}{X}\gnl
\gcl{1}\gnl
\gcl{1}\gsbox{2}\gnot{\hspace{5mm}t}\gnl
\gbrc\gcl{1}\gnl
\gcl{1}\gsbox{2}\gnot{\hspace{5mm}B}\gnl
\gcl{1}\gcn{1}{1}{2}{1}\gnl
\gmu\gnl
\gob{2}{A}
\gend}~.
$$
\end{lemma} 

\begin{proof}
This is basically a reformulation of \deref{defcoFrobCoalgMon} in the special case where
$\Cc={\cal T}_A^\#$. 
(a) and (b) express 
the fact that $t$ and $B$ are morphisms in ${\cal T}_A^\#$,
and (c) and (d) are a reformulation of (a) and (b) in \equref{coFrobCoalg}. We leave the 
verification of the details to the reader.  
\end{proof}

\subsection{Natural transformations, Frobenius elements and Casimir morphisms}\selabel{4.3}
\begin{definition}\delabel{generator} \cite[Def. 3.1]{dbbt2}
An object $P$ in a monoidal category $\Cc$ is called a left $\ot$-generator if the following condition is satisfied:
if $f,~g:\ YZ\to W$ are morphisms in $\Cc$ such that $f\circ h Z= g\circ h Z$ for all $h:\ P\to Y$ in 
$\Cc$, then $f=g$.
\end{definition}

It is easy to see that a left $\ot$-generator is a generator in the classical sense.
If $\un{1}$ is a left $\ot$-generator for $\Cc$, then $C$ is a Frobenius coalgebra if and only if the forgetful functor  
$\Cc^C\to \Cc$ is Frobenius. In \thref{FrobGenEntwMod}, we will prove the following result:
under the hypothesis that  $\un{1}$ is a left $\ot$-generator in $\Cc$,
the forgetful functor $F : \Cc(\psi)_A^X\ra \Cc_A$ is Frobenius if and only if $(X, \psi)$ is a Frobenius coalgebra in ${\cal T}_A^\#$.
We have to determine when $G$ is a left adjoint of $F$, and to this end we have to investigate
natural transformations $\theta:\ \Id_{\Cc_A}\ra F G$ and $\vartheta:\ GF\ra \Id_{\Cc(\psi)_A^C}$. 
We show that these natural transformations correspond to Frobenius elements (\prref{FrobElem}) and
Casimir morphisms (\prref{CasMor}). 
We recall from \leref{bim1a} that, for $N\in \Cc_A$, we have an isomorphism $\alpha:\ \Cc_A(A,N)\to \Cc(\un{1}, N)$,
$\alpha(\un{h})=h$, with $h=\un{h}\circ \eta$ and $\un{h}=\mu\circ hA$.

\begin{proposition}\prlabel{FrobElem}
Let $(A, X, \psi)$ be a cowreath in $\Cc$.
If $\un{1}$ is a left $\ot$-generator for $\Cc$, then we have an isomorphism
$\Nat(\Id_{\Cc_A}, F G)\cong \Tc^{\#}_A(\un{1}, X)$.
\end{proposition}

\begin{proof}
Consider a natural transformation $\theta:\ \Id_{\Cc_A}\ra F G$. We claim that $t=\alpha(\theta_A)= \theta_A\circ \eta:\ \un{1}\to AX$
is a morphism $\un{1}\to X$ in  $\Tc_A^{\#}$. Take $h\in \Cc(\un{1}, A)$. From the naturality of $\theta$, it follows that
$\theta_A\circ \un{h}=\un{h}X\circ \theta_A$, hence
$$
mX\circ At \circ h= mX\circ hAX\circ t= \un{h}X\circ \theta_A \circ \eta= \theta_A\circ \un{h} \circ \eta
= \theta_A\circ m\circ A\eta\circ h=\theta_A\circ h,
$$
so that 
\begin{equation}\eqlabel{FrobElem1}
mX\circ At=\theta_A,
\end{equation}
since $\un{1}$ is a left $\ot$-generator. Using the right $A$-linearity of $\theta_A$, we find that
$$\theta_A=\theta_A\circ m\circ \eta A= mX\circ A\psi \circ \theta_A A\circ  \eta A=mX\circ A\psi \circ tA.$$
It follows that
$\theta_A= mX\circ At=mX\circ A\psi \circ tA$,
which is precisely (\ref{eq:coFrobCoTsh}.a), expressing that $t\in \Tc^{\#}_A(\un{1}, X)$.\\
Our next aim is to show that $\theta$ is completely determined by $t$. $\theta_A$ is given by \equref{FrobElem1}.
Take $N\in \Cc_A$ and $h:\ \un{1}\to N$ in $\Cc$. Then
$$
\theta_N\circ \mu\circ hA=\theta_N\circ\un{h}\equal{(*)} g_hX\circ \theta_A=\mu X\circ hAX\circ \theta_A=
\mu X\circ N\theta_A\circ hA.
$$
At $(*)$, we used the naturality of $\theta$. From the fact that $\un{1}$ is a left $\ot$-generator, it follows that
$\theta_N\circ \mu= \mu X\circ N\theta_A$ (\ref{eq:coFrobCoTsh}.a) and
$\theta_N= \theta_N\circ \mu\circ  N\eta= \mu X\circ N\theta_A\circ N\eta= \mu X\circ N t$. 
We conclude that
\begin{equation}\eqlabel{FrobElem2}
\theta_N=\mu X\circ N t
\end{equation}
is completely determined by $t$.

Finally, for $t\in \Tc^{\#}_A(\un{1}, X)$, we define $\theta$ using \equref{FrobElem2}. We show that $\theta_N$ is 
right $A$-linear, for all $N\in \Cc_A$.
\begin{eqnarray*}
\mu_{NX}\circ \theta_NA&\equal{\equref{FrobElem2}}& \mu X\circ N\psi\circ \mu XA\circ N tA
= \mu X\circ \mu AX\circ NA\psi \circ NtA\\
&\equal{(x)}&\mu X\circ NmX\circ NA\psi \circ NtA
\equal{(y)} \mu X\circ NmX\circ NAt\\
&\equal{(x)}&\mu X\circ \mu AX\circ NAt= \mu X\circ Nt \circ \mu \equal{\equref{FrobElem2}}  \theta_N\circ \mu.
\end{eqnarray*}
At $(x)$ we used the associativity of $\mu$, and at $(y)$, we used the fact that $t\in \Tc^{\#}_A(\un{1}, X)$.
\end{proof}

\begin{proposition}\prlabel{CasMor}
Let $(A,X,\psi)$ be a cowreath in $\Cc$. 
If $\un{1}$ is a left $\ot$-generator for $\Cc$, then we have a bijective correspondence between 
$\Nat(GF, \Id_{\Cc(\psi)_A^X})$ and the set 
of Casimir morphisms for $(X, \psi)$, that is the subset of $\Tc^{\#}_A(XX, \un{1})$ consisting of morphisms
$B:\ XX\to A$ satisfying (\ref{eq:coFrobCoTsh}.c).
\end{proposition}

\begin{proof}
Consider a natural transformation $\vartheta:\ GF\to  \Id_{\Cc(\psi)_A^X}$, and take $N\in \Cc_A$.
Then $G(X)=NX$ is an entwined module, and for all $h\in \Cc(\un{1}, N)$, we have that
\begin{eqnarray*}
\vartheta_{NX}\circ hXX
&=&
\vartheta_{AX}\circ \un{h}XX\circ \eta XX
\equal{{(a)}} \un{h}X\circ \vartheta_{AX}\circ \eta XX\\
&=& \mu X\circ hAX\circ \vartheta_{AX}\circ \eta XX
= \mu X\circ N \vartheta_{AX} \circ hAXX \circ \eta XX\\
&=& \mu X\circ N\vartheta_{AX} \circ N\eta XX\circ hXX.
\end{eqnarray*}
At $(a)$, we used the naturality of $\vartheta$. Let $\zeta= \vartheta_{AX} \circ \eta XX$. From the fact that
$\un{1}$ is a left $\ot$-generator for $\Cc$, it follows that
\begin{equation}\eqlabel{CasMor1}
\vartheta_{NX}=\mu X\circ N\zeta.
\end{equation}
For an entwined module $M$, the coaction $\rho:\ M\to MX$ is a morphism of entwined modules, and it follows from
the naturality of $\vartheta$ that 
\begin{equation}\eqlabel{CasMor3}
\rho\circ \vartheta_M=\vartheta_{MX}\circ \rho X.
\end{equation} 
This enables us to compute that
\begin{eqnarray}
\vartheta_M&\equal{\equref{c2}}&\mu\circ M\epsilon \circ \rho\circ \vartheta_M
\equal{(\ref{eq:CasMor1},\ref{eq:CasMor3})} \mu\circ M\epsilon \circ \mu X\circ M\zeta\circ \rho X\nonumber\\
&=&  \mu\circ Mm\circ MA\epsilon \circ M\zeta \circ \rho X= \mu \circ MB \circ \rho X\eqlabel{Casmor1b},
\end{eqnarray}
with 
\begin{equation}\eqlabel{CasMor4}
B=m\circ A\epsilon \circ \zeta= m\circ A\epsilon \circ \vartheta_{AX}\circ \eta XX:\ XX\to A.
\end{equation}
This shows that $\vartheta$ is completely determined by $B$.\\

We claim that $B\in \Tc_A^{\#}(XX, \un{1})$. To this end, we need to show that (\ref{eq:coFrobCoTsh}.b) holds,
that is, 
\begin{equation}\eqlabel{CasMor5}
m^2\circ A\epsilon A\circ \vartheta_{AX}A\circ \eta XXA=
m\circ A\epsilon \circ mX\circ A\vartheta_{AX} \circ A\eta XX\circ \psi^2.
\end{equation}
Observe that
$$mX\circ A\vartheta_{AX} \circ A\eta XX=
\vartheta_{AX}\circ mXX\circ A\eta XX= \vartheta_{AX}
= \vartheta_{AX}\circ mXX\circ \eta AXX,$$
so that the right hand side of \equref{CasMor5} equals
$m\circ A\epsilon \circ mX\circ A\vartheta_{AX} \circ \eta AXX\circ \psi^2$.
\equref{CasMor5} then follows from the commutativity of the diagram
\vspace*{-6mm}
$$\xymatrix{
XXA\ar[r]^{\eta XXA}\ar[d]_{\psi^2}&
AXXA\ar[d]_{A\psi^2}\ar[r]^{\vartheta_{AX}A}&
AXA\ar[r]^{A\psi}\ar@/^1.4pc/[rr]^{A\epsilon A}&
AAX\ar[d]^{mX}\ar[r]^{AA\epsilon}&
AAA\ar[d]^{mA}\\
AXX\ar[r]^{\eta AXX}&
AAXX\ar[r]^{mXX}&
AXX\ar[r]^{\vartheta_{AX}}&
AX\ar[r]^{A\epsilon}&
AA}$$
The commutativity of the two squares is obvious, and the commutativity of the rectangle in the middle follows from
the right $A$-linearity of $\vartheta_{AX}$. It follows from \equref{pdf} that $AA\epsilon\circ A\psi=A\epsilon A$.\\

Our next step is to show that $B$ as defined in \equref{CasMor4} satisfies (\ref{eq:coFrobCoTsh}.c), or
\begin{equation}\eqlabel{CasMor6}
mX\circ A\psi\circ AXm \circ AXA\epsilon\circ AX\zeta\circ \delta X
= m^2X\circ AA\epsilon X\circ A\zeta X\circ \psi XX\circ X\delta.
\end{equation}
We will show that the two sides of \equref{CasMor6} are equal to $\zeta$. First
\begin{eqnarray*}
\zeta&=& \vartheta_{AX}\circ \eta XX\\
&\equal{\equref{Casmor1b}}&
mX\circ A\psi\circ AXm\circ AXA\epsilon\circ AX\zeta\circ mXXX\circ A\delta X\circ \eta XX\\
&=&
mX\circ A\psi\circ AXm\circ AXA\epsilon\circ AX\zeta\circ mXXX\circ \eta AXXX\circ \delta X\\
&=&
mX\circ A\psi\circ AXm \circ AXA\epsilon\circ AX\zeta\circ \delta X,
\end{eqnarray*}
the left hand side of \equref{CasMor6}. The diagram below is commutative. 
$$\xymatrix{
XX\ar[d]_{\eta XX}\ar[r]^{X\delta}&
XAXX\ar[r]^{\psi XX}&
AXXX\ar[d]_{\eta AX^3}\ar[dr]^{=}&&\\
AXX\ar[r]^{AX\delta}\ar[dr]_{\theta_{AX}}&
AXAXX\ar[r]^{A\psi X^2}&
AAXXX\ar[r]^{mX^3}&
AXXX\ar[d]_{\vartheta_{AX}X}\ar[dr]^{A\zeta X}&\\
& AX\ar[r]^{A\delta}\ar[rdd]_{A\eta X}&
AAXX\ar[d]^{AA\epsilon X}\ar[r]^{mX^2}&
AXX\ar[d]_{A\epsilon X}&
AAXX\ar[l]_{mX^2}\ar[d]^{AA\epsilon X}\\
&&AAAX\ar[d]^{AmX}\ar[r]^{mAX}&
AAX\ar[d]_{mX}&
AAAX\ar[l]_{mAX}\ar[dl]^{m^2X}\\
&&AAX\ar[r]^{mX}&
AX&}$$
The septangle in the middle commutes because $\theta_{AX}$ preserves the right $X$-coaction. The commutativity
of all the other parts of the diagram is obvious. It follows from the commutativity of the diagram that
$\zeta= mX\circ A\eta X\circ \vartheta_{AX}\circ \eta XX$ is equal to the right hand side of \equref{CasMor6}.\\

Finally, for $B\in \Tc^{\#}_A(XX, \un{1})$ satisfying (\ref{eq:coFrobCoTsh}.c), we define $\vartheta$ using
the formula
\begin{equation}\eqlabel{CasMor7}
\vartheta_M=\mu \circ MB \circ \rho X:\ MX\to M,
\end{equation}
for any entwined module $M$. It is left to the reader to show that $\vartheta_M$ is a morphism of entwined modules,
and that $\vartheta$ is natural in $M$.
\end{proof}

\subsection{Frobenius functors and Frobenius systems}\selabel{4.4}
\begin{theorem}\thlabel{FrobGenEntwMod}
Let $(A,X,\psi)$ be a cowreath in $\Cc$. 
If $\un{1}$ is a left $\ot$-generator for $\Cc$, then the forgetful functor ${F}: \Cc(\psi)_A^X\ra \Cc_A$ is Frobenius if and only if
$(X, \psi)$ is a Frobenius coalgebra in ${\cal T}_A^\#$.  
\end{theorem}

\begin{proof}
Let $F\dashv G$ be the adjunction described in \leref{rightadj}. The functor $F$ is Frobenius if and only if $G\dashv F$,
and this is equivalent to the existence of $\theta\in \Nat(\Id_{\Cc_A}, F G)$ and $\vartheta\in \Nat(GF, \Id_{\Cc(\psi)_A^X})$ 
such that $F(\vartheta_M)\circ \theta_{F(M)}=F(M)$ and $\vartheta_{G(N)}\circ G_{\theta_N}=G(N)$, for all
$M\in \Cc(\psi)_A^C$ and $N\in \Cc_A$.\\
Fix $\theta\in \Nat(\Id_{\Cc_A}, F G)$ and $\vartheta\in \Nat(GF, \Id_{\Cc(\psi)_A^X})$, and let
$t\in \Tc^{\#}_A(\un{1}, X)$ and $B\in \Tc^{\#}_A(XX, \un{1})$ be the Frobenius element and the Casimir morphism
corresponding to $\theta$ and $\vartheta$, see Propositions \ref{pr:FrobElem} and \ref{pr:CasMor}. $(t,B)$ is
a Frobenius system if and only if (\ref{eq:coFrobCoTsh}.d) holds, which comes down to the following: $f=g=\epsilon$,
where 
$f=m\circ AB\circ \psi X\circ Xt$ and $g=m\circ AB\circ tX$.
It is easy to verify that $f_M=F(\vartheta_M)\circ \theta_{F(M)}=\mu\circ MB\circ \rho X\circ \mu X\circ Mt$,
the composition in the top row of the diagram
\vspace*{-2mm}
$$\xymatrix{
M\ar[d]_{\rho}\ar[r]^{Mt}\ar@/^1.4pc/[rrrrr]^{f_M}&
MAX\ar[d]^{\rho AX}\ar[r]^{\mu X}&
MX\ar[r]^{\rho X}&
MXX\ar[r]^{MB}&
MA\ar[r]^{\mu}&
M\\
MX\ar[r]^{MXt}\ar@/^-1.4pc/[rrrrr]_{Mf}&
MXAX\ar[rr]^{M\psi X}&&
MAXX\ar[u]^{\mu XX}\ar[r]^{MAB}&
MAA\ar[u]_{\mu A}\ar[r]^{Mm}&
MA\ar[u]_{\mu}}$$
The pentangle in the diagram commutes by \equref{c3}. The right square commutes by the associativity of $\mu$, and the
commutativity of the two other squares is obvious. We conclude that the diagram commutes.\\
If $f=\epsilon$, then it follows that $F(\vartheta_M)\circ \theta_{F(M)}=f_M=\mu\circ M\epsilon\circ \rho\equal{\equref{c2}} M$.\\
Conversely, if $f_M=M$, for every entwined module $M$, in particular, $f_{AX}=AX$. From the commutativity of the above diagram,
it follows that $\mu_{AX}\circ AXf\circ \rho_{AX}=AX$. Using the unit property of $A$, we find that
\begin{equation}\eqlabel{epsilon}
m\circ A\epsilon\circ \eta X= m\circ \eta A\circ \epsilon=\epsilon.
\end{equation}
Then we compute that
\begin{eqnarray*}
\epsilon&=&
m\circ A\epsilon\circ \mu_{AX}\circ AXf\circ \rho_{AX}\circ \eta X\\
&=& m\circ A\epsilon\circ mX \circ A\psi \circ AXf\circ mXX\circ A\delta\circ \eta X\\
&=& m\circ mA \circ AA\epsilon\circ A\psi \circ AXf\circ \delta\\
&\equal{(\ref{eq:pdf}.c)}& m\circ mA \circ A\epsilon A\circ AXf\circ \delta\\
&=& m\circ Af \circ mX \circ A\epsilon X\circ\delta\\
&\equal{(\ref{eq:pdf}.d)}& M\circ Af\circ \eta X= m\circ \eta A\circ f=f.
\end{eqnarray*}
We conclude that $f=\epsilon$ if and only if $F(\vartheta_M)\circ \theta_{F(M)}=F(M)$, for all $M\in \Cc(\psi)_A^X$. 
In a similar way, we show that 
$g=\epsilon$ if and only if $g_N=\vartheta_{G(N)}\circ G_{\theta_N)}=G(N)$, for all $N\in \Cc_A$. 
$$g_N=\mu X\circ N\psi\circ NXB\circ \mu XXX\circ N\delta X\circ \mu XX\circ NtN,$$
and some (complicated) diagram chasing arguments  using
(\ref{eq:coFrobCoTsh}.a-c) show that
\begin{equation}\eqlabel{FrobGenEntwMod1}
g_N=\mu X\circ N\mu X\circ NAgX\circ N\delta.
\end{equation}
If $g=\epsilon$, then it follows that
$g_N=\mu X\circ N\mu X\circ NA\epsilon X\circ N\delta\equal{(\ref{eq:pdf}.d)} \mu X\circ N\eta X=NX$.\\
The converse implication is more subtle. The tensor product in $\Tc_A^{\#}$ of $t\in \Tc_A^{\#}(\un{1}, X)$ and
the identity $\eta X\in \Tc_A^{\#}(X,X)$ is $tX\in  \Tc_A^{\#}(X,XX)$. The composition in $\Tc_A^{\#}$ 
of $tX\in  \Tc_A^{\#}(X, XX)$ and $B\in \Tc_A^{\#}(XX, \un{1})$ is then precisely $g\in \Tc_A^{\#}(X, \un{1})$, see \seref{2}.
We conclude that $g$ is a morphism in $\Tc_A^{\#}(X, \un{1})$, so that it follows from \equref{2.1.0} that
\begin{equation}\eqlabel{FrobGenEntwMod2}
m\circ A g\circ \psi=m\circ gA.
\end{equation}
If $g_A=AX$, then it follows that
\begin{eqnarray*}
\epsilon&\equal{\equref{epsilon}}&m\circ A\epsilon\circ \eta X=m\circ A\epsilon\circ AX\circ  \eta X
~\equal{\equref{FrobGenEntwMod1}}~ m\circ A\epsilon\circ m^2X\circ AAgX\circ A\delta\circ \eta X\\
&=& m\circ Am\circ AgA \circ AX\epsilon \circ\delta
~\equal{\equref{FrobGenEntwMod2}}~
m\circ Am\circ AAg \circ A\psi \circ AX\epsilon \circ\delta\\
&=& m\circ mA\circ AAg \circ A\psi \circ AX\epsilon \circ\delta
= m\circ Ag \circ mX\circ A\psi \circ AX\epsilon \circ \delta\\
&\equal{(\ref{eq:pdf}.e)}
& m\circ Ag \circ \eta X= m\circ \eta A\circ g=g.
\end{eqnarray*}
\end{proof}

\section{Frobenius coalgebras versus Frobenius corings}\selabel{5}
\setcounter{equation}{0}
Throughout this Section, $\Cc$ is a (strict) monoidal category with coequalizers, and $A$ is a left coflat algebra in $\Cc$.
${}^!\Tc^{\#}_A$ is the full subcategory of $\Tc^{\#}_A$ consisting of right transfer morphisms $(X,\psi)$ with
$X$ left coflat in in $\Cc$.

\begin{lemma}\lelabel{6.1}
We have a fully faithful strong monoidal functor $H:\ {}^!\Tc^{\#}_A\to \ACA$.
\end{lemma}

\begin{proof}
Take $(X,\psi)\in {}^!\Tc^{\#}_A$. 
It follows from \prref{bim3} that $AX$ is robust as a left $A$-module; $AX$ is left coflat since $A$ and $X$ are
left coflat. $AX$ is an $A$-bimodule, with left $A$-action $\nu=mX$ and right $A$-action
$\mu= mX\circ A\psi$. We conclude that $AX\in \ACA$, and we define $H(X,\psi)=AX$.\\
Let $f:\ X\to Y$ in ${}^!\Tc^{\#}_A$, and define $H(f)=mY\circ Af$. It is clear that $Hf$ is left $A$-linear, and the
right $A$-linearity follows from \equref{2.1.0}. $H:\ {}^!\Tc^{\#}_A(X,Y)\to \ACA(AX,AY)$ is bijective. The inverse
of $\un{f}:\ AX\to AY$ is $H^{-1}(\un{f})=\un{f}\circ \eta f$. Thus $H$ is a fully faithful functor; the monoidal structure is
the following: $\varphi_0=A:\ A\to H(\un{1})=A$, and $\varphi_2(X,Y):\ AX\bullet AY\to AXY$ is the unique isomorphism
of coequalizers, see \prref{bim1}. 
\end{proof}

In fact, if we make the identification of coequalizers $(AX\bullet AY,q)= (AXY, \mu_{AX}Y)$, then $H$ becomes strictly monoidal.
It follows immediately from \leref{6.1} that there exists a bijective correspondence between
coalgebra structures on $(X,\psi)\in {}^!\Tc^{\#}_A$ and $A$-coring structures on $AX$. Moreover, Frobenius systems on
$(A,\psi) $ correspond bijectively to Frobenius systems on $AX$, and the following result follows.

\begin{theorem}\thlabel{coFrobTvscoring}
With notation and assumptions as above, $(X,\psi)$ is a Frobenius coalgebra in $\Tc_A^\#$
if and only if $\mfC=AX$ is a Frobenius $A$-coring, i.e. a Frobenius coalgebra in $\ACA$.
\end{theorem}

Combining Theorems \ref{th:FrobGenEntwMod} and \ref{th:coFrobTvscoring}, we obtain the following result.

\begin{corollary}\colabel{6.2}
Assume that $\un{1}$ is a left $\ot$-generator for $\Cc$. Let $A$ be a left coflat algebra in $\Cc$, and take
$(X,\psi)\in {}^!\Tc^{\#}_A$. Then the following assertions are equivalent:
\begin{itemize}
\item[(i)] The forgetful functor $U: \Cc^{AX}\ra \Cc_A$ is Frobenius;
\item[(ii)] $AX$ is a Frobenius $A$-coring.
\end{itemize}
\end{corollary}

\begin{proof}
From \cite[Theorem 4.8]{bc4} we know that the categories $\Cc^\mfC$ and $\Cc(\psi)_A^X$ are isomorphic. 
Thus $U$ is a Frobenius functor if and only if the forgetful functor $F: \Cc(\psi)_A^X\ra \Cc_A$ 
is Frobenius. Since $\un{1}$ is a left $\ot$-generator for $\Cc$, it follows from \thref{FrobGenEntwMod} that 
$F$ is Frobenius if and only if $(X, \psi)$ is a Frobenius coalgebra in ${\cal T}_A^\#$, and this is equivalent
to $AX$ being a Frobenius $A$-coring, by \thref{coFrobTvscoring}.
\end{proof}

Another immediate corollary of \leref{6.1} is the following. If $(X,\psi)$ has a right adjoint in ${}^!\Tc^{\#}_A$,
then $AX$ has a right adjoint in $\ACA$. In particular, if $(X,\psi)$ is a Frobenius coalgebra in ${}^!\Tc^{\#}_A$,
then it is selfadjoint, and $AX$ is selfadjoint in $\ACA$. \prref{6.3} is a result of the same type, but with a more
complicated proof. 
Take $(X,\psi)\in {}^!\Tc^{\#}_A$. If $X$ has a right adjoint $Y$ in $\Cc$, then $(Y,\varphi)\in {}^{\#}_A\Tc$,
and $YA$ is an $A$-bimodule.

\begin{proposition}\prlabel{6.3}
With notation as above, we assume that $YA\in \ACA$. Then $AX\dashv YA$ in $\ACA$.
\end{proposition}

\begin{proof}
We have an adjunction $(X,Y,b,d)$ in $\Cc$. Recall that $b:\ \un{1}\to YX$, $d:\ XY\to \un{1}$.
$$B= Y\psi\circ bA:\ A\to YAX=YA\bullet AX~~{\rm and}~~m\circ AdA: AXYA\to A$$
are morphisms of $A$-bimodules. Let $D$ be the unique morphism of $A$-bimodules 
that makes the following diagram commutative:
\begin{equation}\eqlabel{6.3.1}
\xymatrix{
AXAYA\ar@<-.5ex>[rr]_(.54){\mu YA} 
\ar@<.5ex>[rr]^(.55){AX\nu}&&
AXYA\ar[d]_{AdA}\ar[rr]^q&&
AX\bullet YA\ar[d]^{\exists ! D}\\
&&AA\ar[rr]^m&&A}
\end{equation}
Applying Propositions \ref{pr:bim1} and \ref{pr:bim6}, we have isomorphisms of colimits
\begin{equation}\eqlabel{6.3.2}
(AX\bullet YA\bullet AX,q_2)\cong (AX\bullet YAX,q\circ AXYmX)\cong ((AX\bullet YA)X,\mu_{AX\bullet YA}\circ qAX).
\end{equation}
Consider the diagram
\begin{equation}\eqlabel{6.3.3}
\xymatrix{
&&AXA\ar[d]_{AXB}\ar[rr]^{q}&& AX\bullet A\ar[d]^{AX\bullet B}\\
AXYAAX\ar[d]_{=}\ar[rr]^{AXYmX}&&AXYAX\ar[rr]^q&&AX \bullet YAX\ar[d]^{\cong}\ar@{->}@/^4pc/[dd]^{\delta}\\
AXYAAX\ar[d]_{AdAAY}\ar[rr]^{qAX} &&(AX\bullet YA)AX\ar[d]_{DAX}\ar[rr]^{\mu X}&&(AX\bullet YA)X\ar[d]^{D\bullet AX}\\
AAAX\ar[rr]^{mAX}&&AAX\ar[rr]^{mX}&&AX}
\end{equation}
The commutativity of top and bottom right squares follows from \equref{tensor}; the commutativity of the
rectangle in the middle follows from \equref{6.3.2}; the commutativity of the bottom left square follows from the definition of $D$.
We conclude that the diagram commutes. Let $\delta$ be the composition of $D\bullet AX$ and the isomorphism
$AX\bullet YAX\cong (AX\bullet YA)X$, as indicated in the diagram, and consider the diagram
\begin{equation}\eqlabel{6.3.4}
\xymatrix{
AXYAAX\ar[d]_{AdAAX}\ar[rr]^{AXYmX}&&AXYAX\ar[d]^{Ad'AX}\ar[rr]^q&&AX\bullet YAX\ar[d]^{\delta}\\
AAAX\ar[rr]^{AmX}&&AAX\ar[rr]^{mX}&&AX}
\end{equation}
$$
\delta\circ q\circ AXYmX\equal{\equref{6.3.3}} mX\circ mAX\circ Ad AAY
= mX\circ AmX\circ Ad AAY=mX\circ AdAX\circ AXYmY.
$$
We used the associativity of $m$ and the fact that the left square in \equref{6.3.4} commutes. It follows that
$\delta\circ q=mX\circ AdAX$, since $(AXYAX,AXYmX)$ is a coequalizer, see \prref{bim1}. We conclude that the
diagram \equref{6.3.4} commutes. We therefore have commutative diagrams
$$\xymatrix{
AXA\ar[d]_{AXB}\ar[rr]^q&& AX\bullet A\ar[d]^{AX\bullet B}\\
AXYAX\ar[d]_{AdAX}\ar[rr]^q&&AX\bullet YAX\ar[d]^{\delta}\\
AAX\ar[rr]^{mX}&&AX}\hspace*{1cm}{\rm and}\hspace*{1cm}
\xymatrix{AXA\ar[drr]_{\mu_{AX}}\ar[rr]^q&& AX\bullet A\ar[d]^{\cong}\\
&&AX}$$
Now
$$
mX\circ AdAX\circ AXB= mX\circ AdAX\circ AXY\psi\circ AXbA
= mX\circ A\psi\circ AdXA\circ AXbA=\mu_{AX}.
$$
Since $(AX\bullet A,q)$ is a coequalizer, we conclude that $\delta\circ (AX\bullet B)$ is the canonical isomorphism
$AX\bullet A\cong AX$, as needed.\\
In a similar way, the diagram commutes
$$\xymatrix{
&&AYA\ar[d]_{BYA}\ar[rr]^q&&A\bullet YA\ar[d]^{B\bullet YA}\\
YAAXYA\ar[d]_=\ar[rr]^{YmXYA}&&YAXYA\ar[rr]^q&&YAX\bullet YA\ar[d]^{\cong}\ar@{->}@/^4pc/[dd]^{\delta'}\\
YAAXYA\ar[d]_{YAAdA}\ar[rr]^{YAq}&&YA(AX\bullet YA)\ar[d]_{YAD}\ar[rr]^{Y\nu}&&Y(AX\bullet YA)\ar[d]_{YA\bullet D}\\
YAAA\ar[rr]^{YAm}&&YAA\ar[rr]^{Ym}&&YA}$$
Then we consider the diagram
$$\xymatrix{
YAAXYA\ar[d]_{YAAdA}\ar[rr]^{YmXYA}&&YAXYA\ar[d]^{YAdA}\ar[rr]^q&&YAX\bullet YA\ar[d]^{\delta}\\
YAAA\ar[rr]^{YmA}&&YAA\ar[rr]^{Ym}&&YA}$$
We compute that
$$\delta\circ q\circ YmXYA=Ym\circ YAm\circ YAAdA=Ym\circ YmA\circ YAAdA=Ym\circ YAdA\circ YmXYA,$$
and conclude that $\delta\circ q=Ym\circ YAdA$, so that we have the commutative diagrams
$$\xymatrix{
AYA\ar[d]_{BYA}\ar[rr]^q&& A\bullet YA\ar[d]^{B\bullet YA}\\
YAXYA\ar[d]_{YAdA}\ar[rr]^q&& YAX\bullet YA\ar[d]^{\delta'}\\
YAA\ar[rr]^{Ym}&&YA}\hspace*{1cm}{\rm and}\hspace*{1cm}
\xymatrix{AYA\ar[drr]_{\nu_{YA}}\ar[rr]^q&& A\bullet YA\ar[d]^{\cong}\\
&&YA}$$
Now
$$Ym\circ YAdA\circ BYA=Ym\circ YAdA\circ Y\psi YA\circ bAYA=\nu_{YA},$$
and we conclude that $\delta' \circ B\bullet YA$ is the canonical isomorphism $Y\bullet YA\cong YA$,
finishing the proof.
\end{proof}

In the setting of \prref{6.3}, assume moreover that $(X,\psi)$ is a coalgebra in ${}^!\Tc^{\#}_A$.
It follows from \leref{6.1} that $AX$ is an $A$-coring; we know from \prref{6.3} that $AX\dashv YA$ in $\ACA$,
hence $YA$ is an $A$-ring, with structure given by \equref{dualalgstrofcoalg}. The structure maps are
denoted as $m_{YA}:\ YA\bullet YA=YYA\to YA$ and $\eta_A:\ A\to YA$.\\
We have seen in \prref{duality(co)wreaths} that $(A,Y,\varphi)$ is a right wreath, and we can consider
the wreath product $YA=Y\#_{\varphi}A$, which is an algebra in $\Cc$. As in \prref{duality(co)wreaths},
the multiplication and the unit are denoted as $m_{\#}:YAYA\to YA$ and $\eta_{\#}:\ 1\to YA$.\\
$AX$ is an entwined module, and therefore a right $YA$-module in $\Cc$, with right $YA$-action
$\ov{\mu}:\ AXYA\to AX$ given by \equref{algstr}. Since $AX\dashv YA$ in $\ACA$, $AX$ is also a right $YA$-module
in $\ACA$, with right $YA$-action $\tilde{\mu}:\ AX\bullet YA\to AX$ given by \equref{modstrcoFrob}.\\
In \prref{6.4} we investigate the relation between these two sets of structures.

\begin{proposition}\prlabel{6.4}
With notation as above, $m_{\#}=m_{YA}\circ Y\nu_{YA}$, $\eta_{\#}=\eta_{YA}\circ \eta$ and $\ov{\mu}=\tilde{\mu}\circ q$.
\end{proposition}

\begin{proof}
According to \equref{dualalgstrofcoalg}, 
$m_{YA}=YA\bullet D^2 \circo Y\Delta\bullet YYA\circo B\bullet YYA$, where
$\Delta=mX\circ A\delta:\ AX\to AX\bullet AX=AXX$ is the comultiplication on $AX$.
Consider the commutative diagram
$$\xymatrix{
AYYA\ar[d]_{BYYA}\ar[rr]^{\nu_{YYA}}&& YYA=A\bullet YYA\ar[d]^{B\bullet YYA}\\
YAXYYA\ar[d]_{Y\Delta YYA}\ar[rr]^q&&YAX\bullet YYA\ar[d]^{Y\Delta\bullet YYA}\\
YAXXYYA\ar[d]_{YAd^2A}\ar[rr]^q&& YAXX\bullet YYA\ar[d]^{YA\bullet D^2}\\
YAA\ar[rr]^{Ym}&&YA\bullet A=YA}$$
$m_{YA}$ is the unique morphism such that $m_{YA}\circ \nu_{YYA}=f$, where $f$ is southwestern composition
$f=Ym\circ YAd^2A\circ Y\Delta YYA \circ BYYA$
in the diagram. We have to show that the triangle in the diagram
$$\xymatrix{
AYAYA\ar[d]_{\nu_{YA}YA}\ar[rr]^{AY\nu_{YA}}&&AYYA\ar[d]^{\nu_{YYA}}\ar@{->}@/^3pc/[dd]^{f}\\
YAYA\ar[rr]^{\nu_{YA}}\ar[drr]_{m_{\#}}&& YYA\ar[d]^{m_{YA}}\\
&&YA}$$
commutes. Since $(YAYA, \nu_{YA}YA)$ is a coequalizer, and since the rectangle in the diagram commutes,
it suffices to show that
\begin{equation}\eqlabel{6.4.1}
m_{\#}\circ \nu_{YA}YA= f\circ AY\nu_{YA}.
\end{equation}
The fact that $\delta:\ X\to AXX$ defines a morphism in $\Tc_A^{\#}$ is expressed by the formula
$$mXX\circ A\delta \circ \psi = mXX\circ A\psi X\circ AX\psi\circ \delta A=\mu_{AAX}\circ \delta A.$$
Using this formula and the definition of $B$, we can show that $f\circ AY\nu_{YA}=g_1\circ g$, with
\begin{eqnarray*}
g_1&=& Ym\circo YAd^2A\circo Y\mu_{AXX}Y\nu_{YA}:\ YAXXAYAYA\to YA;\\
g&=& Y\delta AYAYA\circo b AYAYA:\ AYAY\to YAXXAYAYA.
\end{eqnarray*}
Using \equref{3.4.1}, we compute that $m_{\#}\circo \nu_{YA}YA=g_2\circo g$, with
$$g_2=Ym^2\circo YAAdA \circo Y\psi YA\circo YAXdAYA\circo YAXX\nu_{YA}YA:\ YAXXAYAYA\to YA.$$
Therefore it suffices to show that $g_1=g_2$. To this end, we consider the following diagram.
We slightly simplified the notation for the morphisms in the diagram, deleting identity morphisms on tensor factors;
for example, $\psi$ in the top left corner is a shorter notation for $YAX\psi YAYA$.
It follows from \equref{6.4.2} that the top left square in the diagram commutes. We easily deduce from \equref{6.4.2}
that $dA\circ X\nu_{YA}=m\circ AdA =\psi YA$, telling us that the lower pentagon in the diagram commutes. The
associativity of $m$ entails that the lower right square commutes. The commutativity of the three remaining squares,
the remaining pentangle, and the octangle at the lower left is an obvious consequence of the property that
$Cg\circ fB=fD\circ Ag$ for $f:\ A\to C$ and $g:\ B\to D$. Our conclusion is that the diagram commutes. The composition
taking $YAXXAYAYA$ in the top left corner to $YA$ via the southwestern route is $g_1$, and the composition via the
northeastern route is $g_2$. It follows that $g_1=g_2$ and \equref{6.4.1} is satisfied.
$$\xymatrix{
YAXXAYAYA\ar[d]_{\psi}\ar[r]^{\varphi}&
YAXXYAAYA\ar[d]^{d}\ar[r]^{m}&
YAXXYAYA\ar[d]^{d}&\\
YAXAXYAYA\ar[d]_{\psi}\ar[r]^{d}&
YAXAAYA\ar[d]^{\psi}\ar[r]^{m}&
YAXAYA\ar[dr]^{\psi}&\\
YAAXXYAYA\ar[d]_{m}\ar[r]^{d}&
YAAXAYA\ar[d]^{\nu_{YA}}\ar[r]^{\psi}&
YAAAXYA\ar[d]^{d}\ar[r]^{Am}&
YAAXYA\ar[d]^{d}\\
YAAXXYAYA\ar[dr]^{\nu_{YA}}&
YAAXYA\ar[dr]^{dA}&
YAAAA\ar[d]^{YAAm}\ar[r]^{YAmA}&
YAAA\ar[d]^{Ym^2}\\
&YAXXYYA\ar[dr]^{d^2}&
YAAA\ar[r]^{Ym^2}&YA\\
&&YAA\ar[ur]^{Ym}&}$$
 From the definition of $\varphi$,
see \equref{psidual}, it easily follows that
\begin{equation}\eqlabel{6.4.2}
Ad\circ \psi Y=dA\circ X\varphi.
\end{equation}
Let us prove the second formula. As an application of \equref{bim1a.1}, we find that $YA\bullet \varepsilon_{AX}=
Y\varepsilon_{AX}:\ YA\bullet AX=YAX\to YA\bullet A=YA$.
\begin{eqnarray*}
&&\hspace*{-2cm}
\eta_{YA}\circ \eta~\equal{\equref{dualalgstrofcoalg}}~YA\bullet AX \circ B\circ \eta = Y\varepsilon_{AX} \circ Y\psi\circ bA\circ \eta
= Ym\circ YA\epsilon \circ Y\psi\circ bA\circ \eta\\
&\equal{(\ref{eq:pdf}(c))}& Ym\circ Y\epsilon A\circ bA\circ \eta
= Ym\circ YA\eta\circ Y\epsilon \circ b~\equal{\equref{3.4.1}}~ \eta_\#.
\end{eqnarray*}

It follows from \equref{modstrcoFrob} that $\tilde{\mu}=AX\bullet D\circo \Delta\bullet YA$. This fits into the diagram
$$\xymatrix{
AXYA\ar[d]_{\Delta}\ar[rr]^q&&
AX\bullet YA\ar[d]^{\Delta\bullet YA}\\
AXXYA\ar[d]_{AXDA}\ar[rr]^q&&
AXX\bullet YA\cong AX\bullet AX\bullet YA\ar[d]^{AX\bullet D}\\
AXA\ar[rr]^{\mu_{AX}}&&AX}$$
It follows that $\tilde{\mu}\circ q= \mu_{AX}\circ AXdA\circ D=\ov{\mu}$.
\end{proof}

From now on we will make the following assumptions: $\Cc$ is a (strict) monoidal category with coequalizers;
every object of $\Cc$ is left coflat; $A$ is an algebra in $\Cc$
such that every left $A$-module in $\Cc$ is robust. In this situation the categories $\ACA$ and ${}_A\Cc_A$ coincide.\\
An algebra morphism $A\to S$ is also called an algebra extension; algebra extensions correspond to $A$-rings,
these are algebras in the category ${}_A\Cc_A$. $A\to S$ is called a Frobenius algebra extension if $S$ is
a Frobenius algebra in ${}_A\Cc_A$. Now we can state the main result of this Section.

\begin{theorem}\thlabel{Frobcaractfinitecase}
Let $(A, X, \psi)$ be a cowreath, and assume that $X\dashv Y$ in $\Cc$.
Then the following assertions are equivalent:
\begin{itemize}
\item[(i)] $(X, \psi)$ is a Frobenius coalgebra in ${\cal T}^\#_A$;
\item[(ii)] $YA$ is a Frobenius $A$-ring;
\item[(iii)] $(Y, \varphi)$ is a Frobenius algebra in ${}_A^\#{\cal T}$;
\item[(iv)] the algebra extension $Ym\circ \eta_{YA}A:\ A\to YA$ is Frobenius;
\item[(v)] $AX$ and $YA$ are isomorphic as left $A$, right $YA$-modules in $\Cc$.
\item[(vi)] $AX$ and $YA$ are isomorphic as left $A$-modules and as entwined modules;
\item[(vii)] there exists  $t: \un{1}\ra X$ in ${\cal T}_A^\#$ (that is a Frobenius 
element for $(X, \psi)$ in ${\cal T}_A^\#$) such that 
\begin{equation}\eqlabel{Frobcar1}
\Phi=m^2X\circ AA\psi A\circ AAXdA\circ A\delta YA\circ tYA:\ YA\to AX
\end{equation}
is an isomorphism in $\Cc$;
\item[(viii)] there exists $B: X\ot X\ra \un{1}$ in ${\cal T}_A^\#$ satisfying (\ref{eq:coFrobCoTsh}.c) 
(that is a Casimir morphism for $(X, \psi)$ in ${\cal T}_A^\#$) such that
\begin{equation}\eqlabel{Frobcar2}
\Psi=Ym\circ YAB\circ Y\psi X\circ bAX:\ AX\to YA
\end{equation}
is an isomorphism in $\Cc$.
\end{itemize}
If $\un{1}$ is a left $\ot$-generator of $\Cc$, then these statements are equivalent to
\begin{itemize}
\item[(ix)] The functor $F:\ \Cc(\psi)_A^X\ra \Cc_A$ is a Frobenius functor. 
\end{itemize}
\end{theorem}

\begin{proof}
\un{$(i)\Leftrightarrow (ii)$.} By \thref{coFrobTvscoring}, $(X, \psi)$ is a Frobenius coalgebra in ${\cal T}_A^\#$ 
if and only if $AX$ is a Frobenius coalgebra in ${}_A\Cc_A$. Since $AX\dashv YA$ in ${}_A\Cc_A$,
this is equivalent to $YA$ is a Frobenius algebra in ${}_A\Cc_A$, which is a Frobenius $A$-ring.\\
\un{$(ii)\Leftrightarrow (iii)\Leftrightarrow (iv)$} follows from \cite[Corollary 8.8]{dbbt2} applied to the 
wreath $(A, Y, \varphi)$ in $\ov{\Cc}$. 
Observe that it is possible to give a direct proof of the equivalence of $(ii)$ and $(iii)$ using 
arguments similar to the ones  in the proof of \thref{coFrobTvscoring}. \\
\un{$(i)\Leftrightarrow (v)$.} We use the characterization (i) in \reref{equivcoFrob}, applied to the colgebra $AX$ 
in ${}_A\Cc_A$. We have that $(X, \psi)$ in ${\cal T}_A^\#$ is a Frobenius coalgebra 
if and only if $AX$ is a Frobenius $A$-coring. Since $AX\dashv YA$ in ${}_A\Cc_A$, this is
equivalent to $AX$ and $YA$ being isomorphic as right $YA$-modules 
in ${}_A\Cc_A$. By \prref{6.4} it then follows that $(X, \psi)$ is a Frobenius coalgebra 
if and only if $AX$ and $YA$ are isomorphic as $A$-bimodules and right 
$YA$-modules in $\Cc$. But the right $A$-module structure on 
$AX$ and $YA$ is inherited from the right action of $YA$ 
on them via the restriction of scalars functor defined by the algebra extension 
$Ym\circ \eta_{YA}A:\ A\to YA$. 
So we conclude that $(X, \psi)$ is a Frobenius coalgebra in ${\cal T}_A^\#$ if and only if 
$AX$ and $Y A$ are isomorphic as left $A$, right $YA$-modules.\\
\un{$(v)\Leftrightarrow (vi)$} is an immediate consequence of \thref{gntasmod}.\\
\un{$(v)\Leftrightarrow (vii)$} is based on the elementary observation that a right 
$YA$-linear morphism $\Phi: YA\to AX$ is completely determined 
by the morphism $t=\Phi\circ \eta_{\#}:\ \un{1}\to AX$  in $\Cc$. Looking at the right $A$-action
$\ov{\mu}:\ AXYA\to AX$ (see the proof of \thref{gntasmod} and \prref{6.4}), we easily find that $\Phi$
is given by \equref{Frobcar1}. Since 
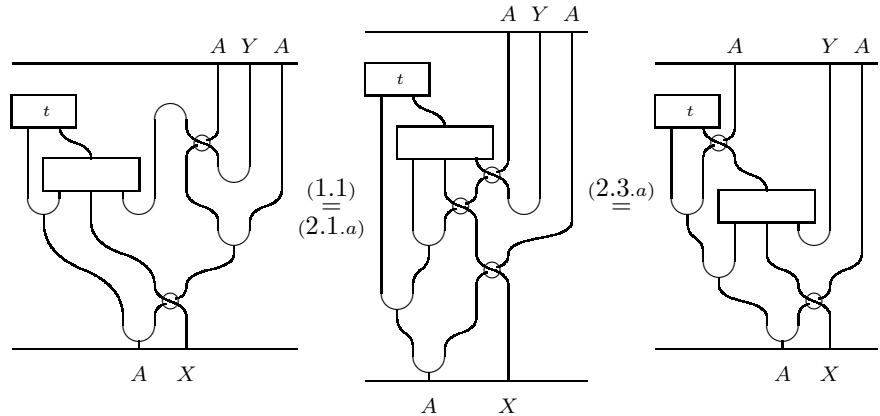
\[{\footnotesize
\gbeg{9}{11}
\gvac{6}\got{1}{A}\got{1}{Y}\got{1}{A}\gnl
\gvac{6}\gcl{1}\gcl{1}\gcl{1}\gnl
\gsbox{2}\gnot{\hspace{5mm}t}\gvac{4}\gdb\gvac{0}\gcl{1}\gcl{1}\gcl{1}\gnl
\gcl{1}\gcn{1}{1}{1}{3}\gvac{2}\gcl{1}\gbrc\gcl{1}\gcl{1}\gnl
\gcl{1}\gsbox{3}\gvac{3}\gcl{1}\gcl{1}\gev\gvac{0}\gcl{1}\gnl
\gmu\gcl{1}\gev\gvac{0}\gcn{1}{1}{1}{3}\gvac{1}\gcn{1}{1}{3}{1}\gnl
\gcn{1}{3}{2}{7}\gvac{1}\gcn{1}{2}{1}{5}\gvac{3}\gmu\gnl
\gvac{6}\gcn{1}{1}{2}{-1}\gnl
\gvac{4}\gbrc\gnl
\gvac{3}\gmu\gcl{1}\gnl
\gvac{3}\gob{2}{A}\gob{1}{X}
\gend}
\equalupdown{(\ref{eq:evcoev})}{\equuref{ta}{a}}
{\footnotesize
\gbeg{7}{13}
\gvac{4}\got{1}{A}\got{1}{Y}\got{1}{A}\gnl
\gvac{4}\gcl{4}\gcl{4}\gcl{4}\gnl
\gsbox{2}\gnot{\hspace{5mm}t}\gnl
\gcl{1}\gcn{1}{1}{1}{3}\gnl
\gcl{1}\gsbox{3}\gnl
\gcl{1}\gcl{1}\gcl{1}\gbrc\gcl{1}\gcl{2}\gnl
\gcl{1}\gcl{1}\gbrc\gev\gnl
\gcl{1}\gmu\gcl{1}\gvac{2}\gcn{1}{1}{1}{-3}\gnl
\gcl{1}\gcn{1}{1}{2}{1}\gvac{1}\gbrc\gnl
\gmu\gvac{1}\gcl{1}\gcl{3}\gnl
\gcn{1}{1}{2}{3}\gvac{1}\gcn{1}{1}{3}{1}\gnl
\gvac{1}\gmu\gnl
\gvac{1}\gob{2}{A}\gvac{1}\gob{1}{X}
\gend}
\equal{\equuref{pdf}{a}}
{\footnotesize
\gbeg{7}{11}
\gvac{2}\got{1}{A}\gvac{2}\got{1}{Y}\got{1}{A}\gnl
\gvac{2}\gcl{2}\gvac{2}\gcl{5}\gcl{6}\gnl
\gsbox{2}\gnot{\hspace{5mm}t}\gnl
\gcl{1}\gbrc\gnl
\gcl{1}\gcl{1}\gcn{1}{1}{1}{3}\gnl
\gmu\gsbox{3}\gnl
\gcn{1}{1}{2}{3}\gvac{1}\gcl{1}\gcl{1}\gev\gnl
\gvac{1}\gmu\gcn{1}{1}{1}{3}\gvac{1}\gcn{1}{1}{3}{1}\gnl
\gvac{1}\gcn{1}{1}{2}{5}\gvac{2}\gbrc\gnl
\gvac{3}\gmu\gcl{1}\gnl
\gvac{3}\gob{2}{A}\gob{1}{X}
\gend}
\]
we obtain that $\Phi$ is left $A$-linear if and only if $t$ is a morphism in ${\cal T}_A^\#$.\\
\un{$(v)\Leftrightarrow (viii)$.} If $\Psi:\ AX\to YA$ is left $A$-linear, then
\[{\footnotesize
\gbeg{3}{6}
\got{1}{A}\got{1}{A}\got{1}{X}\gnl
\gmu\gcl{2}\gnl
\gcn{1}{1}{2}{3}\gnl
\gvac{1}\gsbox{2}\gnot{\hspace{5mm}\Psi}\gnl
\gvac{1}\gcl{1}\gcl{1}\gnl
\gvac{1}\gob{1}{Y}\gob{1}{A}
\gend
}=
{\footnotesize
\gbeg{5}{8}
\gvac{2}\got{1}{A}\got{1}{A}\got{1}{X}\gnl
\gvac{2}\gcl{2}\gcl{1}\gcl{1}\gnl
\gdb\gvac{1}\gsbox{2}\gnot{\hspace{5mm}\Psi}\gnl
\gcl{1}\gbrc\gcl{1}\gcl{2}\gnl
\gcl{1}\gcl{1}\gev\gnl
\gcl{1}\gcn{1}{1}{1}{3}\gvac{1}\gcn{1}{1}{3}{1}\gnl
\gcl{1}\gvac{1}\gmu\gnl
\gob{1}{Y}\gvac{1}\gob{2}{A}
\gend
}~,~{\rm so}~
\Psi={\footnotesize
\gbeg{5}{8}
\gvac{2}\got{1}{A}\gvac{1}\got{1}{X}\gnl
\gvac{2}\gcl{1}\gu{1}\gcl{1}\gnl
\gdb\gvac{0}\gcl{1}\gsbox{2}\gnot{\hspace{5mm}\Psi}\gnl
\gcl{1}\gbrc\gcl{1}\gcl{1}\gnl
\gcl{1}\gcl{1}\gev\gvac{0}\gcl{1}\gnl
\gcl{1}\gcn{1}{1}{1}{3}\gvac{1}\gcn{1}{1}{3}{1}\gnl
\gcl{1}\gvac{1}\gmu\gnl
\gob{1}{Y}\gvac{1}\gob{2}{A}
\gend}=
{\footnotesize
\gbeg{4}{8}
\gvac{2}\got{1}{A}\got{1}{X}\gnl
\gvac{2}\gcl{1}\gcl{1}\gnl
\gdb\gvac{0}\gcl{1}\gcl{1}\gnl
\gcl{1}\gbrc\gcl{1}\gnl
\gcl{1}\gcl{1}\gsbox{2}\gnot{\hspace{5mm}B}\gnl
\gcl{1}\gcl{1}\gcn{1}{1}{2}{1}\gnl
\gcl{1}\gmu\gnl
\gob{1}{Y}\gob{2}{A}
\gend
}~,~{\rm where}~
B={\footnotesize
\gbeg{3}{5}
\got{1}{X}\gvac{1}\got{1}{X}\gnl
\gcl{1}\gu{1}\gcl{1}\gnl
\gcl{1}\gsbox{2}\gnot{\hspace{5mm}\Psi}\gnl
\gev\gvac{0}\gcl{1}\gnl
\gvac{2}\gob{1}{A}
\gend
}~.
\]
This shows that any left $A$-linear morphism $\Psi:\ AX\to YA$ is of the form \equref{Frobcar2}, for some
$B:\ X\ot X\ra A$ in $\Cc$. A long but straightforward computation shows that $\Psi$ is right $YA$-linear
if and only if  $B$ is a morphism in 
${\cal T}_A^\#$ satisfying (\ref{eq:coFrobCoTsh}.c).\\
\un{$(i)\Leftrightarrow (ix)$} follows from \thref{FrobGenEntwMod}.
\end{proof}

\section{Separability properties for entwined modules}\selabel{6}
\setcounter{equation}{0}
The aim of this Section is to study the separability of the forgetful functor $F:\ \Cc(\psi)_A^X\ra \Cc_A$ and its
right adjoint $G$. Separable functors were introduced in \cite{nbo}. Consider a pair of adjoint functors $F\dashv G$
between two categories $\Dc$ and ${\cal E}$, with unit $\eta: \Id_\Dc\ra GF$ and 
counit $\varepsilon: FG\ra \Id_{\cal E}$. The following result is due to Rafael \cite{rafael}.
\begin{itemize}
\item $F$ is separable if and only if the unit $\eta$ of the adjunction splits: 
there is a natural transformation $\vartheta:\ GF\rightarrow \Id_\Dc$ such that $\vartheta\circ \eta= \Id_\Dc$;
\item $G$ is separable if and only if the counit $\varepsilon$ of the adjunction cosplits: there is a natural transformation 
$\theta:\ \Id_{\cal E}\rightarrow FG$
such that $\varepsilon\circ \theta=\Id_{\cal E}$.
\end{itemize}
We will apply Rafael's Theorem to $F:\ \Cc(\psi)_A^X\ra \Cc_A$ and its right adjoint $G$. The natural transformations
$\vartheta$ and $\theta$ will be obtained as an application of Propositions \ref{pr:FrobElem} and \ref{pr:CasMor}.

\begin{proposition}\prlabel{prSepCarFunct}
Assume that $\un{1}$ is a left $\ot$-generator of the (strict) monoidal category $\Cc$, and let $(A,X,\psi)$ be a cowreath.
\begin{enumerate}
\item The forgetful functor $F:\ \Cc(\psi)_A^X\ra \Cc_A$ is separable if and only if there exists a Casimir morphism
$B:\ XX\to A$ for the coalgebra $(X,\psi)$ in $ \Tc_A^{\#}$ such that
$m\circ AB\circ \delta=\epsilon$.
\item $G:\ \Cc_A\to \Cc(\psi)_A^X$ is separable if and only if there exists a morphism $t:\ \un{1}\to X$ in ${\cal T}_A^\#$
such that $m\circ A\epsilon \circ t=\eta$.
\end{enumerate}
\end{proposition}

\begin{proof}
By Rafael's Theorem, $F$ is separable if and only if there exists $\vartheta:\ GF\rightarrow \Id_{\Cc(\psi)_A^X}$
such that $\vartheta\circ \eta$ is the identity natural transformation. Let $B$ be the Casimir morphism corresponding
to $\vartheta$, see \prref{CasMor}. Then we can easily show that
$\vartheta_M\circ \eta_M=\mu\circ MB\circ \rho X\circ \rho=\mu\circ Mh\circ \rho$,
where $h=m\circ AB\circ \delta$.\\
If $h=\epsilon$, then it follows that $\vartheta_M\circ \eta_M=\mu\circ M\epsilon \circ \rho\equal{\equref{c2}}M$.\\
Conversely, if $\theta\circ \eta$ is the identity natural transformation, then 
$$AX=\theta_{AX}\circ \eta_{AX}=\mu_{AX}\circ AXh\circ \rho_{AX}=mX\circ A\psi\circ AXh\circ mXX\circ A\delta,$$
and
\begin{eqnarray*}
\epsilon&\equal{\equref{epsilon}}&m\circ A\epsilon\circ \eta X=m\circ A\epsilon\circ AX \circ \eta X
= m^2\circ AA\epsilon\circ A\psi\circ AXh\circ mXX\circ A\delta\circ \eta X\\
&\equal{(\ref{eq:pdf}.c)}&
m^2\circ A\epsilon A\circ AXh\circ mXX\circ A\delta\circ \eta X
= m\circ Ah\circ mX\circ AmX\circ AA\epsilon X\circ A\delta\circ \eta X\\
&\equal{(\ref{eq:pdf}.d)}&
m\circ Ah\circ mX\circ A\eta X\circ \eta X=m\circ Ah\circ \eta X=m\circ \eta A\circ h=h.
\end{eqnarray*}
The proof of the second statement is similar. $G$ is separable if and only if there exists 
$\theta:\ \Id_{\Cc_A}\rightarrow FG$
such that $\varepsilon\circ \theta$ is the identity natural transformation, that is, $\varepsilon_N\circ \theta_N=N$, for all $N\in \Cc_A$.
Fix $\theta$, and let $t\in \Tc^{\#}_A(\un{1}, X)$ be the Frobenius element corresponding to $\theta$, see \prref{FrobElem}.
Then $\varepsilon_N\circ \theta_N=\mu\circ N\epsilon\circ \mu X\circ Nt$, fitting into the commutative diagram
$$\xymatrix{
N\ar[rr]^{Nt}&&
NAX\ar[rr]^{\mu X}\ar[d]_{NA\epsilon}&&
NX\ar[d]^{N\epsilon}\\
&&NAA\ar[d]_{Nm}\ar[rr]^{\mu A}&&NA\ar[d]^{\mu}\\
&&NA\ar[rr]^{\mu}&&N}$$
If $m\circ A\epsilon \circ t=\eta$, then $\varepsilon_N\circ \theta_N=\mu\circ N\eta=N$.\\
Conversely, if $\varepsilon\circ \theta=\Id_{\Cc_A}$, then $\varepsilon_A\circ \theta_A=A$, and we find from the
commutativity of the diagram that $\eta= m\circ Am\circ AA\epsilon\circ At\circ \eta=m\circ mA\circ \eta AA\circ A\epsilon\circ t= m\circ A\epsilon \circ t$.
\end{proof}

Coseparable coalgebras were introduced by Larson in \cite{larson}. This notion can be generalized to
coalgebras in (strict) monoidal categories. Remark that a coalgebra $C$ is a $C$-bicomodule, with
left and right $C$-coaction induced by comultiplication.

\begin{definition}\delabel{7.3}
A coalgebra $C$ is coseparable if it is a relative injective $C$-bicomodule in $\Cc$, which comes down
to the following property. If $i:\ M\to N$ in ${}^C\Cc^C$ has a left inverse $p:\ M\to N$ in $\Cc$, then 
every $f: M\ra C$ in ${}^C\Cc^C$ factors through $i$ in ${}^C\Cc^C$: there exists a 
$C$-bicolinear morphism $g: N\ra C$ such that $g\circ i=f$.   
\end{definition}

\begin{proposition}\prlabel{cosepcoalgmoncat}
For a coalgebra $C$ in a (strict) monoidal category $\Cc$, the following assertions are equivalent.
\begin{enumerate}
\item[(i)] $C$ is coseparable;
\item[(ii)] the comultiplication $\Delta$ has a $C$-bicolinear left inverse $\gamma:\ CC\to C$;
\item[(iii)] there exists a morphism $B:\ CC\to \un{1}$ in $\Cc$ such that
\begin{equation}\eqlabel{cosepeq2}
B\circ \Delta=\varepsilon~~{\rm and}~~CB\circ \Delta C=BC\circ C\Delta.
\end{equation}
\end{enumerate}
\end{proposition}

\begin{proof} We just sketch it since is similar to the one of \cite[Lemma 1]{larson}.

\un{$(i)\Rightarrow (ii)$}. $C\varepsilon:\ CC\to C$ is a left inverse of $\Delta$, so the identity $C\to C$ factors through
$\Delta$ in ${}^C\Cc^C$, which means that $\Delta$ has a $C$-bicolinear left inverse.\\
\un{$(ii)\Rightarrow (iii)$}. Let $B=\varepsilon\circ \gamma:\ CC\to \un{1}$. $B\circ \Delta = \varepsilon$ follows immediately
from $\gamma\circ \Delta=C$. The left $C$-colinearity of $\gamma$ means that $C\gamma\circ \Delta C=\Delta\circ \gamma$,
and this implies that $CB\circ \Delta C=C\varepsilon \circ C\gamma\circ \Delta C=C\varepsilon \circ \Delta\circ \gamma=\gamma$.
The right $C$-colinearity entails that $BC\circ C\delta =\gamma$, and the second formula in \equref{cosepeq2} follows.\\
\un{$(iii)\Rightarrow (ii)$}. $\gamma= CB\circ \Delta C=BC\circ C\Delta$ is a $C$-bicolinear left inverse of $\Delta$.\\
\un{$(ii)\Rightarrow (i)$}. Let $i,~p,~f$ be as in \deref{7.3}. Then $g=\gamma\circ C(\varepsilon\circ f\circ p)C\circ
CN\rho\circ \lambda$ is $C$-bicolinear and $g\circ i=f$. $\lambda$ and $\rho$ are the left and right $C$-coaction on $N$.
\end{proof}

A morphism $B:\ CC\to \un{1}$ satisfying the second condition in \equref{cosepeq2} is a Casimir morphism for $C$,
see \deref{defcoFrobCoalgMon}. A Casimir morphism is called normalized if it also satisfies the first condition in
\equref{cosepeq2}. A coseparable coalgebra is coalgebra together with a normalized Casimir morphism.

\begin{proposition}\prlabel{cosepcoalgTs}
For a cowreath $(A,X, \psi)$ in $\Cc$, the following assertions are equivalent.
\begin{enumerate}
\item[(i)] $(X, \psi)$ is a coseparable coalgebra in ${\cal T}_A^\#$;
\item[(ii)] there exists a morphism $\gamma:\ XX\to AX$ in $\Cc$ such that
\begin{equation}\eqlabel{cosepcoalgTsh}
(a)~{\footnotesize
\gbeg{3}{6}
\got{1}{X}\got{1}{X}\got{1}{A}\gnl
\gcl{1}\gbrc\gnl
\gbrc\gcl{1}\gnl
\gcl{1}\gsbox{2}\gnot{\hspace{5mm}\gamma}\gnl
\gmu\gcl{1}\gnl
\gob{2}{A}\gob{1}{X}
\gend
}={\footnotesize
\gbeg{3}{6}
\got{1}{X}\got{1}{X}\got{1}{A}\gnl
\gcl{1}\gcl{1}\gcl{2}\gnl
\gsbox{2}\gnot{\hspace{5mm}\gamma}\gnl
\gcl{1}\gbrc\gnl
\gmu\gcl{1}\gnl
\gob{2}{A}\gob{1}{X}
\gend
}~,~(b)~
{\footnotesize
\gbeg{4}{7}
\got{1}{X}\gvac{1}\got{1}{X}\gnl
\gcl{1}\gvac{1}\gcl{1}\gnl
\gcl{1}\gsbox{3}\gnl
\gbrc\gcl{1}\gcl{3}\gnl
\gcl{1}\gsbox{2}\gnot{\hspace{5mm}\gamma}\gnl
\gmu\gcl{1}\gnl
\gob{2}{A}\gob{1}{X}\gob{1}{X}
\gend}=
{\footnotesize
\gbeg{4}{7}
\got{1}{X}\got{1}{X}\gnl
\gcl{1}\gcl{1}\gnl
\gsbox{2}\gnot{\hspace{5mm}\gamma}\gnl
\gcl{1}\gcn{1}{1}{1}{3}\gnl
\gcl{1}\gsbox{3}\gnl
\gmu\gcl{1}\gcl{1}\gnl
\gob{2}{A}\gob{1}{X}\gob{1}{X}
\gend}=
{\footnotesize
\gbeg{4}{8}
\gvac{1}\got{1}{X}\gvac{1}\got{1}{X}\gnl
\gvac{1}\gcl{1}\gvac{1}\gcl{3}\gnl
\gsbox{3}\gnl
\gcl{1}\gcl{1}\gcl{1}\gnl
\gcl{1}\gcl{1}\gsbox{2}\gnot{\hspace{5mm}\gamma}\gnl
\gcl{1}\gbrc\gcl{1}\gnl
\gmu\gcl{1}\gcl{1}\gnl
\gob{2}{A}\gob{1}{X}\gob{1}{X}
\gend}~,~(c)~
{\footnotesize
\gbeg{3}{7}
\gvac{1}\got{1}{X}\gnl
\gvac{1}\gcl{1}\gnl
\gsbox{3}\gnl
\gcl{1}\gcl{1}\gcl{1}\gnl
\gcl{1}\gsbox{2}\gnot{\hspace{5mm}\gamma}\gnl
\gmu\gcl{1}\gnl
\gob{2}{A}\gob{1}{X}
\gend}=
{\footnotesize
\gbeg{2}{3}
\gvac{1}\got{1}{X}\gnl
\gu{1}\gcl{1}\gnl
\gob{1}{A}\gob{1}{X}
\gend
}~.
\end{equation}
\item[(iii)] there exists a Casimir morphism $B$ for the coalgebra $(X, \psi)$ in ${\cal T}_A^\#$ such that
$m\circ AB\circ \delta=\epsilon$.
\end{enumerate}
If $X\dashv Y$ in $\Cc$, then these conditions are equivalent to
\begin{enumerate}
\item[(iv)] there exists a left $A$-linear $\Psi:\ AX\to YA$ in $\Cc(\psi)_A^X$ such that
\begin{equation}\eqlabel{normcondPsi}
m\circ AdA \circ AX\Psi\circ AX\eta X\circ \delta=\epsilon;
\end{equation}
\item[(v)] there exists a morphism $\ov{\Psi}: X\ra YA$ in $\Cc$ satisfying the equations:
\begin{equation}\label{sepcondofovPsi}
(a)~
{\footnotesize
\gbeg{4}{8}
\got{1}{X}\got{1}{X}\got{1}{A}\gnl
\gcl{1}\gbrc\gnl
\gbrc\gcn{1}{1}{1}{2}\gnl
\gcl{1}\gcl{1}\gsbox{2}\gnot{\hspace{5mm}\ov{\Psi}}\gnl
\gcl{1}\gev\gvac{0}\gcl{1}\gnl
\gcn{1}{1}{1}{3}\gvac{2}\gcn{1}{1}{1}{-1}\gnl
\gvac{1}\gmu\gnl
\gvac{1}\gob{2}{A}
\gend=
{\footnotesize
\gbeg{4}{5}
\got{1}{X}\got{2}{X}\got{1}{A}\gnl
\gcl{1}\gcn{1}{1}{2}{2}\gvac{1}\gcl{1}\gnl
\gcl{1}\gsbox{2}\gnot{\hspace{5mm}\ov{\Psi}}\gvac{2}\gcl{1}\gnl
\gev\gvac{0}\gmu\gnl
\gvac{2}\gob{2}{A}
\gend
}
}~,~(b)~
{\footnotesize
\gbeg{6}{9}
\gvac{1}\got{1}{X}\gvac{2}\got{2}{X}\gnl
\gvac{1}\gcl{1}\gvac{2}\gcn{1}{1}{2}{2}\gnl
\gsbox{3}\gvac{4}\gsbox{2}\gnot{\hspace{5mm}\ov{\Psi}}\gnl
\gcl{1}\gcl{1}\gcl{1}\gvac{1}\gcn{1}{1}{1}{-1}\gcl{1}\gnl
\gcl{1}\gcl{1}\gev\gvac{1}\gcl{1}\gnl
\gcl{1}\gcl{1}\gvac{3}\gcn{1}{1}{1}{-5}\gnl
\gcl{1}\gbrc\gnl
\gmu\gcl{1}\gnl
\gob{2}{A}\gob{1}{X}
\gend
}=
{\footnotesize
\gbeg{5}{9}
\got{1}{X}\gvac{1}\got{1}{X}\gnl
\gcl{1}\gvac{1}\gcl{1}\gnl
\gcl{1}\gsbox{3}\gnl
\gbrc\gcn{1}{1}{1}{2}\gcn{1}{1}{1}{3}\gnl
\gcl{1}\gcl{1}\gsbox{2}\gnot{\hspace{5mm}\ov{\Psi}}\gvac{2}\gcl{1}\gnl
\gcl{1}\gev\gvac{0}\gcl{1}\gcl{3}\gnl
\gcn{1}{1}{1}{3}\gvac{2}\gcn{1}{1}{1}{-1}\gnl
\gvac{1}\gmu\gnl
\gvac{1}\gob{2}{A}\gvac{1}\gob{1}{X}
\gend
}~,~(c)~
{\footnotesize
\gbeg{4}{9}
\gvac{1}\got{1}{X}\gnl
\gvac{1}\gcl{1}\gnl
\gsbox{3}\gnl
\gcl{1}\gcl{1}\gcn{1}{1}{1}{2}\gnl
\gcl{1}\gcl{1}\gsbox{2}\gnot{\hspace{5mm}\ov{\Psi}}\gnl
\gcl{1}\gev\gvac{0}\gcl{1}\gnl
\gcn{1}{1}{1}{3}\gvac{2}\gcn{1}{1}{1}{-1}\gnl
\gvac{1}\gmu\gnl
\gvac{1}\gob{2}{A}
\gend}
=\epsilon~.
\end{equation}
\end{enumerate}
\end{proposition}

\begin{proof}
\un{$(i)\Leftrightarrow (ii)\Leftrightarrow (iii)$}. 
\equuref{cosepcoalgTsh}{a} says that $\gamma$ is a morphism in ${\cal T}_A^\#$, 
\equuref{cosepcoalgTsh}{b} that $\gamma$ is an $(X, \psi)$-bicolinear morphism in 
${\cal T}_A^\#$ and \equuref{cosepcoalgTsh}{c} that $\gamma$ is a left inverse of
 the comultiplication $\d$ of the coalgebra $(X, \psi)$ in ${\cal T}_A^\#$, so condition
 (ii) is condition (ii) from \prref{cosepcoalgmoncat} in the special case where
 $\Cc= {\cal T}_A^\#$. A similar observation holds for condition (iii), and the equivalence
 of (i), (ii) and (iii) follows.\\
 \un{$(iii)\Leftrightarrow (iv)$}. 
We have seen in the proof of the equivalence {$(v)\Leftrightarrow (viii)$} in 
\thref{Frobcaractfinitecase} that there is a bijective correspondence between
left $A$-linear morphisms $\Psi: A X\ra YA$ in 
$\Cc(\psi)_A^X$ and Casimir morphisms 
$B$ for the coalgebra $(X, \psi)$ in ${\cal T}_A^\#$. Moreover, it is easy to show that $\Psi$ satisfies 
\equref{normcondPsi} if and only if the corresponding $B$ has the property that
$m\circ AB\circ \delta=\epsilon$. \\
\un{$(iv)\Leftrightarrow (v)$}. It follows from \leref{bim1a} that we have an isomorphism
$\alpha:\ {}_A\Cc(AX,AY)\to \Cc(X,AY)$, given by $\alpha(\Psi)=\Psi\circ\eta A$ and
$\alpha^{-1}(\ov{\Psi})=\nu_{YA}\circ A\ov{\Psi}=Ym\circ YAdA\circ Y\psi AYA\circ bAYA\circ A\ov{\Psi}$.
It is left to the reader to check that
 $\Psi$ is a left $A$-linear morphism in $\Cc(\psi)_A^X$ if and only if $\alpha(\Psi)=\ov{\Psi}$ satisfies 
 (\ref{sepcondofovPsi}.a, b). Finally, \equref{normcondPsi} is equivalent to
(\ref{sepcondofovPsi}.c).
\end{proof}

Our next result is a generalization of \cite[Theorem 2.3]{cimz}.

\begin{theorem}\thlabel{ForFunsepvsSepcoal}
Assume that $\un{1}$ is a left $\ot$-generator for $\Cc$. For 
a cowreath $(A,X,\psi)$ in $\Cc$, the following statements are equivalent:
\begin{itemize}
\item[(i)] The forgetful functor $F:\ \Cc(\psi)_A^X\ra \Cc_A$ is separable;
\item[(ii)] $(X, \psi)$ is a coseparable coalgebra in ${\cal T}_A^\#$.
\end{itemize}
If $\Cc$ has coequalizers and $A$ and $X$ are left coflat in $\Cc$, these
statements are also equivalent to 
\begin{itemize}
\item[(iii)] $AX$ is a coseparable $A$-coring in $\Cc$, that is a coseparable coalgebra 
in the monoidal category $\ACA$;
\item[(iv)] the forgetful functor $U: \Cc^{AX}\ra \Cc_A$ is separable. 
\end{itemize}
\end{theorem}

\begin{proof}
\un{$(i)\Leftrightarrow (ii)$.} Follows from \prref{prSepCarFunct} and the equivalence $(i)\Leftrightarrow (iii)$ in 
\prref{cosepcoalgTs}. \\
\un{$(ii)\Leftrightarrow (iii)$.} We proceed as in the proof of \thref{coFrobTvscoring}. As before we identify $AX\bullet AX=AXX$. Applying
\leref{bim1a}, we obtain an isomorphism $\alpha:\ {}_A\Cc(AXX,AX)\to \Cc(XX,AX)$. A direct verification shows
that $\Omega\in {}_A\Cc(AXX,AX)$ is right $A$-linear if and only if $\alpha(\Omega)=\gamma$ satisfies \equuref{cosepcoalgTsh}{a}.
In this situation, $\Omega$ is left and right $AX$-colinear if and only if $\gamma$ satisfies \equuref{cosepcoalgTsh}{b}.
Finally $\Delta\circ \Omega= AX$ if and only if \equuref{cosepcoalgTsh}{c} holds.\\
\un{$(i)\Leftrightarrow (iv)$.} The categories $\Cc^{AX}$ and $\Cc(\psi)_A^X$ are isomorphic, 
see \cite[Theorem 4.8]{bc4}, and this implies immediately that the separability of $F$ and $U$ is equivalent.
\end{proof}

More equivalent conditions for the coseparability of a coalgebra $(X,\psi)$  in ${\cal T}_A^\#$ can be given under
the assumption that $X\dashv Y$. For the definition of a separable algebra extension in a monoidal category,
we refer to \cite[Def. 4.5 (ii)]{dbbt2}.

\begin{proposition}\prlabel{SepFunctfinitcase}
Let $\Cc$ be a monoidal category with coequalizers, and assume that every object of $\Cc$ is flat.
Let $(A,X,\psi)$ be a cowreath in $\Cc$. If $X\dashv Y$ in $\Cc$ and every left $A$-module is robust, then the 
following assertions are equivalent:
\begin{itemize}
\item[(i)] $(X, \psi)$ is a coseparable coalgebra in ${\cal T}_A^\#$;
\item[(ii)] $(Y,\varphi)$ is a separable algebra in ${}_A^\#{\cal T}$, where 
$\ov{\psi}$ is defined in \equref{psidual};
\item[(iii)] The smash product $ YA$ is a separable 
algebra extension of $A$ in $\Cc$;
\item[(iv)] $YA$ is a separable $A$-ring, that is a separable
algebra in ${}_A\Cc_A$.  
\end{itemize}
If $\un{1}$ is a left $\ot$-generator in $\Cc$ then (i)-(iv) are also equivalent to 
\begin{itemize}
\item[(v)] The restriction of scalars functor $F': \Cc_{YA}\to \Cc_A$ 
is separable.
\end{itemize}
\end{proposition}

\begin{proof}
This is an immediate consequence of Theorems \ref{th:gntasmod} and \ref{th:ForFunsepvsSepcoal}, and
\cite[Cor. 8.9]{dbbt2} applied to the wreath $(Y, A, \varphi)$.
\end{proof}

Our final result is a Maschke type Theorem for entwined modules. It generalizes \cite[Theorem 2.7]{cimz} and \cite[Theorem 4.2]{brFM}.

\begin{theorem}\thlabel{MaschkeGenEntwinMod}
Let $(X, \psi)$ be a coseparable 
coalgebra in ${\cal T}_A^\#$.
\begin{itemize}
\item[(i)] If a morphism in $\Cc(\psi)_A^X$ has a section (resp. a retraction) in $\Cc_A$ then it has 
a section (resp. a retraction) in $\Cc(\psi)_A^X$;
\item[(ii)] If an object in $\Cc(\psi)_A^X$ is semisimple (resp. projective, injective) 
as a right $A$-module then it is semisimple (resp. projective, injective) as an entwined module 
over $(A, X, \psi)$. 
\item[(iii)] Every $M\in \Cc(\psi)_A^X$ is relative injective (see \deref{7.3} for the definition of relative injectivity). 
\end{itemize}
\end{theorem}

\begin{proof}
The forgetful functor $F:\  \Cc(\psi)_A^X\ra \Cc_A$ is separable since $(X, \psi)$ 
is a coseparable coalgebra in ${\cal T}_A^\#$. The three assertions then follow immediately from
\cite[Prop. 47 and 48, Cor.7]{cmz}. 
\end{proof}


\end{document}